\documentclass[11pt,leqno]{article}
\usepackage{hyperref,amssymb,amsmath,amsthm,mathtools}
\usepackage[english]{babel}
\usepackage[utf8]{inputenc}
\usepackage[abbrev,nobysame,lite]{amsrefs}
\AtBeginDocument{\def\MR#1{}} 

\usepackage{mathrsfs}
\usepackage[right = 2.5cm, left=2.5cm, top = 2.5cm, bottom =2.5cm]{geometry}
\usepackage[normalem]{ulem}
\renewcommand{\epsilon}{\varepsilon}

\makeatletter
\def\@seccntformat#1{\csname the#1\endcsname.\quad}
\makeatother

\usepackage{xcolor,tikz}
\usetikzlibrary{patterns}
\usepackage[outline]{contour}
\contourlength{1.7pt}

\makeatletter
\def\blfootnote{\xdef\@thefnmark{}\@footnotetext}
\makeatother

\usepackage{soul}

\newcommand{\eq}[2]{\begin{equation}\label{#1}#2\end{equation}}
\newcommand{\pde}[2]{\ensuremath{\left\{\begin{array}{l}#1\vspace{0.3cm}\\#2\end{array}\right.}}
\newcommand{\pdet}[3]{\ensuremath{\left\{\begin{array}{l}#1\vspace{0.3cm}\\#2\vspace{0.3cm}\\#3\end{array}\right.}}
\newcommand{\pdef}[4]{\ensuremath{\left\{\begin{array}{l}#1\vspace{0.3cm}\\#2\vspace{0.3cm}\\#3\vspace{0.3cm}\\#4\end{array}\right.}}

\newcommand{\R}{\ensuremath{\mathbb{R}}}
\newcommand{\C}{\ensuremath{\mathbb{C}}}
\newcommand{\Z}{\ensuremath{\mathbb{Z}}}
\newcommand{\N}{\ensuremath{\mathbb{N}}}
\newcommand{\T}{\ensuremath{\mathbb{T}}}
\newcommand{\cE}{\ensuremath{\mathcal{E}}}
\newcommand{\Schwartz}{\ensuremath{\mathscr{S}}}

\newcommand{\abs}[1]{\left\vert#1\right\vert}
\newcommand{\norm}[1]{\left\lVert#1\right\rVert}
\newcommand{\brak}[1]{\langle #1 \rangle} 
\newcommand{\Brak}[1]{\left\langle #1 \right\rangle} 
\newcommand{\set}[1]{\left\{ #1 \right\} }
\newcommand{\eps}{\varepsilon}
\newcommand{\grad}{\nabla}

\newtheorem{lem}{Lemma}[section]
\newtheorem{lemma}[lem]{Lemma}
\newtheorem{thrm}[lem]{Theorem}
\newtheorem{prop}[lem]{Proposition}
\newtheorem{cor}[lem]{Corollary}
\theoremstyle{definition}
\newtheorem{defn}[lem]{Definition}

\newtheorem{rem}[lem]{Remark}

\newcommand{\bpf}{\begin{proof}}
\newcommand{\epf}{\end{proof}}

\renewcommand{\Re}{\operatorname{Re}}
\renewcommand{\Im}{\operatorname{Im}}
\newcommand{\mc}{\mathcal}
\newcommand{\mb}{\mathbf}
\newcommand{\mr}{\mathrm}

\DeclareMathOperator{\supp}{supp}

\newcommand{\<}{\langle}
\renewcommand{\>}{\rangle}
\newcommand{\cL}{\mc{L}}

\newcommand{\cD}{\mc{D}}
\DeclareMathOperator{\Sp}{Span}
\newcommand{\cM}{\mc{M}}
\newcommand{\fM}{\mc{M}}
\DeclareMathOperator{\Div}{div}

\newcommand{\cT}{\mc{T}}
\newcommand{\BMO}{\mr{BMO}}

\newcommand{\ft}{\mathfrak t}
\newcommand{\fn}{\mathfrak n}
\newcommand{\fb}{\mathfrak b}
\DeclareMathOperator{\dist}{dist}
\DeclareMathOperator{\tr}{tr}

\newcommand{\fD}{\mathfrak D}
\DeclareMathOperator{\Curl}{curl}
\newcommand{\bD}{\mathbf D}

\newcommand{\fJ}{J}
\newcommand{\bd}{\mb d}
\newcommand{\bN}{\mb N}
\newcommand{\bF}{\mb F}
\newcommand{\bS}{\mathbb S}
\newcommand{\fE}{\mathfrak{E}}
\newcommand{\fF}{\mathfrak{F}}
\newcommand{\bK}{\mathbb{K}}
\newcommand{\onabla}{\overline{\nabla}}
\newcommand{\oDiv}{\overline{\Div}}
\newcommand{\RHS}{\mr{RHS}}
\newcommand{\oDelta}{\overline\Delta}

\mathtoolsset{showonlyrefs=true}

\title{Vortex filament solutions of the Navier-Stokes equations}
\author{Jacob Bedrossian\thanks{\footnotesize University of Maryland, College Park. \href{mailto:jacob@cscamm.umd.edu}{\texttt{jacob@cscamm.umd.edu} }. The  author  was  partially  supported  by  NSF
CAREER grant DMS-1552826 and NSF RNMS 1107444 (Ki-Net).} \and Pierre Germain\thanks{\footnotesize Courant Institute of Mathematical Sciences, New York University. \href{mailto:pgermain@cims.nyu.edu}{\texttt{pgermain@cims.nyu.edu}}. The  author  was  partially  supported  by  NSF grant DMS-1501019.} \and Benjamin Harrop-Griffiths\thanks{\footnotesize  Department of Mathematics, UCLA. \href{mailto:harropgriffiths@math.ucla.edu}{\texttt{harropgriffiths@math.ucla.edu}}. This work was completed while the author was supported by a Junior Fellow award from the Simons Foundation.}}
\date{}

\numberwithin{equation}{section}

\begin{document}

\maketitle
\begin{abstract}
We consider solutions of the Navier-Stokes equations in \(3d\) with vortex filament initial data of arbitrary circulation, that is, initial vorticity given by a divergence-free vector-valued measure of arbitrary mass supported on a smooth curve. First, we prove global well-posedness for perturbations of the Oseen vortex column in scaling-critical spaces. Second, we prove local well-posedness (in a sense to be made precise) when the filament is a smooth, closed, non-self-intersecting curve. Besides their physical interest, these results are the first to give well-posedness in a neighborhood of large self-similar solutions of $3d$ Navier-Stokes, as well as solutions which are locally approximately self-similar.
\end{abstract}
 
\setcounter{tocdepth}{1}
{\small\tableofcontents}


\section{Introduction}

\subsection{Vortex filaments}

The incompressible \(3d\) Navier-Stokes equations in vorticity form on $\mathbb R^3$, with viscosity normalized to $1$, are
\eq{NSE}{\tag{NS}
\boxed{
\partial_t \omega + u\cdot\nabla\omega - \omega\cdot\nabla u = \Delta \omega, 
}
}
where the velocity \(u\) and vorticity \(\omega\) are related by the Biot-Savart law
\eq{BS}{\tag{BS}
\boxed{
u = (-\Delta)^{-1}\nabla\times \omega.
}
}
As usual, we also have the divergence-free requirement on the vorticity
\begin{align}
\boxed{
\grad \cdot \omega = 0,
}
\end{align}
for which it suffices to choose divergence-free initial data. 

In this article we consider solutions with vortex filament initial data, i.e. the initial vorticity
\eq{IntroData}{
\omega(t=0) = \alpha\delta_\Gamma,
}
where the circulation \(\alpha\in \R\), and, for a smooth oriented curve \(\Gamma\subset\R^3\), we define $\delta_\Gamma$ to be the vector-valued measure satisfying for any test function \(\varphi\in \mc C^\infty_c(\R^3)\)
\[
\<\delta_\Gamma,\varphi\> = \int_\Gamma \varphi \cdot d\vec s.
\]

Of particular interest throughout this article will be the case of large Reynolds numbers, which corresponds to the limit \(|\alpha|\rightarrow\infty\). For $\alpha$ large, this data falls outside the realm of previously existing local well-posedness theory of mild solutions and, as the velocity is not in $L^2_{\mr{loc}}$, one cannot construct Leray-Hopf weak solutions either. 

The term `vortex filament' refers to a configuration of intense vorticity approximately concentrated along a curve. In experiments, such structures tend to move in a coherent manner over relatively long time-scales (see e.g. the experiments on knotted vortices \cite{KlecknerIrvine13}) and are also thought to be potentially related to intermittent behavior in turbulent flows (see e.g. \cite{DouadyEtAl91,MoffattEtAl94,PullinSaffman98}). 
The mathematical study of vortex filaments dates back to the work of Helmholtz~\cite{MR1579057}, with other early studies by Kelvin~\cite{Kelvin1880,Thomson1868} and da Rios ~\cite{daRios1906}. The latter formally derived the first dimension reduced model, now called the local induction approximation, which sought to simplify the dynamics to the evolution of a curve rather than an entire vorticity field.  
This was later rediscovered in the 1960s with a renewal of interest in vortex filament motion, where more refined models were also considered (see e.g.~\cite{Ricca96,JerrardSeis17,MajdaBertozzi} and the references therein). The binormal flow, which is derived from the local induction approximation, is an interesting equation in its own right and has been the subject of much research (e.g.~\cite{MR3363420,MR1602745,MR1090719,MR3353807,MR3059243,MR3429472,MR3116002,MR3204898,MR2862038,MR2472037,2018arXiv180706948B}).

The above derivations of dimension reduced models are not mathematically rigorous and essentially neglect viscosity, instead modeling the filament as a smooth object of finite width in the Euler equations. However, passing the width to zero in the Euler equations is a very singular limit. Work has been done to rigorously justify the dimension reduction conditional on certain hypotheses about the solution of Euler (see in particular~\cite{JerrardSeis17} and the references therein) and in the case of axisymmetric vortex rings \cite{BCM00,2016arXiv160902030G}. To our knowledge, a complete description of vortex filaments either in the Euler or Navier-Stokes equations remains open. In the viscous case, it is natural to model filaments with the data \eqref{IntroData} (as suggested as early as \cite{MR993821} at least) and for this initial data, with $\alpha$ large,  most of the fundamental questions of existence, uniqueness, continuity and dynamics in \(3d\) remain open. In this article we develop a framework to study general (smooth, non-self-intersecting) vortex filaments in \(3d\) and use this to prove several existence and uniqueness results that hold for arbitrary circulation numbers \(\alpha\in \R\).

\subsection{Criticality} The standard approach to constructing well-posed solutions in low regularity spaces is that of the mild solution.  
That is, one formally writes the solution of \eqref{NSE} with initial data \(\omega(t = 0) = \omega_0\) using the Duhamel formula as
\begin{equation}
\omega(t) = e^{t\Delta}\omega_0 - \int_0^te^{(t-s)\Delta}\Div\Bigl(u(s)\otimes \omega(s) - \omega(s)\otimes u(s)\Bigr)\,ds, \label{milddefDuhamel}
\end{equation}
and may then attempt to use a contraction mapping argument in a suitable space. 
A natural question is: what is the largest space of functions in which the mild formulation of \eqref{NSE} is well-posed (in the sense of Hadamard)? 
Taking coordinates \(y\in \R^3\), we observe that the equation \eqref{NSE} is invariant under the scaling
\[
\omega(t,y)\mapsto \lambda^2 \omega(\lambda^2 t,\lambda y)\quad\text{for}\quad \lambda>0,
\]
where we note that the corresponding scaling of the velocity is
\[
u(t,y)\mapsto \lambda u(\lambda^2t,\lambda y).
\]
Heuristic considerations suggest that the largest possible spaces in which one can obtain mild solutions of \eqref{NSE} are \emph{critical} in the sense that the corresponding norm is invariant under this scaling. 

We refer to a space $X$ as \emph{ultra-critical} if $X$ is critical and the Schwartz functions are \emph{not} dense (we will see some examples below). A common feature of such ultra-critical spaces is that they contain initial data that are invariant under the scaling, so one can expect self-similar solutions to live in precisely these classes. Another (closely related) common feature is that, in general, one only obtains global existence for small data, and local well-posedness if the distance of the data to Schwartz functions is small. The circulation number \(\alpha\) of a vortex filament is invariant under this scaling, and so one may view the problem of local well-posedness for initial data of the form \eqref{IntroData} as a large data problem in (ultra-)critical spaces.

\subsection{2d local well-posedness}
Taking \(\R^3 = \R^2\times \R\) with coordinates \((x,z)\in \R^2\times \R\), an explicit and important example of a vortex filament is obtained when $\Gamma$ is the line $\{ x = 0 \}$. It is called the Oseen vortex, and is given by
\eq{omegag}{
\omega^g =  \begin{bmatrix}0\\\frac 1 t G(\frac x{\sqrt{ t}})\end{bmatrix},\quad\text{where}\quad G(\xi) = \frac1{4\pi}e^{-\frac14 |\xi|^2},
}
with corresponding velocity field
\eq{ug}{
u^g = \begin{bmatrix}\frac 1 {\sqrt{ t}}g(\frac x{\sqrt{ t}})\\0\end{bmatrix},\quad\text{where}\quad g(\xi) = \frac1{2\pi}\frac{\xi^\perp}{|\xi|^2}(1 - e^{-\frac14 |\xi|^2}),
}
and \(\xi^\perp = (-\xi_2,\xi_1)^T\). It is natural to expect that the Oseen vortex provides the microscopic structure for the evolution of \textit{any} smooth, non-self-intersecting vortex filament. This expectation is, in some sense, confirmed by our results. The Oseen vortex is also a two-dimensional solution, and a detailed understanding of its stability in $2d$ is a key element of our investigations. 

We see that $L^1$ is the critical Lebesgue space for the vorticity in $2d$, whereas $\fM$, the space of finite measures equipped with the total variation norm, is an ultra-critical space. Note that the Oseen vortex \eqref{omegag} is a mild solution with ultra-critical initial\ data.  
Uniqueness of the Oseen vortex with $\delta$ initial data was proved in~\cite{MR2176270,MR2123378}, and then uniqueness for arbitrary $\fM$ initial data in the work of Gallagher and Gallay~\cite{MR2178064} (see also~\cite{MR3269635,MR2968938}). See also earlier work of Giga, Miyakawa, and Osada \cite{GM88}, Gallagher and Planchon~\cite{GallagherPlanchon02}, and~\cite{MR2200642}.

Besides a reduction to $2d$, another possible symmetry reduction is axisymmetry. For the \(3d\) problem with large circulation numbers, one can consider the axisymmetric case without swirl and take vortex ring initial data of the form \eqref{IntroData} with \(\Gamma = \{|x|=R\}\) for some \(R>0\). Global existence of axisymmetric solutions of this type with arbitrary circulation number was proved by Feng and \v{S}ver\'ak~\cite{MR3296145} and uniqueness by Gallay and \v{S}ver\'ak~\cite{2016arXiv160902030G} (see also~\cite{2015arXiv151001036G}).

Finally, we note that a careful analysis of the linearized stability of the Oseen vortex (and more general column vortices), without any symmetry assumptions, has recently been carried out for the inviscid case (\(\nu = 0\)) by Gallay and Smets~\cite{MR4010658,gallay2018spectral}.

\subsection{$3d$ local well-posedness}
The first critical well-posedness results for \eqref{NSE} dealt with spaces excluding self-similar data. Fujita and Kato~\cite{MR0142928,MR0166499} proved well-posedness of \eqref{NSE} for the velocity in the space \(\dot H^{\frac12}\) (see also~\cite{MR1145160}). Twenty years later Kato~\cite{MR760047} proved a similar result for the velocity in the larger space \(L^3\) (see also~\cite{MR833416}). Both of these results prove local well-posedness for arbitrary initial data and global well-posedness for sufficiently small initial data. See ~\cite{MR2776367} for a separate line of research focused on identifying data that are large in critical norms, but give rise to global solutions.

The next step in the theory was to deal with ultra-critical spaces, hence allowing self-similar data. Well-posedness in critical Besov spaces was proved by Cannone~\cite{MR1617394,MR1688096}, Planchon~\cite{MR1395675}, Chemin~\cite{MR1753481} and Cannone and Planchon~\cite{MR1373769}. Giga and Miyakwa~\cite{MR993821} considered solutions with the vorticity in the critical Morrey space defined as the set of signed measures satisfying $\sup_{r>0,y\in \R^3} r^{-1} \abs{\mu (B(y,r)) } < \infty$ and they observed there that data of the type \eqref{IntroData} falls into precisely this class. Their results were subsequently improved by Taylor~\cite{MR1187618} (see also~\cite{MR1179482,MR1274547}). The largest space of initial data for which well-posedness for small data is known is the space \(\BMO^{-1}\) of Koch and Tataru~\cite{MR1808843} (see also~\cite{MR2352218}).

Finally, let us mention the work of Jia and \v{S}ver\'ak~\cite{MR3179576}, who proved the existence of smooth self-similar solutions for arbitrarily large initial velocities that are locally H\"older continuous away from zero. This provides some large data solutions in the ultra-critical space $L^{3,\infty}$ (weak $L^3$). Furthermore, in~\cite{MR3341963} they proved a conditional \emph{non-uniqueness} result for self-similar initial data in $L^{3,\infty}$ under suitable spectral assumptions on the corresponding linearized operator (see also~\cite{GS17}). Roughly speaking, they prove that if the linearization around the self-similar solution (in self-similar variables) has eigenvalues that move from stable to unstable, then one can perform a bifurcation and construct additional smooth solutions. 
Our work will show that a similar bifurcation \emph{cannot} happen for the solution \eqref{omegag}.

\medskip

Our results, which will be presented in the next subsection, are the first to give local well-posedness for the $3d$ Navier-Stokes equations in a class of solutions containing large self-similar solutions.
The well-posedness class is essentially the mild solutions which are sufficiently close to the self-similar Gaussian in a certain scaling-critical sense as $t \searrow 0$. Subcritical contributions are vanishingly small for short time, so no smallness requirement will be present. Moreover, it will turn out that the curvature of the filament is effectively subcritical. Note that, in particular, this indeed rules out other self-similar solutions in a certain neighborhood of \eqref{omegag}, but does not rule out the existence of other self-similar solutions with the same initial data that are sufficiently different from \eqref{omegag}.

\subsection{A sketch of obtained results}

Our first results deal with perturbations of the straight filament $\alpha\delta_{\{x = 0\}}$. We prove local well-posedness for arbitrary perturbations in a subcritical space. For small perturbations in a critical space, we are able to obtain global solutions, which relax to the Oseen vortex. A simplified statement is as follows:

\begin{thrm}[Simplified statement] There exists a scale invariant space $X$ and $\epsilon(\alpha) >0$ such that: if $\| \mu^b\|_X < \epsilon$ and $\grad \cdot \mu^b = 0$ in the sense of distributions, there exists a unique global solution $\omega$ to~\eqref{NSE} with data
$$
\omega(t=0) = \alpha \delta_{\{x=0\}} + \mu^b
$$
which can be decomposed into
$$
\omega(t,x,z) = \begin{bmatrix}0\\\frac \alpha{ t}G(\frac x{\sqrt{ t}})\end{bmatrix} + \tfrac{1}{t} \Omega^c\left(\log t,\tfrac{x}{\sqrt{t}},z\right) + \omega^b(t,x,z),
$$
where 
$$
\sup_{t>0} \left[ \| \langle \xi \rangle^2 \Omega^c \|_{L^\infty_z L^2_\xi} + t^{\frac 14} \| \omega^b \|_{L^\infty_z L^{4/3}_x} \right] \lesssim \| \mu^b \|_{X}.
$$
Furthermore, the map $\mu^b \mapsto (\Omega^c,\omega^b)$ is continuous.
\end{thrm}

Our second main result deals with perturbations of arbitrary vortex filaments. Consider a closed, non-self-intersecting curve $\Gamma$, and define $\Phi$ a smooth map from a tubular neighborhood of $\{(0,z):z\in \T\}\subset \R^2\times \T$ to a tubular neighborhood of $\Gamma$ (we refer to the next section for a more detailed description).

\begin{thrm}[Simplified statement]
For any $\mu^b \in W^{1, \frac{12}{11}}$ satisfying $\grad \cdot \mu^b = 0$ in the sense of distributions, there exists $T>0$ and a unique solution $\omega$ to~\eqref{NSE} on $[0,T]$ with data
$$
\omega(t=0) = \alpha \delta_\Gamma + \mu^b
$$
which, in a tubular neighborhood of $\Gamma$, can be decomposed into
$$
\left(\left(\det \grad \Phi \right)(\grad \Phi)^{-1}\right)\!(x,z)\; \omega(t,\Phi(x,z)) = 
\begin{bmatrix}0\\\frac \alpha{ t}G(\frac x{\sqrt{ t}})\end{bmatrix} + \tfrac{1}{t} \Omega^c\left(\log t,\tfrac{x}{\sqrt{t}},z\right) + \omega^b(t,x,z),
$$
where 
$$
\sup_{0<t\leq T} \left[ \| \langle \xi \rangle^2 \Omega^c \|_{L^\infty_z L^2_\xi} + t^{\frac14} \| \omega^b \|_{L^\infty_z L^{4/3}_x} \right] < \infty.
$$
Furthermore, the map $\mu^b \mapsto (\Omega^c,\omega^b)$ is continuous.
\end{thrm}

We note that in the case of the curved filament we do not expect the above decomposition of the vorticity to be valid on longer timescales. 
Indeed, at high Reynolds number, the filament will evolve in a fully nonlinear manner, e.g. under the local induction approximation and its refinements \cite{Ricca96,MajdaBertozzi}. 
A rigorous proof of these dynamics in general remains an important open problem and our results may be viewed as a first step towards a solution (see \cite{BCM00,JerrardSeis17} and the references therein for progress on the inviscid problem and~\cite{2016arXiv160902030G} for the case of vortex rings in Navier-Stokes). 

\subsection*{Notations and conventions} 
Throughout this article we will typically not distinguish the target space of various functions, using \(L^p_x\) to denote the usual Lebesgue space with measure \(dx\) for scalar fields, vector fields and tensor fields alike.

We follow the following conventions regarding vector calculus: 
\begin{itemize}
\item $(a_{ij})_{ij}$, $(a^i_j)_{ij}$, $(a^{ij})_{ij}$ all denote the matrix with line index $i$, column index $j$.
\item If $a$, $b$ are vectors, then $a \otimes b = (a^i b^j)_{ij}$.
\item If $u$ is a vector field, then $\nabla u = (\partial_j u^i)_{ij}$.
\item Given a (\(k\times k\)-)tensor field $f$ (i.e. a \(k\times k\)  matrix-valued function) we write $\nabla \cdot f = \operatorname{div} f = (\partial_i f^{ij})_{j}$, where we use the Einstein summation convention, i.e. \(\partial_if^{ij} = \sum_i \partial_if^{ij}\)
\item Given two vector fields $f$ and $g$ we define the bilinear operator \(B[f,g] = \Div(f \otimes g - g \otimes f)\) 
\item For both vectors and matrices, we denote $\abs{v}$ to be the usual norm induced by the Euclidean metric.
\end{itemize}

We denote $g \lesssim f$ if there exists a constant $C > 0$ such that $g \leq Cf$ and we use $g \lesssim_{\alpha,\beta,\dots} f$ to emphasize dependence of $C$ on parameters $\alpha,\beta,\dots$. 
We similarly write $g \approx f$ if we have both $g \lesssim f$ and $f \lesssim g$.

As usual we denote Sobolev spaces as (with the usual extension to $\textup{ess}\sup$ for $p = \infty$) 
\[ \norm{f}_{W^{k,p}} = \left(\sum_{|\alpha| \leq k}\int \abs{\grad^\alpha f}^p dy \right)^{\frac1p},
\]
and Fourier multipliers $m(\frac1i\grad)f$ as $\widehat{m(\frac1i\grad)f}(\xi) = m(\xi) \widehat{f}(\xi)$ (with the usual specialization in the event that we are only taking the Fourier transform in $z$). 
We use $\brak{x} = (1 + \abs{x}^2)^{1/2}$. As is customary, we subsequently write $\sup$ indistinguishably from $\textup{ess} \sup$ for notational simplicity.

Finally, the coefficient
$$
a(\tau) = 1 - e^{-\tau}
$$
will be handy in many estimates.

\subsection*{Acknowledgements} The authors wish to thank the anonymous referees for their careful reading of the manuscript and their numerous insightful and constructive comments.

\section{Statement of results and outline of the proof}

\subsection{Function spaces}
In order to state our results, it will be useful to first define several function spaces. 

To handle the self-similar part of the solution, for \(1\leq p<\infty\) and \(m\geq 0\) we define the weighted Lebesgue space \(L^p_\xi(m)\) with norm
\[
\|f\|_{L^p_\xi(m)}^p = \int_{\R^2}\<\xi\>^{pm}|f(\xi)|^p\,d\xi.
\]
In order to control the eigenfunctions of several linear operators, we extend this definition to \(m = \infty\) by defining the Hilbert space \(L^2_\xi(\infty)\) with inner product
\[
\< f,h\>_{L^2_\xi(\infty)} = \int_{\R^2} f(\xi)\cdot\overline h(\xi)\,G(\xi)^{-1}\,d\xi,
\]
where the Gaussian \(G\) is defined as in \eqref{omegag}.

We adopt the following normalization for the Fourier transform in the $z$-direction:
\[
\widehat f(\zeta) = \frac1{\sqrt{2\pi}}\int f(z)e^{-iz\zeta}\,dz.
\]
To control the regularity in the translation-invariant \(z\)-direction, for a Banach space \(X\) of functions defined on \(\R^2\), we define the X-valued Wiener algebra \(B_zX\) as the space of functions defined on \(\R^2\times \R\) or \(\R^2\times \T\) with norm
\[
\|f\|_{B_zX} = \begin{cases}\displaystyle\int_\R \|\widehat f(\cdot,\zeta)\|_X\,d\zeta,&\text{ for }(x,z)\in\R^2\times \R,\vspace{0.3cm}\\\displaystyle\sum\limits_{\zeta\in \Z}\|\widehat f(\cdot,\zeta)\|_X,&\text{ for }(x,z)\in\R^2\times \T.\end{cases}
\]

For initial data in ultra-critical spaces one generally cannot expect to have strong continuity up to time \(t = 0\). As a consequence, given a space of functions \(X\) continuously embedded in the space of tempered distributions \(\Schwartz'\), we say that \(\omega\in \mc C_w([0,T];X)\) if \(\omega\in L^\infty([0,T];X)\) and for all \(t_0\in[0,T]\) and test functions \(\phi\in \Schwartz\) we have
\[
\lim\limits_{\substack{t\rightarrow t_0\\t\in[0,T]}}\<\omega(t),\phi\>=\<\omega(t_0),\phi\>.
\]

In the case of the straight filament we require function spaces with some additional spatial summability. For a Sobolev-type space \(X\), and a smooth partition of unity \(1 = \sum_{M\in 2^\Z} \chi_M\) so that \(\chi_M = \chi_M(x)\) is a smooth, non-negative, radially symmetric, bump function supported in the annulus \(\{\frac M2\leq |x|\leq 2M\}\), we define
\begin{align}
\|f\|_{\ell^p X}^p = \sum\limits_{M\in 2^\Z}\|\chi_M f\|_X^p, \label{def:ell}
\end{align}
with the obvious modification for \(p = \infty\).

Finally, we give a rigorous definition of what we mean by a \emph{mild solution} of \eqref{NSE}:

\begin{defn}\label{milddef}
Let $\cM^{\frac{3}{2}}$ be the space of vector-valued regular Borel measures such that $$\norm{\mu}_{\cM^{\frac32}} := \sup_{r>0,y\in \R^3} \Big\{r^{-1} \abs{\mu\left(B(y,r)\right)}\Big\} < \infty.$$ 
Given a $T > 0$, we call a function $\omega \in \mc C_w([0,T];\cM^{\frac32})$ a \emph{mild solution} to \eqref{NSE} with initial data $\omega_0 \in \cM^{\frac32}$ provided 
\begin{itemize}
\item[(i)] the initial data is attained $\omega(0) = \omega_0$ (hence $\omega(t) \rightharpoonup^\ast \omega(0)$ as $t \searrow 0$); 
\item[(ii)] the equations are satisfied in the sense of Duhamel's formula \eqref{milddefDuhamel} (and in particular, the Duhamel integral is well-defined);
\item[(iii)] $\omega(t)$ is divergence free in the sense of distributions for all $t \in [0,T]$.  
\end{itemize}
\end{defn}

\subsection{The straight filament}
We are now in a position to state our main result for critical perturbations of the straight filament:
\begin{thrm}[Critical perturbations] \label{thm:StrtFil} For any $\alpha\in \R$, and any $m\geq 2$, there exists $\epsilon_0 > 0$ such that if \(\mu^b\colon\R^2\times \R\rightarrow\R^3\) or \(\mu^b\colon\R^2\times \T\rightarrow\R^3\) satisfies \(\nabla\cdot\mu^b = 0\) in the sense of distributions and the estimate
\eq{topologydata}{
\| \mu^b \|_{B_z L^1_x} +  \|x \cdot (\mu^b)^x \|_{\ell^1B_z L^{2}_x} = \epsilon < \epsilon_0,
}
where \(\mu^b=((\mu^b)^x,(\mu^b)^z)^T\in \R^2\times \R\), then the following holds: 
\begin{itemize}
\item[(i)](Existence) There exists a global mild solution of the Navier-Stokes equation~\eqref{NSE} with initial data
\eq{StraightInit}{
\omega(t=0) = 
\begin{bmatrix}
0 \\ 
\alpha \delta_{x = 0}
\end{bmatrix} + \mu^b,
}
which can be decomposed into
\begin{align}
\omega(t,x,z)  = 
\begin{bmatrix}
0 \\
\frac{\alpha}{t}G\left(\frac{x}{\sqrt{t}}\right)
\end{bmatrix}
+ 
\tfrac{1}{t}\Omega^c \left(\log t, \tfrac{x}{\sqrt{t}},z \right) + \omega^b(t,x,z),  \label{def:destrt}
\end{align}
where the ``core" part $\Omega^c$ and the ``background" part $\omega^b$ satisfy the estimates
\eq{topologysolution}{
\sup_{-\infty<\tau<\infty} \norm{\Omega^c(\tau)}_{B_z L^2_\xi(m)}  + \sup_{0<t<\infty} t^{\frac14}\norm{\omega^b(t)}_{B_z L^{4/3}_x}\lesssim \eps.  
}
\item[(ii)](Uniqueness) If $\omega'$ is another mild solution with initial data \eqref{StraightInit} admitting the decomposition
$$
\omega'(t,x,z) = 
\begin{bmatrix}
0 \\ 
\frac{\alpha}{t}G\left(\frac{x}{\sqrt{t}}\right)
\end{bmatrix}
+ 
\tfrac{1}{t}(\Omega^c)' \left(\log t, \tfrac{x}{\sqrt{t}},z \right) + (\omega^b)'(t,x,z),
$$
where $(\Omega^c)'$ and $(\omega^b)'$ satisfy the bounds~\eqref{topologysolution}, then $\omega = \omega'$.
\item[(iii)](Lipschitz dependence) The solution map from the data to solution
$$
\mu^b \mapsto (\omega^b,\Omega^c)
$$
is locally Lipschitz continuous if one endows the data space with the norm~\eqref{topologydata} and the solution space with the norm~\eqref{topologysolution}. Similarly, the solution also depends on $\alpha$ in a locally Lipschitz manner.
\end{itemize}
\end{thrm}

\begin{rem}
This theorem remains true if $B_z$ is replaced everywhere by the space of Fourier transform of measures $\widehat{\fM}$: for a Banach space $X$, $\widehat{\fM}_z X_x$ is the space of Fourier transforms (in $z$) of $X$ (in $x$)-valued measures. The proof is identical. This framework allows data and solutions that do not decay as $z \to \infty$.
\end{rem}
\begin{rem}
Further refinements of Theorem~\ref{thm:StrtFil} part (ii) have been investigated in~\cite{bedrossian2020uniqueness}, where it is shown that any (not too singular) solution which is `sufficiently two dimensional' in a suitable sense as $t \searrow 0$ is the same as the solution we construct (one can also consider the condition to be that the initial datum is attained in a suitably stronger sense than just $\mc C_w([0,T];\cM^{\frac32})$.  
\end{rem}

The proof of Theorem~\ref{thm:StrtFil} follows from applying the contraction principle to the equations satisfied by the core and background pieces.
The decomposition is reminiscent of that used in the proof of uniqueness in \cite{MR2178064} and the contraction principle variant thereof used in \cite{MR3269635}. 
 In order to obtain bounds for these pieces we first introduce the self-similar coordinates
\[
\tau = \log t,\qquad \xi = \frac{x}{\sqrt{ t}},\qquad z = z,
\]
where we note that as \(\omega^g\) (defined as in \eqref{omegag}) is translation-invariant in \(z\) we do not rescale the \(z\)-coordinate. We then define
\[
\Omega(\tau,\xi,z) =  e^\tau  \omega(e^\tau,e^{\frac \tau2}\xi,z),\qquad U(\tau,\xi,z) =e^{\frac\tau2} u (e^\tau,e^{\frac \tau2}\xi,z),
\]
and may write the equation \eqref{NSE} as
\eq{NSESS}{
\partial_\tau \Omega + U\cdot\onabla\Omega - \Omega\cdot\onabla U = \left(\cL + e^\tau\partial_z^2\right)\Omega,
}
where the rescaled gradient \(\onabla\) and the \(2d\) Fokker-Planck operator \(\cL\) are defined by,
\[
\onabla = \begin{bmatrix}\nabla_\xi\\e^{\frac \tau 2}\partial_z\end{bmatrix},\qquad \cL = \Delta_\xi + \frac12 \xi \cdot\nabla_\xi + 1.
\]
We also note that under this change of variables the Biot-Savart law becomes
\eq{BSSS}{
U = (- \oDelta)^{-1}\onabla\times\Omega,
}
where the rescaled Laplacian,
\[
\oDelta = \Delta_\xi + e^\tau\partial_z^2.
\]
Finally, we will denote $\Omega^g$ and $U^g$ for the rescaled versions of $\omega^g$ and $u^g$,
$$
\Omega^g(\tau,\xi,z) = \begin{bmatrix}0\\  G(\xi)\end{bmatrix}, \qquad U^g(\tau,\xi,z) = \begin{bmatrix} g(\xi) \\ 0\end{bmatrix}.
$$

The core piece, $\Omega^c$, is taken to satisfy the equation
\[
\pde{
\partial_\tau \Omega^c + U\cdot\onabla(\alpha\Omega^g + \Omega^c) - (\alpha\Omega^g + \Omega^c)\cdot\onabla U = \left(\cL + e^\tau\partial_z^2\right)\Omega^c,
}{
\lim\limits_{\tau\rightarrow-\infty}\Omega^c(\tau) = 0,
}
\]
In order to construct solutions we first prove estimates for the solution operator \(\Omega(\tau) = S(\tau,\sigma)\Omega(\sigma)\) of the corresponding linearized equation
\[
\pde{
\partial_\tau \Omega + \alpha [ U^g\cdot\onabla\Omega + U\cdot\onabla \Omega^g - \Omega\cdot\onabla U^g - \Omega^g\cdot\onabla U ] = \left(\cL + e^\tau\partial_z^2\right)\Omega,
}{
U = (-\oDelta)^{-1}\onabla\times \Omega. 
}
\]
The key to our argument is the observation that, in the limit \(\tau\rightarrow-\infty\), the equations decouple into a pair of \(z\)-independent linear equations, with a coupling $(Z^\xi,Z^z)$ which contains $e^{\frac \tau2}\partial_z$ derivatives and so is formally time-integrable:
\[
\pde{
\partial_\tau\Omega^\xi + \alpha [g\cdot\nabla_\xi\Omega^\xi - \Omega^\xi\cdot\nabla_\xi g ]- \cL \Omega^\xi = \alpha Z^\xi ,
}{
\partial_\tau \Omega^z + \alpha [ g\cdot\nabla_\xi\Omega^z - (-\Delta_\xi)^{-1}\nabla_\xi^\perp \Omega^z\cdot \nabla_\xi G ] - \cL\Omega^z = \alpha Z^z.
}
\]
The first of these linear equations appeared in the context of Burgers vortices in~\cite{MR2770021}, whereas the second is precisely the \(2d\) Navier-Stokes equations linearized around the self-similar solution, which has been extensively studied in~\cite{MR1912106,MR2123378} (see also \cite{LWZ17}). In order to rigorously reduce the full system to the limiting case \(\tau = -\infty\) we use translation-invariance in \(z\) to take the Fourier transform in $z$ and estimate frequency-by-frequency. Exponential decay and smoothing estimates for the $2d$ linear semigroups defined by the operators on the left and taking advantage of the general structure of the $Z^\xi$ and $Z^z$ terms permits one to obtain uniform-in-frequency stability. The analysis of the linear propagator \(S(\tau,\sigma)\) is carried out in Section \ref{sec:SSlin}.

The background piece, $\omega^b$, is taken to satisfy the equation
\[
\pde{
\partial_t \omega^b + u\cdot\nabla\omega^b - \omega^b\cdot\nabla u = \Delta \omega^b,
}{
\omega^b(0) = \mu^b.
}
\]
Solutions are then constructed by establishing estimates for the solution operator \(\omega(t) = \bS(t,s)\omega(s)\) for the corresponding linearized equation
\[
\partial_t \omega + \alpha u^g\cdot\nabla\omega - \alpha \omega\cdot\nabla u^g = \Delta \omega.
\]
The analysis in this case is similar to the core piece, taking the Fourier transform in \(z\) and treating the resulting system as a perturbation of a system of \(2d\) equations. The $2d$ semigroup estimates are obtained by methods similar to those applied for the \(2d\) case considered in~\cite{MR2178064}. However, here the vortex stretching causes additional difficulties in obtaining estimates for the operator \(\bS(t,0)\) that are not present in \(2d\). These difficulties are overcome by taking advantage of the special structure of the equation satisfied by the radial component of the vorticity \(x\cdot\omega^x\). The analysis of the linear propagator \(\bS(t,s)\) is carried out in Section \ref{sec:StretchLin}.

The bulk of the work for the straight filament is to obtain suitable estimates for the linear propagators \(S(\tau,\sigma)\), \(\bS(t,s)\). Given these bounds, the proof of Theorem~\ref{thm:StrtFil} follows from an elementary application of the contraction principle that we carry out in Section \ref{sec:Straight}.

We remark that closing the contraction in Theorem~\ref{thm:StrtFil} essentially relies on the fact that the operator \(\bS(t,s)\) for the background piece satisfies the estimate
\[
t^{\frac14}\|\bS(t,0)\mu^b\|_{B_zL^{4/3}_x}\ll1.
\]
As usual, if we work with subcritical perturbations of the straight filament the smallness of the data is replaced by a short-time assumption. 

\begin{thrm}[Subritical perturbations] \label{thm:StrtFil2} For any $\alpha\in \R$, \(1<p\leq \frac43\), $m\geq2$, and function \(\mu^b\colon\R^2\times \R\rightarrow\R^3\) or \(\mu^b\colon\R^2\times \T\rightarrow \R^3\) satisfying \(\nabla\cdot\mu^b = 0\) in the sense of distributions and
\eq{topologydata2}{
\| \mu^b \|_{B_z L^p_x} +  \|x \cdot (\mu^b)^x \|_{B_z L^{\frac{2p}{2-p}}_x} = K,
}
there exists \(T = T(\alpha,p,m,K)>0\) such that:

\begin{itemize}
\item[(i)](Existence) There exists a mild solution of the Navier-Stokes equation~\eqref{NSE} on the time interval \([0,T]\) with initial data
\eq{StraightInit2}{
\omega(t=0) = 
\begin{bmatrix}
0 \\
\alpha \delta_{x = 0}
\end{bmatrix} + \mu^b,
}
which can be decomposed into
\[
\omega(t,x,z)  = 
\begin{bmatrix}
0 \\
\frac{\alpha}{t}G\left(\frac{x}{\sqrt{t}}\right)
\end{bmatrix}
+ 
\tfrac{1}{t}\Omega^c \left(\log t, \tfrac{x}{\sqrt{t}},z \right) + \omega^b(t,x,z), 
\]
where the ``core" part $\Omega^c$ and the ``background" part $\omega^b$ satisfy the estimates
\eq{topologysolution2}{
\lim_{t \searrow 0}\left( \sup_{-\infty<\tau\leq \ln t} \norm{\Omega^c(\tau)}_{B_z L^2_\xi(m)}  + \sup_{0<s\leq t} s^{\frac14}\norm{\omega^b(s)}_{B_z L^{4/3}_x}\right) = 0. 
}
\item[(ii)](Uniqueness) If $\omega'$ is another mild solution on the time interval \([0,T]\) with initial data \eqref{StraightInit2} admitting the decomposition
$$
\omega'(t,x,z) = 
\begin{bmatrix}
0 \\
\frac{\alpha}{t}G\left(\frac{x}{\sqrt{t}}\right)
\end{bmatrix}
+ 
\tfrac{1}{t}(\Omega^c)' \left(\log t, \tfrac{x}{\sqrt{t}},z \right) + (\omega^b)'(t,x,z),
$$
where $(\Omega^c)'$ and $(\omega^b)'$ satisfy~\eqref{topologysolution2}, then $\omega = \omega'$.
\item[(iii)](Lipschitz dependence) The solution map from the data to solution
$$
\mu^b \mapsto (\omega^b,\Omega^c)
$$
is locally Lipschitz continuous if one endows the data space with the norm~\eqref{topologydata2} and the solution space with the norm appearing in~\eqref{topologysolution2}. Similarly, the solution depends in a locally Lipschitz manner on $\alpha$. 
\end{itemize}
\end{thrm}

\subsection{The curved filament} \label{sec:curvedIntro}
Our second set of results concern the case that \(\Gamma\subset \R^3\) is a smooth, non-self-intersecting, closed curve that, after rescaling, may be assumed to have length \(2\pi\). The key to our approach in this case is that on sufficiently short timescales, the curvature of the filament is expected to be \emph{subcritical}. 
Making this intuition rigorous is rather involved, however, it ultimately allows us to treat the general problem as a perturbation of the straight filament by introducing local coordinates near the filament that ``straighten out'' the curve and choosing $T$ sufficiently small. 

We define a unit speed parameterization \(\gamma\colon \T\rightarrow \R^3\) and an orthonormal frame \(\ft,\fn,\fb\colon \T\rightarrow\R^3\) along \(\Gamma\) so that \(\ft = \gamma'\) is the unit tangent vector and the frame is oriented such that \(\fb = \ft\times \fn\). In the case that \(\Gamma\) has non-vanishing curvature, an explicit example is given by the Frenet-Serret frame,
\[
\ft = \gamma',\qquad \fn = \frac{\gamma''}{|\gamma''|},\qquad \fb = \ft\times \fn,
\]
for which we have the Frenet-Serret formulas,
\[
\ft' = \kappa \fn,\qquad \fn' = -\kappa \ft + \tau \fb,\qquad \fb' = - \tau \fn,
\]
where \(\kappa\) is the curvature and \(\tau\) is the torsion.

For each \(R>0\) we define a tubular neighborhood of \(\Gamma\) of radius \(32R\),
\[
\Gamma_R = \left\{y\in \R^3:\dist(y,\Gamma)<32R\right\},
\]
and a corresponding straight tube
\[
\Sigma_R = \left\{(x,z)\in \R^2\times \T:|x|<32R\right\}.
\]
Choosing \(0<R_0\ll1\) sufficiently small (depending on the curvature of \(\Gamma\)) we may view \(\Gamma_{R_0}\), considered to live in the ``physical frame,'' as the image of the open set \(\Sigma_{R_0}\), considered to live in a ``straightened frame,'' under the map \(\Phi\colon \Sigma_{R_0}\rightarrow \Gamma_{R_0}\) defined by
\[
\Phi(x,z) = \gamma(z) + x_1\fn(z) + x_2\fb(z).
\]
We define the following mapping, which transforms vorticity defined in the straightened frame back into the physical frame   
\[
(Q_\Phi \eta)\circ\Phi = (\det\grad \Phi)^{-1} (\grad \Phi)\  \eta.
\]
Further, we define $\chi_R$ to be a smooth, non-negative, radial bump function supported on $\abs{x} \leq 2R$ and identically equal to $1$ for $\abs{x} \leq R$, and take $\widetilde{\chi}_R = \chi_R \circ \Phi^{-1}$. Finally, define the approximate solution 
\[
\eta^g(x,z) = \begin{bmatrix} 0  \\ \frac1{t} G \left(\frac{x}{\sqrt{t}} \right) \end{bmatrix}, \qquad \omega^g = Q_\Phi\left(\chi_{2R} \eta^g \right). 
\]
\begin{figure}[h]
\centering
\begin{tikzpicture}
\shade[shading=axis, draw=none, left color=black!35!white, right color=black!5!white, shading angle=135] (-.5,0)--(-.5,6) arc (180:0:0.5 and 0.25) -- (.5,6)--(.5,0)  arc (360:180:0.5 and 0.25) --cycle;
\draw[very thick] (0,0)--(0,6);
\draw (0,6) ellipse (0.5 and 0.25);
\draw (-.5,0)--(-.5,6);
\draw (.5,0)--(.5,6);
\draw (.5,0) arc (360:180:0.5 and 0.25);
\draw[dashed] (-.5,0) arc (180:0:0.5 and 0.25);
\filldraw[shading=axis, draw=none, left color=black!35!white, right color=black!5!white, shading angle=135] (7,3) ellipse (2.5 and 1.35);
\filldraw[fill=white, draw=none] (7,3) ellipse (1.5 and 0.65);
\draw[very thick] (7,3) ellipse (2 and 1);
\draw (7,3) ellipse (2.5 and 1.35);
\draw (7,3) ellipse (1.5 and 0.65);
\draw[dashed] (5.5,3) arc (360:180:0.5 and 0.4);
\draw (4.5,3) arc (180:0:0.5 and 0.4);
\draw[dashed] (9.5,3) arc (360:180:0.5 and 0.4);
\draw (8.5,3) arc (180:0:0.5 and 0.4);
\draw (7,1.65) arc (270:450:0.2 and 0.35);
\draw[dashed] (7,2.35) arc (450:630:0.2 and 0.35);
\draw (7,3.65) arc (270:450:0.2 and 0.35);
\draw[dashed] (7,4.35) arc (450:630:0.2 and 0.35);
\draw[very thick,->] (1,3) .. controls (2,3.2) and (3,3.2) .. (4,3);
\node at (2.5,3.5) {\(\Phi\)};
\node at (-1,2) {\(\Sigma_{R_0}\)};
\node at (9.5,1.9) {\(\Gamma_{R_0}\)};
\end{tikzpicture}
\caption{The mapping \(\Phi\) from the straightened frame to the physical frame.}\label{fig:Phi}
\end{figure}
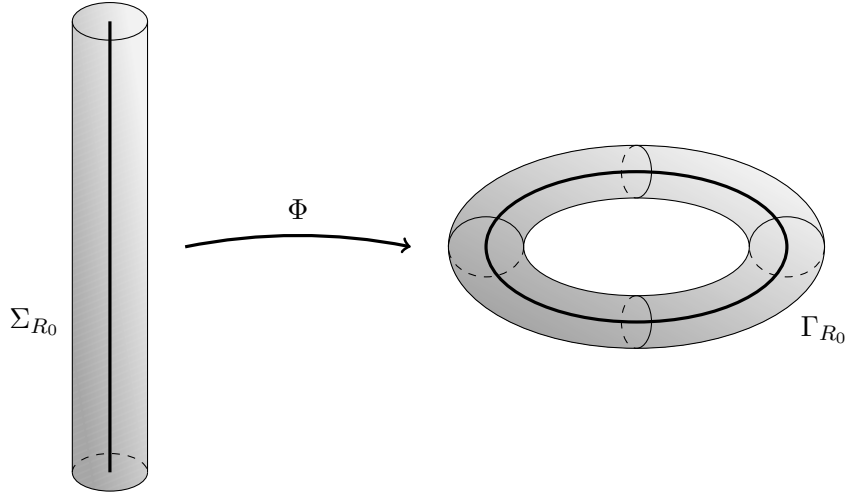

\begin{thrm} \label{thm:curved}
Let $\alpha\in \R$ and $\Gamma \subset \R^3$ be a smooth, non-self-intersecting, closed curve. 
For any initial condition satisfying 
\begin{align}
\omega(t=0) = \alpha\delta_\Gamma + \mu^b, 
\end{align}
where $\mu^b \in W^{1,\frac{12}{11}}$, there is a $T = T(\alpha,\Gamma,\mu^b)$ such that
\begin{itemize} 
\item[(i)](Existence) There exists a mild solution $\omega$ of \eqref{NSE} on $[0,T]$ that admits the decomposition $\omega= \widetilde{\omega}^c + \omega^b$ satisfying (as mild solutions with $u = \grad \times (-\Delta)^{-1}\omega$) 
\begin{align}
&\pde{
\partial_t \widetilde{\omega}^c + B[u,\widetilde{\omega}^c] = \Delta \widetilde{\omega}^c,
}{ 
\widetilde{\omega}^c(t=0) = \alpha \delta_{\Gamma},
}
\\
&\pde{
\partial_t \omega^b +  B[u,\omega^b] = \Delta \omega^b,
}{
\omega^b(t=0) = \mu^b,  
}
\end{align}
such that for any $m \geq 2$ and sufficiently small $0<R \leq R_0$,
\begin{equation} \label{CurveEsts}
\begin{split} 
&\lim_{T \searrow 0} \sup_{0 < t \leq T}\sqrt{t}\norm{\left\<\frac{x}{\sqrt{t}}\right\>^m Q_{\Phi}^{-1} \left(\widetilde\chi_{8R}\left(\widetilde{\omega}^c - \alpha\omega^g\right)\right)}_{B_z L^2_x} = 0,\\
&  \lim_{T \searrow 0} \sup_{0 < t \leq T} \sqrt{t}\norm{(1 - \widetilde{\chi}_{6R})\left\<\frac{\dist(y,\Gamma)}{\sqrt{t}}\right\>^m \widetilde{\omega}^c}_{L^3_y}  = 0,\\
&\lim_{T \searrow 0} \sup_{0 < t \leq T} t^{\frac14}\norm{ \omega^b }_{W^{1,4/3}_y} = 0.
\end{split}
\end{equation}
Further, taking \(\Omega^c\) to be the self-similar scaling of \(\omega^c = \widetilde \omega^c - \alpha\omega^g\), there holds the following decomposition for $(x,z) \in \Gamma_{R/64}$ and \(0<t \leq T\)
\begin{align}
\left(\left(\det\grad \Phi \right) (\grad \Phi)^{-1}\right)\!(x,z)\; \omega(t,\Phi(x,z)) = \begin{bmatrix}
0 \\
\frac{\alpha}{t}G\left(\frac{x}{\sqrt{t}}\right)
\end{bmatrix}
+ 
\tfrac{1}{t}\Omega^c \left(\log t, \tfrac{x}{\sqrt{t}},z \right) + \omega^b(t,x,z).
\end{align}
\item[(ii)] (Uniqueness) Suppose $\omega'$ is another mild solution with initial data \eqref{StraightInit}, suppose $(\widetilde{\omega}^c)'$ and $(\omega^b)'$ are analogous to the definition in (i), and suppose that these satisfy the estimates in \eqref{CurveEsts} for some $m \geq 2$ and sufficiently small $ 0<R\leq R_0$. Then $\omega' = \omega$.
\item[(iii)] (Lipschitz dependence) The solution map from the data to solution
$$
\mu^b \mapsto (\omega^b,\Omega^c)
$$
is Lipschitz continuous if one endows the data space with $\norm{\cdot}_{W^{1,\frac{12}{11}}}$ and the solution space with the norm appearing in \eqref{CurveEsts} (more precise estimates are available below). 
\end{itemize}
\end{thrm}

\begin{rem}
We do not expect that the requirement $\mu^b \in W^{1,\frac{12}{11}}$ is sharp; it would be more natural to expect $\mu^b \in L^{\frac32}$ to be sufficient, however, this would require some non-trivial technical refinements. We also do not expect the uniqueness statement in (ii) above to be sharp.
\end{rem}

Let us briefly mention some of the difficulties in making the nonlinear perturbation argument for Theorem \ref{thm:curved}. 
In Section \ref{sec:ChangeVars}, the properties of the coordinate system that straighten the arbitrary filament are recorded. 
In particular, we see that all the alterations are either lower order (in terms of derivatives) or have coefficients with size $O(\abs{x})$, and hence will be asymptotically small as $t \searrow 0$ as the vorticity will be concentrated mostly in a tubular neighborhood of the filament of size $O(\sqrt{t})$.  
Hence, we can expect all of the curvature effects to be \emph{subcritical}. 
There are two major technical issues with making this rigorous: 
\begin{itemize}
\item[\textbf{(a)}] In the straightened coordinate system $\Delta$ has been replaced by a second order operator with variable coefficients, even for top order terms (see Section \ref{sec:ChangeVars}). This makes the curvature effects difficult to treat in a perturbative manner. 
\item[\textbf{(b)}] The straightened coordinate system only makes sense very close to the filament; away from it, we cannot use the $B_z L^p$ spaces, which are closely adapted to the geometry. This is problematic since the natural anisotropic $B_z L^p$ spaces used in the straight filament assign far more regularity along the filament than transversally. Making a smooth transition to isotropic regularity is delicate as, heuristically, it requires going up in regularity in the transverse directions and down in regularity along the filament.
\end{itemize} 
To deal with the difficulty posed by \textbf{(a)}, we will decompose the natural analogues of $\omega^c$ and $\omega^b$ each into two sub-pieces; a primary $\omega^{\ast 1}$ (for \(\ast=c,b\)) which will describe the leading order `critical' behavior near the filament and a secondary $\omega^{\ast 2}$ which will deal with some of the most problematic subcritical errors coming from the geometry. 
The $\omega^{\ast 2}$ unknowns will live in the original (physical) variables and are solved using the usual heat semigroup, whereas the $\omega^{\ast 1}$ unknowns are naturally formulated in the straightened coordinates and will require the use of the two straight filament propagators \(S(\tau,\sigma)\), \(\bS(t,s)\). Accordingly the $\omega^{\ast 1}$ unknowns are estimated in a manner similar to the straight filament whereas the $\omega^{\ast 2}$ estimates require different arguments. In order to deal with the errors in the viscosity term, the $\omega^{\ast 2}$ unknowns will have slightly lower regularity than the $\omega^{\ast 1}$ counterparts.

In order to deal with \textbf{(b)}, we will need two technical ideas. First, we will change the style of the norms we are using on $\omega^b$ (relative to the straight filament). In particular, we will be using a slightly stronger set of norms that are naturally isotropic but are also critical and satisfy the proper embeddings into the anisotropic spaces. This compromise explains the need to take perturbations that are more subcritical than what was needed in the straight filament, Theorem \ref{thm:StrtFil2}. Second, for $\omega^{c2}$, which interacts directly with the most singular piece, $\omega^{c1}$, we will not be able to avoid transitioning from isotropic to anisotropic. For this we obtain anisotropic estimates near the filament and isotropic estimates at higher regularity far enough from the filament; the overlap region is the most difficult. The details of how to carry out the perturbation argument are rather technical and are left to Sections~\ref{sec:Curve1} and~\ref{sec:Curve2}.

\section{The linearized problem in self-similar variables} \label{sec:SSlin}

\subsection{Statement of the estimates}

In this section we consider the linearization of the equation \eqref{NSE} about the self-similar solution \(\omega^g\),
\eq{LinearizedProblem}{
\pde{
\partial_t\omega + \alpha [u^g\cdot\nabla\omega + u\cdot\nabla\omega^g - \omega^g\cdot\nabla u - \omega \cdot\nabla u^g ] = \Delta \omega,
}{
\nabla\cdot \omega = 0.
}
}
Switching to self-similar coordinates we obtain the system,
\eq{Linear}{
\pdet{
\partial_\tau \Omega^\xi + \alpha [ g\cdot\nabla_\xi \Omega^\xi - \Omega^\xi\cdot\nabla_\xi g - e^{\frac\tau2} G\partial_zU^\xi] = \left(\cL + e^\tau\partial_z^2\right)\Omega^\xi,
}{
\partial_\tau \Omega^z + \alpha [g\cdot\nabla_\xi \Omega^z +  U^\xi\cdot\nabla_\xi G - e^{\frac\tau2} G\partial_zU^z] = \left(\cL + e^\tau\partial_z^2\right)\Omega^z,
}{
\nabla_\xi\cdot \Omega^\xi + e^{\frac \tau2}\partial_z\Omega^z = 0,
}
}
where the Biot-Savart law is given by
\eq{BSMain2}{
\pde{
U^\xi = e^{\frac\tau2}\partial_z(-\overline \Delta)^{-1} (\Omega^\xi)^\perp - \nabla_\xi^\perp (-\overline \Delta)^{-1}\Omega^z,
}{
U^z = \nabla_\xi^\perp\cdot(-\overline \Delta)^{-1}\Omega^\xi.
}
}
with the following notations for differential operators in self-similar coordinates:
\begin{align*}
& \overline{\nabla} = (\nabla_\xi,e^{\frac\tau2} \partial_z)^T, \qquad \qquad \overline{\nabla}^\beta = \partial_{\xi_1}^{\beta_1} \partial_{\xi_2}^{\beta_2} (e^{\frac\tau2} \partial_z)^{\beta_3} \quad\mbox{if $\beta = (\beta_1,\beta_2,\beta_3)$},\\
& \oDiv F = \Div_\xi F^\xi + e^{\frac\tau2}\partial_zF^z, \qquad \overline\Delta = \Delta_\xi + e^\tau \partial_z^2.
\end{align*}

For \(\tau\geq\sigma\) we define the solution operator \(S(\tau,\sigma)\) for the equation \eqref{Linear} by
\[
\Omega(\tau) = S(\tau,\sigma)\Omega(\sigma).
\]
\textit{For the remainder of this section, we adopt the convention that the semigroup $S(\tau,\sigma)$ is defined for $\Omega(\sigma)$ which might have a non-zero divergence (which simply amounts to lifting the last condition in~\eqref{Linear}). Notice that the condition that $\overline{\nabla} \cdot \Omega = 0$ is propagated by the flow.
}

\medskip

In this section we prove the following result:

\begin{thrm}\label{prop:LinearEstimates}
Let \(\alpha\in\R\) and \(m>1\). Then, for all \(\sigma\in \R\) the map \(\tau\mapsto S(\tau,\sigma)\) is continuous as a map from \([\sigma,\infty)\) to the space of bounded operators on \(B_zL^2_\xi(m)\). For all $\gamma>0$, we have the estimate
\eq{BasicBound}{
\|S(\tau,\sigma)F\|_{B_zL^2_\xi(m)}\lesssim e^{\gamma(\tau-\sigma)} \|F\|_{B_zL^2_\xi(m)} ,
}
where the implicit constant depends on \(\alpha,m,\gamma\).

Further, if \(\alpha\neq 0\) there exists $\mu = \mu(\alpha) \in (0,\frac{1}{2})$ such that, whenever \(m>1 + 2\mu\) and \(F\in B_zL^2_\xi(m)\) satisfies \(\int F^z\,d\xi = 0\), we have the estimate,
\eq{MainBound}{
\|S(\tau,\sigma)F\|_{B_z L^2_\xi(m)} \lesssim e^{-\mu(\tau - \sigma)}\|F\|_{B_z L^2_\xi(m)},
}
where the implicit constant depends on \(\alpha,m\).

Finally, if \(\alpha\neq 0\) and \(m>1 + 2\mu\) where \(\mu = \mu(\alpha)\) is as above, then for \(1 < p\leq 2\) and all \(3\times 3\)-tensors \(F\in B_zL^p_\xi(m)\) satisfying \(\oDiv\,\oDiv F = 0\) we have the estimates,
\begin{align}
\|S(\tau,\sigma)\oDiv F\|_{B_z L^2_\xi(m)} &\lesssim \frac{e^{-\mu(\tau - \sigma)}}{a(\tau -\sigma)^{\frac1p}}\|F\|_{B_z L^p_\xi(m)},\label{SmoothedBound}\\
\|\onabla S(\tau,\sigma)\oDiv F\|_{B_z L^2_\xi(m)} &\lesssim \frac{e^{-\mu(\tau - \sigma)}}{a(\tau -\sigma)^{\frac1p + \frac12}}\|F\|_{B_z L^p_\xi(m)},\label{SmoothedBound1}
\end{align}
where \(a(\tau) = 1 - e^{-\tau}\) and the implicit constants depend on \(\alpha,m,p\).
\end{thrm}
Only estimates~\eqref{SmoothedBound} and~\eqref{SmoothedBound1} will be used in controlling the nonlinear problem: the former in the case of the straight filament (with \(p = 4/3\)) and the latter in order to derive fractional regularity by interpolation, which will be needed to deal with the curved filament. As for estimate~\eqref{BasicBound}, it guarantees that the the flow is well-defined on $L^2_\xi(m)$. Finally, estimate~\eqref{MainBound} is used as an intermediary step.

The proof of Theorem~\ref{prop:LinearEstimates} will follow a similar strategy to the proof of~\cite[Proposition~4.6]{MR2178064}, first proving long-time estimates for the operator \(S(\tau,\sigma)\) on \(L^2_\xi(m)\), and then combining this with short time smoothing estimates to obtain the estimates \eqref{SmoothedBound}, \eqref{SmoothedBound1}. A key difficulty we encounter in \(3d\) is that the operator is no longer a compact perturbation of the Fokker-Planck operator \(\cL\), indeed it is translation-invariant in \(z\). However, we may take advantage of this translation-invariance by taking Fourier transform in \(z\) and then estimating the resulting operator frequency-by-frequency. In particular, we will show that we may reduce to the linear operator at fixed \(z\)-frequency, which is a compact perturbation of the Fokker-Planck operator \(\cL\).

\subsection{Long time estimates}
In this section we prove that the solution operator \(S(\tau,\sigma)\) is well-defined, and satisfies the estimates \eqref{BasicBound}, \eqref{MainBound}.

We start by taking the Fourier transform in \(z\) of the equation \eqref{Linear} and setting $w(\tau,\xi,\zeta) = \widehat{\Omega}(\tau,\xi,\zeta)$ to obtain the system,
\eq{FLinear}{
\pde{
\left(\partial_\tau + e^\tau|\zeta|^2 - \cL + \alpha\Gamma\right)w^\xi = \alpha Z^\xi(w),
}{
\left(\partial_\tau + e^\tau|\zeta|^2 - \cL + \alpha\Lambda\right)w^z = \alpha Z^z(w),
}}
where the linear operators are denoted
\[
\Gamma = g\cdot\nabla_\xi - \nabla_\xi g,\qquad \Lambda = g\cdot\nabla_\xi - \nabla_\xi G\cdot \nabla_\xi^\perp(-\Delta_\xi)^{-1},
\]
and the perturbative terms are given by
\begin{align*}
Z^\xi(w) &=  i e^{\frac\tau 2} \zeta G \widehat{U^\xi} \\&= i e^{\frac\tau2} \zeta G \left[ ie^{\frac\tau2}\zeta(e^\tau|\zeta|^2 - \Delta_\xi)^{-1}(w^\xi)^\perp - \nabla_\xi^\perp(e^\tau|\zeta|^2 - \Delta_\xi)^{-1}w^z \right],\\
Z^z(w) &= i e^{\frac\tau2} \zeta G \widehat{U^z} - \nabla_\xi G\cdot\left( \widehat U^\xi + \nabla_\xi^\perp(-\Delta_\xi)^{-1}w^z\right)\\
&= i e^{\frac\tau2}  \zeta\nabla_\xi^\perp\cdot \left( G (e^\tau|\zeta|^2 - \Delta_\xi)^{-1}w^\xi \right) + \nabla_\xi G\cdot\left( (e^\tau|\zeta|^2 - \Delta_\xi)^{-1} - (-\Delta_\xi)^{-1}\right)\nabla_\xi^\perp w^z.
\end{align*}

The existence of the solution operator \(S(\tau,\sigma)\) and the estimates \eqref{BasicBound}, \eqref{MainBound} are given in the following proposition:
\begin{prop}\label{prop:BasicAP}
Let \(m>1\) and \(\zeta\in \R\) be fixed. Then, for all \(\sigma\in \R\) and all \(w_\sigma\in L^2_\xi(m)\), there exists a unique mild solution \(w\in \mc C([\sigma,\infty);L^2_\xi(m))\) of the equation \eqref{FLinear} satisfying \(w(\sigma) = w_\sigma\). For all $\gamma>0$, it satisfies the estimate
\eq{BasicAP}{
\|w(\tau)\|_{L^2_\xi(m)}\lesssim e^{\gamma(\tau-\sigma)}\|w_\sigma\|_{L^2_\xi(m)}.
}

Further, if \(\alpha\neq 0\) then there exists some \(0<\mu = \mu(\alpha)<\frac12\) so that, whenever \(m>1 + 2\mu\) and \(w_\sigma\in L^2_\xi(m)\) satisfies \(\int w_\sigma^z\,d\xi = 0\), we have the improved estimate
\eq{ImprovedAP}{
\|w(\tau)\|_{L^2_\xi(m)}\lesssim e^{-\mu(\tau - \sigma)}\|w_\sigma\|_{L^2_\xi(m)}.
}
In both estimates, the implicit constant is independent of $\zeta$.
\end{prop}
In order to prove Proposition~\ref{prop:BasicAP} we first prove estimates for the Biot-Savart operator in self-similar coordinates:
\begin{lem}\label{lem:Nonlocal} Let \(m>1\) and \(\lambda>0\).

\begin{itemize}
\item[(i)] If \(1< r\leq \infty\) and \(0<\delta \leq\min\{\frac12, 1 - \frac1r\}\),
\eq{Bessel}{
\|(\lambda^2 - \Delta_\xi)^{-1}f\|_{L^r_\xi} \lesssim \lambda^{-(\frac 2r + 2\delta)}\|f\|_{L^2_\xi(m)}.
}

\item[(ii)] If \(1<r\leq 2\) and \(0<\delta\leq1 - \frac1r\),
\eq{Bessel2}{
\|\nabla_\xi(\lambda^2 - \Delta_\xi)^{-1}f\|_{L^r_\xi} \lesssim \lambda^{1 - \frac 2r - 2\delta}\|f\|_{L^2_\xi(m)}.
}

\item[(iii)] If \(2<r<\infty\),
\eq{Bessel3}{
\|\nabla_\xi(-\Delta_\xi)^{-1}f\|_{L^r_\xi}\lesssim \|f\|_{L^2_\xi(m)}.
}
\end{itemize}
\end{lem}

\begin{proof}~
\noindent
\emph{i)} We may write
\[
(\lambda^2 - \Delta_\xi)^{-1}f (\xi) = \int_{\R^2}K(\lambda(\xi - \eta))f(\eta)\,d\eta,
\]
where \(K\) is the \(2d\) Bessel potential. Recalling that \(K\in L^p_\xi\) for all \(1\leq p<\infty\) and taking \(\frac 1p + \frac 1q = 1 + \frac 1r\), we obtain
\[
\|(\lambda^2 - \Delta_\xi)^{-1}f\|_{L^r_\xi}\lesssim \lambda^{-\frac 2p}\|K\|_{L^p_\xi}\|f\|_{L^q_\xi}.
\]
For \(1\leq q\leq2\) and \(m>\frac2q - 1\) we have the embedding,
\[
\|f\|_{L^q_\xi}\lesssim \|f\|_{L^2_\xi(m)}.
\]
In particular, taking \(\frac1p = \frac 1r + \delta\) and \(\frac1q = 1 - \delta\) we obtain the estimate \eqref{Bessel}.

\smallskip
\noindent
\emph{ii)} The estimate \eqref{Bessel2} is similar to \eqref{Bessel}, using that \(\nabla_\xi K\in L^p_\xi\) for all \(1\leq p<2\).

\smallskip
\noindent
\emph{iii)} We recall that \(\nabla_\xi(-\Delta_\xi)^{-1}f = k*f\), where the kernel
\[
k(\xi) = -\frac1{2\pi}\frac{\xi}{|\xi|^2}.
\]
As a consequence, we may apply the Hardy-Littlewood-Sobolev inequality to obtain
\[
\|\nabla_\xi(-\Delta_\xi)^{-1}f\|_{L^r_\xi}\lesssim \|f\|_{L^{\frac{2r}{2 + r}}},
\]
and then use the embedding \(L^2_\xi(m)\subset L^{\frac{2r}{2 + r}}\).
\end{proof}

Applying Lemma~\ref{lem:Nonlocal} with \(\lambda = e^{\frac\tau2}|\zeta|\), we obtain the following estimates for the perturbative term:
\begin{cor}\label{cor:Pert}
For all \(m> 1\) and \(0<\delta\leq \frac12\) we have the estimate,
\eq{Pert}{
\|Z\|_{L^2_\xi(m)} \lesssim_{\delta,m} |\zeta|^{1 - 2\delta}e^{(\frac12-\delta)\tau}\|w\|_{L^2_\xi(m)}.
}
\end{cor}
\begin{proof}
We first observe that when \(\zeta = 0\) we have \(Z = 0\) so it suffices to consider the case \(\zeta\neq0\). Applying the estimates \eqref{Bessel}, \eqref{Bessel2} with \(r = 2\) and \(\lambda = e^{\frac\tau2}|\zeta|\) we then obtain the estimate
\[
\|\widehat{U^\xi} \|_{L^2_\xi}\lesssim |\zeta|^{-2\delta}e^{-\delta\tau}\|w\|_{L^2_\xi(m)}.
\]
The estimate for the \(\xi\)-component, \(Z^\xi\) then follows from the rapid decay of $G$.

To bound \(Z^z\) we first note that \(ie^{\frac\tau2}\zeta G\nabla_\xi^\perp\cdot (e^\tau |\zeta|^2 - \Delta_\xi)^{-1}w^\xi\) may be bounded similarly to \(Z^\xi\). Next we bound \(ie^{\frac12\tau}\zeta (\nabla_\xi^\perp G)\cdot (e^\tau |\zeta|^2 - \Delta_\xi)^{-1}w^\xi\) by applying the estimate \eqref{Bessel} with \(r = \infty\) and again using the rapid decay of \(\nabla_\xi G\).

For the remaining term we apply the resolvent identity to obtain
\[
\nabla_\xi^\perp(e^\tau|\zeta|^2 - \Delta_\xi)^{-1}w^z - \nabla^\perp_\xi(-\Delta_\xi)^{-1}w^z = e^\tau|\zeta|^2(e^\tau|\zeta|^2 - \Delta_\xi)^{-1}\nabla_\xi^\perp \Delta_\xi^{-1}w^z.
\]
Proceeding as in the proof of \eqref{Bessel} we may then bound,
\begin{align*}
\|\nabla_\xi^\perp(e^\tau|\zeta|^2 - \Delta_\xi)^{-1}w^z - \nabla^\perp_\xi(-\Delta_\xi)^{-1}w^z\|_{L^\infty_\xi} &\lesssim |\zeta|^{1 - 2\delta}e^{(\frac12-\delta)\tau}\|\nabla_\xi^\perp\Delta_\xi^{-1} w^z\|_{L^{\frac2{1-2\delta}}_\xi}\\
&\lesssim |\zeta|^{1 - 2\delta}e^{(\frac12-\delta)\tau}\|w\|_{L^2_\xi(m)},
\end{align*}
where the second inequality follows from the estimate \eqref{Bessel3}. The estimate then follows from the rapid decay of \(\nabla_\xi G\).
\end{proof}

Using these estimates we may complete the proof of Proposition~\ref{prop:BasicAP}:
\begin{proof}[Proof of Proposition~\ref{prop:BasicAP}]
From Proposition~\ref{prop:SemigroupL2m} we may define the semigroup
\[
\cT(\tau) = \begin{bmatrix}e^{\tau(\cL - \alpha \Gamma)}&0\\0&e^{\tau(\cL - \alpha\Lambda)}\end{bmatrix},
\]
on \(L^2_\xi(m)\) and from the estimate \eqref{SemigroupL2m} we have for any $\gamma>0$
\[
\|\cT(\tau)\|_{L^2_\xi(m)\rightarrow L^2_\xi(m)}\lesssim_\gamma e^{\gamma \tau}.
\]
Using the Duhamel formula we then write mild solutions of \eqref{FLinear} in the form,
\eq{Duhamel}{
w(\tau) = e^{ - |\zeta|^2(e^\tau - e^\sigma)}\cT(\tau - \sigma)w(\sigma) + \alpha\int_\sigma^\tau e^{ - |\zeta|^2(e^\tau - e^s)}\cT(\tau - s)Z(w(s))\,ds.
}
We now take \(\epsilon>0\) and define the map
\[
\mathscr T\colon \mc C([\sigma,\sigma+\epsilon];L^2_\xi(m))\rightarrow \mc C([\sigma,\sigma + \epsilon];L^2_\xi(m)),
\]
by
\[
\mathscr T(w) = \alpha\int_\sigma^\tau e^{ - |\zeta|^2(e^\tau - e^s)}\cT(\tau - s)Z(w(s))\,ds.
\]
Taking \(\delta = \frac14\) (say) and applying the estimate \eqref{Pert} for the perturbative term \(Z\), we may then bound,
\begin{align*}
\left\|\mathscr T(w)\right\|_{L^2_\xi(m)} &\lesssim \int_\sigma^\tau \frac{e^{(\gamma-\frac14)(\tau - s)}}{a(\tau - s)^{\frac14}}\|w(s)\|_{L^2_\xi(m)}\,ds \lesssim (\tau - \sigma)^{\frac34} \sup\limits_{\tau\in[\sigma,\sigma+\epsilon]}\|w(\tau)\|_{L^2_\xi(m)},
\end{align*}
where we recall that $a(\tau)=1-e^{-\tau}$.
In particular, by choosing \(0<\epsilon = \epsilon(\alpha) \ll 1\) sufficiently small (independently of \(\sigma\)) we may use the contraction principle to find a unique mild solution of \eqref{FLinear} on the time interval \([\sigma,\sigma + \epsilon]\). Further, as \(\epsilon\) is independent of \(\sigma\), we may iterate this argument to obtain a global solution.

To obtain the a priori estimate \eqref{BasicAP} we use an identical argument to obtain the estimate,
\[
e^{e^\tau |\zeta|^2} e^{-\gamma \tau}\|w(\tau)\|_{L^2_\xi(m)}\lesssim e^{e^\sigma |\zeta|^2} e^{-\gamma \sigma} \|w(\sigma)\|_{L^2_\xi(m)} +  \int_\sigma^\tau |\zeta|^{\frac12}e^{(\frac14-\gamma) s} e^{e^s|\zeta|^2}\|w(s)\|_{L^2_\xi(m)}\,ds.
\]
Applying the integrated form of Gronwall's inequality to the continuous non-negative function \(\tau\mapsto e^{e^\tau |\zeta|^2} e^{-\gamma \tau}\|w(\tau)\|_{L^2_\xi(m)}\) we obtain,
\[
e^{e^\tau |\zeta|^2} e^{-\gamma \tau} \|w(\tau)\|_{L^2_\xi(m)} \lesssim e^{e^\sigma |\zeta|^2} e^{-\gamma \sigma}\|w(\sigma)\|_{L^2_\xi(m)}e^{C|\zeta|^{\frac12}(e^{\frac\tau4} - e^{\frac\sigma4})},
\]
from which the estimate \eqref{BasicAP} follows.

For the improved estimate \eqref{ImprovedAP}, we first use the estimates \eqref{SemigroupL2ma},~\eqref{SemigroupL2mb} to find \(0<\mu = \mu(\alpha)<\frac12\) so that for \(m>1 + 2\mu\) we have,
\[
\|e^{\tau(\cL - \alpha\Gamma)}\|_{L^2_\xi(m)\rightarrow L^2_\xi(m)}\lesssim e^{-\mu \tau},\qquad \|e^{\tau(\cL - \alpha\Lambda)}\|_{L^2_{\xi,0}(m)\rightarrow L^2_{\xi,0}(m)}\lesssim e^{-\mu \tau},
\]
where the closed subspace \(L^2_{\xi,0}(m) = \{f\in L^2_\xi(m):\int f\,d\xi =0\}\). Next we observe that the perturbative term,
\[
Z^z = \nabla_\xi^\perp\cdot\left(i e^{\frac\tau2}  \zeta G (e^\tau|\zeta|^2 - \Delta_\xi)^{-1}w^\xi \right) + \nabla_\xi\cdot\left(G\left( (e^\tau|\zeta|^2 - \Delta_\xi)^{-1} - (-\Delta_\xi)^{-1}\right)\nabla_\xi^\perp w^z\right),
\]
so from the estimate \eqref{Pert}, we see that \(Z^z\in L^2_{\xi,0}(m)\). As a consequence, provided \(\int w^z_\sigma\,d\xi = 0\), we may estimate as above to obtain,
\[
e^{\mu\tau + e^\tau|\zeta|^2}\|w(\tau)\|_{L^2_\xi(m)}\lesssim e^{\mu\sigma + e^\sigma|\zeta|^2}\|w(\sigma)\|_{L^2_\xi(m)} + \int_\sigma^\tau |\zeta|^{\frac12}e^{\frac s4} e^{\mu s + e^s|\zeta|^2}\|w(s)\|_{L^2_\xi(m)}\,ds,
\]
and applying the integrated form of Gronwall's inequality to the continuous non-negative function \(\tau\mapsto e^{\mu\tau + e^\tau|\zeta|^2}\|w(\tau)\|_{L^2_\xi(m)}\) we have,
\[
e^{\mu\tau + e^\tau|\zeta|^2}\|w(\tau)\|_{L^2_\xi(m)}\lesssim e^{\mu\sigma + e^\sigma|\zeta|^2}\|w(\sigma)\|_{L^2_\xi(m)}e^{C |\zeta|^{\frac12}(e^{\frac\tau4} - e^{\frac\sigma4})},
\]
from which we obtain the estimate \eqref{ImprovedAP}.
\end{proof}

\subsection{Short time estimates}~
In order to prove the estimates \eqref{SmoothedBound}, \eqref{SmoothedBound1} we will combine the long time estimates \eqref{BasicBound},~\eqref{MainBound} with several short time smoothing estimates. We start with the following estimate that we prove similarly to~\cite[Proposition~4.6]{MR2178064}:
\begin{lem}\label{lem:shorttime1}
Let \(1 < p\leq 2\). Then there exists some \(0<\delta = \delta(\alpha)\ll1\) so that for all \(\sigma\leq \tau\leq \sigma + \delta\) and any \(3\times 3\) tensor field \(F\) satisfying \(\oDiv\,\oDiv F = 0\) we have the estimate
\eq{ShortTime1}{
\|S(\tau,\sigma)\oDiv F\|_{B_zL^2_\xi(m)}\lesssim \frac{1}{(\tau - \sigma)^{\frac1p}}\|F\|_{B_zL^p_\xi(m)}.
}

Further, for \(\tau>\sigma\) there exists a bounded operator \(R(\tau,\sigma)\) on \(B_zL^p_\xi(m)\) so that \(S(\tau,\sigma)\oDiv F = \oDiv R(\tau,\sigma) F\) and we have the estimates
\begin{align}
\|R(\tau,\sigma)F\|_{B_zL^2_\xi(m)} &\lesssim \frac{1}{(\tau - \sigma)^{\frac1p - \frac12}}\|F\|_{B_zL^p_\xi(m)},\label{ShortTime3}\\
\|\onabla R(\tau,\sigma)F\|_{B_zL^2_\xi(m)} & \lesssim \frac{1}{(\tau - \sigma)^{\frac1p}}\|F\|_{B_zL^p_\xi(m)}.\label{ShortTime3a}
\end{align}
\end{lem}

\begin{proof}
Start with the following equation (which can be thought of, formally, as the result of applying \(\oDiv^{-1}\) to the equation \eqref{Linear})
\eq{Div-1SS}{
\pde{
\partial_\tau F - \left(\cL + e^\tau\partial_z^2 - \frac12\right) F = \RHS(F),
}{
\oDiv\,\oDiv F = 0,
}
}
where, for \(\Omega = \oDiv F\) and \(U = (- \oDelta)^{-1}\onabla\times\oDiv F\),
\[
\RHS(F) = - \alpha \begin{bmatrix}g\\0\end{bmatrix}\otimes \Omega - \alpha U \otimes\begin{bmatrix}0\\G\end{bmatrix} + \alpha \begin{bmatrix}0\\G\end{bmatrix}\otimes U + \alpha \Omega\otimes \begin{bmatrix}g\\0\end{bmatrix}.
\]
We then take \(R(\tau,\sigma)\) to be the solution operator for the equation \eqref{Div-1SS}, which (formally) satisfies \(S(\tau,\sigma)\oDiv = \oDiv R(\tau,\sigma)\). 

The solution of \eqref{Div-1SS} may be written using the Duhamel formula as
\[
F(\tau) = e^{(\tau - \sigma)\left(\cL - \frac12\right) + e^\tau a(\tau - \sigma)\partial_z^2}F(\sigma) + \int_\sigma^\tau e^{(\tau - s)\left(\cL - \frac12\right) + e^\tau a(\tau - s)\partial_z^2}\RHS(F(s))\,ds.
\]
Our strategy will now be to apply Banach's fixed point theorem to the mapping
\[
F\mapsto e^{(\tau - \sigma)\left(\cL - \frac12\right) + e^\tau a(\tau - \sigma)\partial_z^2}F(\sigma) + \int_\sigma^\tau e^{(\tau - s)\left(\cL - \frac12\right) + e^\tau a(\tau - s)\partial_z^2}\RHS(F(s))\,ds
\]
in the closed subspace \(X\subset \mc C([\sigma,\sigma+\delta];B_zL^p_\xi(m))\) with finite norm
\[
\|F\|_X = \sup\limits_{\tau\in[\sigma,\sigma + \delta]} \left(\|F(\tau)\|_{B_zL^p_\xi(m)} + (\tau - \sigma)^{\frac1p - \frac12}\|F(\tau)\|_{B_zL^2_\xi(m)} + (\tau - \sigma)^{\frac1p}\|\onabla F(\tau)\|_{B_zL^2_\xi(m)}\right),
\]
where \(\delta>0\) will be chosen sufficiently small (independently of \(\sigma\)).

Noting, on the one hand, that the operator norm on any $L^p$ space of $(e^{\frac\tau2} \partial_z)^\beta e^{e^\tau a(\tau-\sigma) \partial_z^2}$ is $\lesssim a(\tau-\sigma)^{-\frac\beta2}$ and, on the other hand, applying~\eqref{estimatesFP}, for \(\beta\in \N^3\) and \(1\leq p\leq q\leq \infty\) we obtain
\eq{FPSmoothed}{
\|\onabla^\beta e^{(\tau - \sigma)\cL + e^\tau a(\tau - \sigma)\partial_z^2}f\|_{B_zL_\xi^q(m)} \lesssim \frac{1}{(\tau - \sigma)^{\frac1p - \frac1q + \frac{|\beta|}2}}\|f\|_{B_zL_\xi^p(m)}.
}
(Recall that $|\tau-\sigma| \leq \delta\ll1 $ hence $a(\tau-\sigma) \approx \tau-\sigma$.) In particular,
\[
\|e^{(\tau - \sigma)\left(\cL - \frac12\right) + e^\tau a(\tau - \sigma)\partial_z^2}F(\sigma)\|_{X} \lesssim \|F(\sigma)\|_{L^p_\xi(m)}.
\]

As a consequence, it remains to show that the map
\[
\mathscr T\colon F \mapsto \int_\sigma^\tau e^{(\tau - s)\left(\cL - \frac12\right) + e^\tau a(\tau - s)\partial_z^2}\RHS(F(s))\,ds
\]
is a contraction on \(X\).

We recall from Lemma~\ref{lem:BSx} that
\[
\|U\|_{B_zL^4_\xi}\lesssim \|\Omega\|_{B_zL^{4/3}_\xi}\lesssim\|\Omega\|_{B_zL^2_\xi(m)}.
\]
Observing that, for \(1<p\leq 2\),
\[
\|g\|_{L^{\frac{2p}{2 - p}}_\xi} + \|G\|_{L^{\frac{4p}{4 - p}}_\xi(m)}\lesssim 1,
\]
we may then apply H\"older's inequality to obtain the estimate,
\eq{RHSF}{
\|\RHS(F)\|_{B_z L^p_\xi(m)}\lesssim \left(\|U\|_{B_z L^4_\xi} + \|\Omega\|_{B_z L^2_\xi(m)}\right)\lesssim \|\Omega\|_{B_z L^2_\xi(m)},
}
where we note that the implicit constant depends on \(|\alpha|\).

Applying the estimate \eqref{FPSmoothed} for the linear propagator, followed by the estimate \eqref{RHSF} for \(\RHS(F)\) we then obtain
\begin{align*}
\|\mathscr T(F)\|_{B_zL^p_\xi(m)} &\lesssim  \int_\sigma^\tau \|\RHS(F(s))\|_{B_zL^p_\xi(m)}\,ds\\
&\lesssim  \|F\|_X\int_\sigma^\tau \frac{1}{(s - \sigma)^{\frac1p}}\,ds\\
&\lesssim (\tau - \sigma)^{1-\frac{1}{p}}\|F\|_X,
\end{align*}
Similarly, we have
\begin{align*}
(\tau - \sigma)^{\frac1p}\| \onabla \mathscr T(F)\|_{B_zL^2_\xi(m)}
&\lesssim (\tau - \sigma)^{\frac1p}\int_\sigma^\tau \frac{1}{(\tau - s)^{\frac1p}}\|\RHS(F(s))\|_{B_zL^p_\xi(m)}\,ds\\
&\lesssim  (\tau - \sigma)^{\frac1p}\|F\|_X\int_\sigma^\tau \frac{1}{(\tau - s)^{\frac1p}(s - \sigma)^{\frac1p}}\,ds\\
&\lesssim (\tau - \sigma)^{1 - \frac1p}\|F\|_X,
\end{align*}
and an essentially identical estimate yields
\[
(\tau - \sigma)^{\frac1p - \frac12}\|\mathscr T(F)\|_{B_zL^2_\xi(m)}\lesssim (\tau - \sigma)^{1 - \frac1p}\|F\|_X.
\]
Combining these estimates we obtain
\[
\|\mathscr T(F)\|_X\lesssim \delta^{1 - \frac1p}\|F\|_X,
\]
so we may choose \(0<\delta = \delta(\alpha)\ll1\) sufficiently small (independently of \(\sigma\)) to ensure that \(\mathscr T\) is a contraction on \(X\). The estimates \eqref{ShortTime1}, \eqref{ShortTime3}, \eqref{ShortTime3a} are then a consequence of the bounds for the solution \(F\).
\end{proof}

An essentially identical argument applied directly to the equation \eqref{Linear} then yields our second short time smoothing estimate:
\begin{lem}\label{lem:shorttime2}
There exists \(0<\delta = \delta(\alpha)\ll1\) so that for all $\sigma \leq \tau \leq \sigma+\delta$ we have the estimate,
\eq{ShortTime2}{
\| \onabla S(\tau,\sigma) F \|_{B_zL^2_\xi(m)} \lesssim \frac{1}{(\tau - \sigma)^{\frac12}}\| F \|_{B_zL^2_\xi(m)}.
}
\end{lem}
\begin{proof}

We first write the equation \eqref{Linear} in the form,
\eq{LinearRepeated}{
\partial_\tau \Omega - (\mathcal{L} + e^\tau \partial_z^2) \Omega = \RHS(\Omega),
}
where
$$
\RHS(\Omega) = - \alpha g \cdot \nabla_\xi \Omega + \alpha e^{\frac\tau2 } G \partial_z U + \alpha \begin{bmatrix} \Omega^\xi \cdot \nabla_\xi g \\ - U^\xi \cdot \nabla_\xi G \end{bmatrix}.
$$

Following a similar argument to Lemma~\ref{lem:shorttime1} we will solve this by applying Banach's fixed point theorem to the mapping,
$$
\Omega \mapsto e^{(\tau-\sigma) \mathcal{L} + e^\tau a(\tau-\sigma) \partial_z^2} \Omega(s) + \int_\sigma^\tau e^{(\tau-\sigma) \mathcal{L} + e^\tau a(\tau-s)\partial_z^2} \RHS(\Omega(s))\,ds,
$$
in the closed subspace $X\subset \mc C([\sigma,\sigma + \delta];B_zL^2_\xi(m))$ with finite norm,
$$
\| \Omega \|_X = \sup_{\tau\in[\sigma,\sigma+ \delta]}\left(\| \Omega(\tau) \|_{B_z L^2_\xi(m)} + a(\tau-\sigma)^{\frac12} \| \onabla \Omega(\tau) \|_{B_z L^2_\xi(m)} \right).
$$
Applying the estimate \eqref{FPSmoothed} we see that
$$
\| e^{(\tau-\sigma) \mathcal{L} + e^\tau a(\tau-\sigma) \partial_z^2} F \|_X \lesssim \| F \|_{L^2_\xi(m)},
$$
So again matters reduce to proving that the map
$$
\mathscr{T}\colon\Omega\mapsto\int_\sigma^\tau e^{(\tau-s) \mathcal{L} + e^\tau a(\tau-s)\partial_z^2} \RHS(\Omega(s))\,ds
$$
is a contraction on $X$ for $\delta>0$ chosen sufficiently small.

To prove this we first notice that, by H\"older's inequality,
\begin{align*}
& \| g \cdot \nabla_\xi \Omega(s) \|_{B_z L^2_\xi(m)} \lesssim \| \onabla \Omega(s) \|_{B_z L^2_\xi(m)} \lesssim (s-\sigma)^{-\frac12} \| \Omega \|_X ,\\
& \| \Omega^\xi(s) \cdot \nabla_\xi g \|_{B_z L^2_\xi(m)} \lesssim \| \Omega(s) \|_{B_z L^2_\xi(m)} \lesssim \| \Omega \|_X,
\end{align*}
and applying Lemma~\ref{lem:BSx} we may similarly bound
\begin{align*}
& \| G e^{\frac s2} \partial_z U(s) \|_{B_z L^2_\xi(m)} \lesssim \| e^{\frac s2} \partial_z U(s) \|_{B_z L^4_\xi} \lesssim \| e^{\frac s2} \partial_z \Omega(s) \|_{B_z L^2_\xi(m)} \lesssim (s-\sigma)^{-\frac12} \| \Omega \|_X, \\
& \| U^\xi(s) \cdot \nabla_\xi G\|_{B_z L^2_\xi(m)} \lesssim \| U(s) \|_{B_z L^4_\xi} \lesssim \| \Omega(s) \|_{B_z L^2_\xi(m)} \lesssim  \| \Omega \|_X.
\end{align*}
We just proved that
$$
\| \mr{RHS}(\Omega(s)) \|_{B_z L^2_\xi(m)} \lesssim (s-\sigma)^{-\frac12} \| \Omega \|_{X}.
$$
Thus, we may apply the estimate \eqref{FPSmoothed} for the operator \(e^{(\tau - \sigma)\cL + e^\tau a(\tau - \sigma)\partial_z^2}\) to obtain
\begin{align*}
(\tau-\sigma)^{\frac12} \| \onabla \mathscr{T}(\Omega) \|_{B_z L^2_\xi(m)} & \lesssim \int_\sigma^\tau \frac{(\tau-\sigma)^{\frac12}}{(\tau-s)^{\frac12}} \| \RHS(\Omega(s)) \|_{B_z L^2_\xi(m)} \,ds \\
& \lesssim \| \Omega \|_X \int_\sigma^\tau \frac{(\tau-\sigma)^{\frac12}}{(\tau-s)^{\frac12}(s-\sigma)^{\frac12}} \,ds\\
& \lesssim (\tau- \sigma)^{\frac12} \| \Omega \|_X, 
\end{align*}
and similarly,
$$
\| \mathscr{T}(\Omega) \|_{B_z L^2_\xi(m)} \lesssim (\tau- \sigma)^{\frac12}  \| \Omega \|_X.
$$
Overall, we find that
$$
\| \mathscr{T}(\Omega) \|_{X} \lesssim \delta^{\frac12} \| \Omega \|_X,
$$
so choosing $\delta$ small enough $\mathscr{T}$ is a contraction on \(X\), from which the desired estimate follows.
\end{proof}

\subsection{Proof of Theorem~\ref{prop:LinearEstimates}}
From Proposition~\ref{prop:BasicAP} we know that the solution operator \(S(\tau,\sigma)\) is well defined and satisfies the estimates \eqref{BasicBound}, \eqref{MainBound}. Thus it remains to prove the estimates \eqref{SmoothedBound}, \eqref{SmoothedBound1}.

We first take \(\delta>0\) to be the minimum of the \(\delta\)'s from lemmas~\ref{lem:shorttime1} and~\ref{lem:shorttime2}.

\smallskip
\noindent\underline{\emph{Case 1:} \(\sigma\leq \tau\leq \sigma + \delta\).} In this regime, $a(\tau-\sigma) \approx \tau - \sigma$; here the estimate \eqref{SmoothedBound} follows directly from the estimate \eqref{ShortTime1}. For the estimate \eqref{SmoothedBound1} we take \(\eta = \frac12(\tau + \sigma)\in (\sigma,\tau)\) and then apply the estimate \eqref{ShortTime2} on the interval \([\eta,\tau]\) and the estimate \eqref{ShortTime1} on the interval \([\sigma,\eta]\) to obtain
\begin{align*}
\|\onabla S(\tau,\sigma)\oDiv F\|_{B_zL^2_\xi(m)} &\lesssim \|\onabla S(\tau,\eta)S(\eta,\sigma)\oDiv F\|_{B_zL^2_\xi(m)}\\
&\lesssim a(\tau - \eta)^{-\frac12}\|S(\eta,\sigma)\oDiv F\|_{B_zL^2_\xi(m)}\\
&\lesssim a(\tau - \eta)^{-\frac12}a(\eta - \sigma)^{-\frac1p}\|F\|_{B_zL^p_\xi(m)},
\end{align*}
and the estimate then follows since \(a(\tau - \eta) \approx  a(\eta - \sigma) \approx a(\tau - \sigma)\).

\smallskip
\noindent\underline{\emph{Case 2:}  \(\tau>\sigma + \delta\).} We first note that in this case \(a(\tau - \sigma)\approx 1\). Next we show that the estimate \eqref{SmoothedBound1} follows from the estimate \eqref{SmoothedBound}. Indeed, if we assume \eqref{SmoothedBound} is true, we have
\begin{align*}
\|\onabla S(\tau,\sigma)\oDiv F\|_{B_zL^2_\xi(m)} &= \|\onabla S(\tau,\tau - \tfrac\delta2)S(\tau - \tfrac\delta 2,\sigma)\oDiv F\|_{B_zL^2_\xi(m)}\\
&\lesssim \|S(\tau - \tfrac\delta 2,\sigma)\oDiv F\|_{B_zL^2_\xi(m)}\\
&\lesssim e^{-\mu(\tau - \sigma)}\|F\|_{B_zL^p_\xi(m)}.
\end{align*}

It remains to prove the estimate \eqref{SmoothedBound}. Here we first apply Lemma~\ref{lem:shorttime1} on the time interval \([\sigma,\sigma + \tfrac\delta2]\) and, writing \(S(\sigma + \tfrac\delta 2,\sigma)\oDiv F = \oDiv R(\sigma + \tfrac\delta 2,\sigma) F\),
we may decompose
\[
S(\sigma + \tfrac\delta 2,\sigma)\oDiv F = h_1 + e^{\frac \sigma2}\partial_zh_2,
\]
where \(h_1,h_2\) are vector fields satisfying the estimates
\[
\|h_1\|_{B_zL^2_\xi(m)} + \|h_2\|_{B_zL^2_\xi(m)}\lesssim \|F\|_{B_zL^p_\xi(m)},
\]
and \(\int h_1\,d\xi = 0\).

For \(h_1\) we apply the long time estimate \eqref{MainBound} to obtain
\[
\|S(\tau,\sigma + \tfrac\delta 2)h_1\|_{B_zL^2_\xi(m)}\lesssim e^{-\mu(\tau - \sigma)}\|h_1\|_{B_zL^2_\xi(m)}\lesssim  e^{-\mu(\tau - \sigma)}\|F\|_{B_zL^p_\xi(m)}.
\]
For \(h_2\) we instead apply the long time estimate \eqref{BasicBound} with the short time estimate \eqref{ShortTime2} and the fact that \(\partial_z\) commutes with \(S(\tau,\sigma)\) to obtain
\begin{align*}
\|S(\tau,\sigma + \tfrac\delta 2)e^{\sigma/2}\partial_zh_2\|_{B_zL^2_\xi(m)} &= e^{\sigma/2}\|\partial_zS(\tau,\tau - \tfrac\delta2)S(\tau - \tfrac\delta 2,\sigma + \tfrac\delta 2)h_2\|_{B_zL^2_\xi(m)}\\
&\lesssim e^{-\frac12(\tau - \sigma)}\|S(\tau - \tfrac\delta 2,\sigma + \tfrac\delta 2)h_2\|_{B_zL^2_\xi(m)}\\
&\lesssim e^{(\gamma-\frac12)(\tau - \sigma)}\|h_2\|_{B_zL^2_\xi(m)}\\
&\lesssim e^{-\mu(\tau - \sigma)}\|F\|_{B_zL^p_\xi(m)}.
\end{align*}
\qed

\section{Linear estimates for advection-diffusion-stretching by the Oseen vortex} \label{sec:StretchLin}

\subsection{Statement of the estimates}
In this section, we consider the linear equation,
\eq{SmoothTD}{
\pdet{
\partial_t\omega + \alpha[ u^g \cdot\nabla \omega - \omega\cdot\nabla u^g] = \Delta \omega,
}{
\nabla\cdot \omega = 0,
}{
\omega(s) = \omega_s,
}
}
where \(0\leq s<T\) and the velocity \(u^g = \begin{bmatrix}\frac 1{\sqrt t}g(\frac x{\sqrt t})\\0\end{bmatrix}\).

If \(\omega_s\in B_zL^1_x\) satisfies \(\nabla\cdot \omega_s = 0\) in the sense of distributions, we say that \(\omega\in \mc C_w([s,T];B_zL^1_x)\) is a mild solution of \eqref{SmoothTD} if for all \(t\in(s,T]\) we have
\eq{TDDuhamel}{
\omega(t) = e^{(t-s)\Delta}\omega_s - \alpha \int_s^t e^{(t - \sigma)\Delta}\Div\Bigl(u^g(\sigma)\otimes \omega(\sigma) - \omega(\sigma)\otimes u^g(\sigma)\Bigr)\,d\sigma,
}
and \(\langle \omega(t)\,,\, \phi \rangle \rightarrow \<\omega_s\,,\,\phi\>\) as \(t\searrow s\) for all test functions $\phi \in \Schwartz$. We note that the expression \eqref{TDDuhamel} converges in $B_zL^1_x$ since $\| u^g(t) \|_{L^\infty} \lesssim \frac 1 {\sqrt t}$ and
\[
\|e^{t\Delta}\|_{B_zL^1_x\rightarrow B_zL^1_x}\lesssim 1,\qquad \|e^{t\Delta}\Div\|_{B_zL^1_x\rightarrow B_zL^1_x}\lesssim t^{-\frac12}.
\]
We also note that the divergence-free condition is preserved by the flow, i.e. \(\nabla \cdot\omega(t) = 0\) (in the classical sense) for all \(t\in(s,T]\). Finally, as \(u^g\) is smooth on \((0,T]\), every mild solution of \eqref{TDDuhamel} must satisfy \(\omega\in \mc C([s,T]\cap (0,T];B_zL^1_x)\).

We first consider the case that \(0<s<T\), where the velocity field \(u^g\) is smooth. Here we have the following modification of~\cite[Proposition~4.3]{MR2178064}, which we prove in Section~\ref{sect:LinearSmoothing}:

\begin{prop}\label{prop:LinearSmoothing}
Let \(0<s<t\leq T\). Given \(\omega_s\in B_zL^1_x\) satisfying \(\nabla\cdot\omega_s = 0\) in the sense of distributions there exists a unique mild solution \(\omega \in \mc C([s,T];B_zL^1_x)\) of the equation \eqref{SmoothTD}.

Taking \(\mathbb{S}(t,s)\) to be the corresponding solution operator, for any \(1\leq q\leq \infty\) and \(\gamma>0\) we have the estimates,
\begin{align}
\|\bS(t,s)\omega_s\|_{B_zL^q_x} &\lesssim (t - s)^{-\left(1 - \frac1q\right)}\left(\frac ts\right)^\gamma \|\omega_s\|_{B_zL^1_x},\label{LSmooth1}\\
\|\nabla \bS(t,s)\omega_s\|_{B_zL^q_x} &\lesssim (t - s)^{-\left(\frac32 - \frac1q\right)}\left(\frac ts\right)^\gamma \|\omega_s\|_{B_zL^1_x},\label{LSmooth1a}\\
\|\nabla^2 \bS(t,s)\omega_s\|_{B_zL^q_x} &\lesssim (t - s)^{-\left(2 - \frac1q\right)}\left(\frac ts\right)^\gamma \|\omega_s\|_{B_zL^1_x},\label{LSmooth1b}
\end{align}
where the implicit constants depend on \(\gamma,q,\alpha\).

Further, if \(f_s\in B_zL^1_x\) is a $3 \times 3$ tensor field satisfying \(\Div\Div f_s = 0\) in the sense of distributions, for any \(1\leq q\leq \infty\) and \(\gamma>0\) we have the estimates,
\begin{align}
\|\mathbb{S}(t,s)\Div f_s\|_{B_zL^q_x} &\lesssim (t - s)^{-\left(\frac32 - \frac1q\right)}\left(\frac ts\right)^\gamma\|f_s\|_{B_zL^1_x},\label{LSmooth2}\\
\|\nabla \bS(t,s)\Div f_s\|_{B_zL^q_x} &\lesssim (t - s)^{-\left(2 - \frac1q\right)}\left(\frac ts\right)^\gamma\|f_s\|_{B_zL^1_x},\label{LSmooth2a}
\end{align}
where the constants depend on \(\gamma,q,\alpha\).

Finally, the above estimates hold with \(B_z\) replaced by \(L^r_z\) for any \(1\leq r\leq \infty\).
\end{prop}

The main obstruction in extending the solution operator \(\mathbb S(t,s)\) to \(s = 0\) arises from the vortex stretching term, for which we have the (crude) estimate \(\|\nabla u^g\|_{L^\infty}\lesssim\frac\alpha t\). In order to improve this bound we will take advantage of the specific structure of the linear equation \eqref{SmoothTD}. First we observe that the \(z\)-component satisfies a self-contained equation
\eq{z-eqn}{
\boxed{
\partial_t\omega^z + \alpha (u^g)^x\cdot\nabla_x \omega^z = \Delta\omega^z.
}
}
Next we write
\[
(u^g)^x = Vx^\perp,\quad\text{where}\quad V = \frac1{2\pi|x|^2}\left(1 - e^{-\frac{|x|^2}{4t}}\right),
\]
and compute
\begin{equation}\label{p-def}
\nabla_x(Vx^\perp) = V J + \frac{\partial_r V}{r} x^\perp \otimes x,
\end{equation}
where \(r = |x|\), the matrix \(J = \begin{bmatrix}0&-1\\1&0\end{bmatrix}\) and the matrix $x^\perp \otimes x = \begin{bmatrix}-x_1x_2 & -x_2^2\\x_1^2&x_1x_2\end{bmatrix}$. In particular, if we define
\[
\psi = x\cdot\omega^x - 2t\partial_z\omega^z,
\]
we may write the equation for the \(x\)-component as
\eq{x-eqn}{
\boxed{
\partial_t \omega^x + \alpha (u^g)^x\cdot\nabla_x \omega^x - \alpha V J \omega^x = \Delta\omega^x + \alpha W(\psi + 2t\partial_z\omega^z),
}
}
where
\eq{WWW-def}{
W = \partial_r V \frac{x^\perp}{|x|} = -\frac1{\pi|x|^3}\left(1 - \left(1 + \frac{|x|^2}{4t}\right)e^{-\frac{|x|^2}{4t}}\right)\frac{x^\perp}{|x|}.
}
Finally, we turn to deriving the equation on $\psi$. First, use that differentiation in $z$ commutes with \eqref{z-eqn} to derive the equation satisfied by $-2t\partial_z \omega^z$. Second, dot~\eqref{x-eqn} with $x$ to derive the equation satisfied by $x\cdot \omega^x$. Adding both equations, observe that several terms cancel since \(\nabla\cdot\omega = 0\), leading to an  equation identical to \(\omega^z\)
\eq{psi-eqn}{
\boxed{
\partial_t\psi + \alpha (u^g)^x\cdot\nabla_x \psi = \Delta \psi.
}
}

Our strategy will be to first solve for \(\omega^z\) and \(\psi\), and then solve for \(\omega^x\), where the troublesome vortex stretching term now appears as an inhomogeneous term depending on \(\psi\). By scaling it is natural to try and bound \(\psi\) in the space \(B_zL^2_x\). However, such an approach does not directly yield an improved estimate as \(\|W\|_{L^2_x}\lesssim \frac1 t\). However, we observe that
\[
|W(t,x)|\lesssim \min\left\{\frac{|x|}{t^2},\frac1{|x|^3}\right\}
\]
so taking \(M\in 2^\Z\), we obtain
\[
\|W\|_{L^2_x(|x|\approx M)} \lesssim  \min\left\{\frac{M^2}{t^2},\frac1{M^2}\right\},
\]
and hence we have the estimate (recalling the notation \eqref{def:ell})
\eq{WInf2Bound}{
\|W\|_{\ell^\infty L^1([0,\infty);L^2_x)}\lesssim 1.
}

The estimate \eqref{WInf2Bound} motivates defining the closed subspace \(Y\subset C_w([s,T];B_zL^1_x)\) with finite norm
\[
\|\omega\|_Y = \|\omega\|_{L^\infty([s,T];B_zL^1_x)} + \|\psi\|_{\ell^1L^\infty([s,T]; B_zL^2_x)},
\]
and corresponding initial data space \(Y_s\subset B_zL^1_x\) with finite norm
\[
\|\omega_s\|_{Y_s} = \|\omega_s\|_{B_zL^1_x} + \|x\cdot\omega^x_s - 2 s\partial_z\omega_s^z\|_{\ell^1B_zL^2_x},
\]
where we note\footnote{A modification of the estimate \eqref{psi-p} below yields the bound
\(
\|\omega^z\|_{\ell^1L^\infty([0,T];B_zL^2_x)}\lesssim t^{-\frac12}\|\omega^z_0\|_{B_zL^1_x}
\),
which is sufficient to prove that for any \(\phi\in \Schwartz\) we have \(\<t \partial_z\omega^z(t),\phi\>\rightarrow 0\) as \(t\searrow 0\).
}
that for \(s = 0\) we have
\[
\|\omega_0\|_{Y_0} = \|\omega_0\|_{B_zL^1_x} + \|x\cdot\omega^x_0\|_{\ell^1B_zL^2_x}.
\]

We then have the following Proposition, which we prove in Section~\ref{sect:CriticalData}:
\begin{prop}\label{prop:MildBasic}
Let \(0\leq s< t\leq T\). Given \(\omega_s\in Y_s\) satisfying \(\nabla\cdot\omega_s = 0\) in the sense of distributions there exists a unique mild solution \(\omega \in Y\) of the equation \eqref{SmoothTD} and we have the estimate
\eq{YBound}{
\|\omega\|_Y\lesssim \|\omega_s\|_{Y_s},
}
where the constant depends on \(\alpha\).

Further, taking \(\mathbb S(t,s)\) to be the corresponding solution operator, for any \(1\leq q \leq \infty\) we have the estimate
\eq{SBBound}{
\|\mathbb{S} (t,s)\omega_s\|_{B_zL^q_x} \lesssim (t - s)^{-\left(1 - \frac1q\right)} \|\omega_s\|_{Y_s},
}
where the constant depends on \(q,\alpha\).
\end{prop}

Finally, we record subcritical estimates, which will be used to prove local well-posedness for (large) subcritical data (Theorem~\ref{thm:StrtFil2}).

\begin{prop} 
\label{subcriticallinear}
For any $p \in(1,\frac{4}{3}]$,
$$
\| \mathbb{S}(t,0) \omega_0 \|_{B_z L^{4/3}_x} \lesssim t^{-\left(\frac{1}{p}-\frac{3}{4}\right)} \left[ \| \omega_0 \|_{B_z L^p_x} + \| x \cdot (\omega_0)^x \|_{B_z L^{\frac{2p}{2-p}}} \right].
$$
\end{prop}

Before turning to the proof of these propositions, we briefly recall some properties of the \(2d\) scalar advection-diffusion equation,
\eq{2dTD}{
\pde{
\partial_t b + \alpha(u^g)^x\cdot\nabla_x b = \Delta_x b,
}{
b(s) = b_s.
}
}
We recall that for any \(1\leq p\leq \infty\) and \(b_s\in L^p_x\), the equation \eqref{2dTD} has a unique mild solution \(b\in \mc C_w([s,\infty);L^p_x)\) (see, e.g., \cite{MR1381974,MR916761} or \cite[Sectionhttps://www.overleaf.com/project/5d0ce4f27f43e260f1d89712 A.3]{MR3269635} for details). Moreover, the maximum principle ensures that the corresponding solution operator is sign-preserving.

From~\cite[Theorem~5]{MR1381974}, solutions of \eqref{2dTD} satisfy the estimate
\eq{CarlenLoss}{
\|b(t)\|_{L^q_x}\lesssim (t-s)^{-\left(\frac1p - \frac1q\right)}\|b_s\|_{L^p_x},
}
whenever \(1\leq p\leq q\leq \infty\). Further, again from~\cite[Theorem~5]{MR1381974} (also see~\cite{MR916761}), for any \(0<\beta<1\) we have the estimate
\eq{CarlenLoss2}{
|b(t,x)|\lesssim_\beta \frac{1}{4\pi t}\int_{\R^2}e^{-\beta\frac{|x - y|^2}{4t}}|b_0(y)|\,dy.
}

\subsection{Proof of Proposition~\ref{prop:LinearSmoothing}}\label{sect:LinearSmoothing}~
We now prove Proposition~\ref{prop:LinearSmoothing}. We first note that the existence of a unique mild solution follows from an elementary contraction mapping argument on sufficiently short time intervals. As a consequence, it will suffice to prove the estimates~\eqref{LSmooth1}--\eqref{LSmooth2a} for the solution operator \(\mathbb S(t,s)\).

For simplicity we will only present the proof of the estimates for the spaces \(B_zL^p_x\). In several places we will reduce matters to \(2d\) by writing \(\widehat \omega(t,x,\zeta) = e^{-(t - s)|\zeta|^2}b(t,x,\zeta)\) and then consider estimates for fixed \(\zeta\), where we note that for fixed \(\zeta\) the function \(b\) is well-defined. In order to replace \(B_z\) by \(L^r_z\), we argue similarly, but in physical space (variable $z$) rather than in frequency space (variable $\zeta$). Namely, we let $b$ satisfy \(\omega(t,x,z) = e^{(t - s)\partial_z^2}b(t,x,z)\), and deduce bounds for fixed $z$. The boundedness of the heat kernel (in the $z$ variable) on $L^r_z L^p_x$ concludes the argument.

In order to both state and prove our results it will be useful to recall that the equation \eqref{SmoothTD} decouples as
\[
\pde{
\partial_t\omega^x + \alpha [(u^g)^x\cdot\nabla_x\omega^x - \omega^x\cdot\nabla_x (u^g)^x ] = \Delta \omega^x,
}{
\partial_t \omega^z + \alpha (u^g)^x\cdot\nabla_x\omega^z = \Delta \omega^z,
}
\]
and hence the solution operator has a diagonal structure,
\[
\bS = \begin{bmatrix}\bS^x&0\\0&\bS^z\end{bmatrix}.
\]

As in \cite[Proposition~4.3]{MR2178064} our strategy will be to combine short time smoothing estimates with long time estimates in the space \(B_zL^1_x\). We start with the following a priori estimates that follow from \eqref{CarlenLoss}:
\begin{lem}\label{lem:CLBAD}
For any \(0<s<t\leq T\), \(1\leq q\leq \infty\) and \(\omega_s\in B_zL^1_x\) satisfying \(\nabla\cdot\omega_s = 0\) in the sense of distributions we have the estimates
\begin{align}
\|\mathbb{S}^x(t,s)\omega^x_s\|_{B_zL^q_x} &\lesssim (t - s)^{-\left(1 - \frac1q\right)}\left(\frac ts\right)^{C|\alpha|}\|\omega^x_s\|_{B_zL^1_x},\label{CLBAD1}\\
\|\mathbb{S}^z(t,s)\omega^z_s\|_{B_zL^q_x} &\lesssim (t - s)^{-\left(1 - \frac1q\right)}\|\omega^z_s\|_{B_zL^1_x}.\label{CLBAD2}
\end{align}
\end{lem}
\begin{proof}
Taking \(\widehat \omega(t) = e^{-(t - s)|\zeta|^2}b(t)\) we obtain the \(2\)-dimensional equations
\[
\pde{
\partial_t b^x + \alpha[(u^g)^x\cdot\nabla_xb^x - b^x\cdot\nabla_x (u^g)^x] = \Delta_x b^x,
}{
\partial_t b^z + \alpha (u^g)^x\cdot\nabla_x b^z = \Delta_x b^z.
}
\]

For the \(z\)-component, we apply~\eqref{CarlenLoss} to obtain
\[
\|b^z(t)\|_{L^q_x}\lesssim (t - s)^{-\left(1 - \frac1q\right)}\|b^z(s)\|_{L^1_x}.
\]
The estimate \eqref{CLBAD2} then follows from multiplying both sides by \(e^{-(t-s)|\zeta|^2}>0\) and integrating in \(\zeta\).

At least formally, the estimate \eqref{CLBAD1} for the \(x\)-component is proved similarly, by applying \eqref{CarlenLoss} to the equation satisfied by \(|b^x|\).

To make this argument rigorous, we introduce
\[
\varphi(t,x) = e^{(t - s)\Delta_x}G = \frac1{4\pi(1 + t - s)}e^{-\frac{|x|^2}{4(1 + t - s)}},
\]
and for \(\delta>0\) take the smooth approximation
\(
c = \left(\delta^2\varphi^2 + |b^x|^2\right)^{\frac12}
\)
to \(|b^x|\). Writing \(b^x = (b_1,b_2)^T\), we compute
\[
\Delta_x c = \frac1c\left(\delta^2\varphi\Delta_x\varphi + \Re(\bar b_1\Delta_x b_1 + \bar b_2\Delta_xb_2)\right) - F_1,
\]
where the error term
\begin{align*}
F_1 &= - \frac{\delta^2}{c^3}\left(|b_1|^2|\nabla_x\varphi|^2 + \varphi^2|\nabla_xb_1|^2 - 2\Re\left(\varphi \bar b_1\nabla_xb_1\cdot\nabla_x\varphi\right)\right)\\
&\quad - \frac{\delta^2}{c^3}\left(|b_2|^2|\nabla_x\varphi|^2 + \varphi^2|\nabla_xb_2|^2 - 2\Re\left(\varphi\bar b_2\nabla_x b_2\cdot\nabla_x\varphi\right)\right)\\
&\quad - \frac1{2c^3}\left(|b_1|^2|\nabla_xb_1|^2  - \Re\left(\bar b_1^2\nabla_x b_1\cdot\nabla_x b_1\right)\right)\\
&\quad - \frac1{2c^3}\left(|b_2|^2|\nabla_xb_2|^2  - \Re\left(\bar b_2^2\nabla_x b_2\cdot\nabla_x b_2\right)\right)\\
&\quad - \frac1{c^3}\left(|b_1|^2|\nabla_xb_2|^2 + |b_2|^2|\nabla_xb_1|^2 - \Re\left(\bar b_1\bar b_2\nabla_xb_1\cdot\nabla_xb_2 + b_1\bar b_2\nabla_x \bar b_1\cdot\nabla_x b_2) \right)\right),
\end{align*}
is readily seen to be non-positive. This yields the equation
\[
\partial_t c + \alpha(u^g)^x\cdot\nabla_x c = \Delta_xc + F,
\]
where the inhomogeneous term
\[
F = F_1 + F_2,
\]
and, recalling \eqref{WWW-def},
\[
F_2 =  \frac\alpha c\Re\left((\overline{ b^x}\cdot W)(x\cdot b^x)\right).
\]

Choosing \(\delta\)-independent \(C>0\) so that
\[
|F_2|\leq C\frac{|\alpha|}t c,
\]
we then have
\[
\left(\partial_t + \alpha(u^g)^x\cdot\nabla_x - \Delta_x\right)\left( \left(\frac st\right)^{C|\alpha|}c\right)\leq 0.
\]
As the solution operator of \eqref{2dTD} is sign-preserving and \(c\) is positive, we may apply \eqref{CarlenLoss} to obtain
\[
\|c(t)\|_{L^q_x}\lesssim (t - s)^{-\left(1 - \frac1q\right)}\left(\frac ts\right)^{C|\alpha|}\|c(s)\|_{L^1_x}.
\]
Sending \(\delta\to 0\) then yields the estimate
\[
\|b^x(t)\|_{L^q_x}\lesssim (t - s)^{-\left(1 - \frac1q\right)}\left(\frac ts\right)^{C|\alpha|}\|b^x(s)\|_{L^1_x}.
\]
Again multiplying by \(e^{-(t-s)|\zeta|^2}\) and integrating, we yield \eqref{CLBAD1}.
\end{proof}

In order to improve the power of \(\frac ts\) in the estimate \eqref{CLBAD1} and obtain the estimate \eqref{LSmooth1} we require the following lemma, which relies on semigroup estimates proved in Appendix~\ref{app:Semigroups} using similar arguments to \cite[Section~4]{MR2123378}:
\begin{lem}\label{lem:LongTime}
For any \(0<s<t\leq T\), \(\gamma>0\)  and \(\omega_s\in B_zL^1_x\) satisfying \(\nabla\cdot\omega_s = 0\) in the sense of distributions we have the estimate,
\eq{LongTime}{
\|\mathbb S^x(t,s)\omega^x_s\|_{B_zL^1_x}\lesssim\left(\frac ts\right)^\gamma\|\omega^x_s\|_{B_zL^1_x},
}
where the implicit constant depends on \(\alpha,\gamma\).
\end{lem}
\bpf
Taking \(\widehat \omega(t) = e^{-(t - s)\zeta^2}b(t) \) as in Lemma~\ref{lem:CLBAD} we see that it suffices to prove that
\eq{LongTimeGoal}{
\|b^x(t)\|_{L^1_x}\lesssim \left(\frac ts\right)^\gamma\|b^x(s)\|_{L^1_x}.
}
Rescaling $b^x$ to $B^\xi(\tau,\xi,\zeta) = e^\tau b^x(e^\tau,e^{\tau/2} \xi,\zeta)$, we obtain
\[
\partial_\tau B^\xi + \alpha \Gamma B^\xi = \cL B^\xi,
\]
where we recall the definition of the operator,
\[
\Gamma = g\cdot\nabla_\xi - \nabla_\xi g.
\]
From the estimate \eqref{SemigroupL1}, for any \(\gamma>0\) we have
\[
\|e^{\tau(\cL - \alpha\Gamma)}\|_{L^1_\xi\rightarrow L^1_\xi}\lesssim e^{\gamma \tau},
\]
and hence we obtain the estimate
\[
\|B^\xi(\tau)\|_{L^1_\xi}\lesssim e^{\gamma (\tau-\sigma)}\|B^\xi(\sigma)\|_{L^1_\xi}.
\]
Returning to the original variables, we obtain the estimate \eqref{LongTimeGoal}.
\epf

To prove the estimates \eqref{LSmooth1a}, \eqref{LSmooth1b} we require a short time smoothing estimate. Here we have the following lemma:
\begin{lem}\label{lem:AnotherSmoothingBound}
There exists \(0<\delta = \delta(\alpha)\ll1\) so that for all \(0<s<t\leq (1 + \delta)s\), \(1\leq q\leq\infty\) and \(\omega_s\in B_zL^q_x\) satisfying \(\nabla\cdot\omega_s = 0\) in the sense of distributions we have the estimates
\begin{align}
\|\nabla \bS(t,s)\omega_s\|_{B_zL^q_x}&\lesssim (t - s)^{-\frac12}\|\omega_s\|_{B_zL^q_x},\label{CLBAD3}\\
\|\nabla^2 \bS(t,s)\omega_s\|_{B_zL^q_x}&\lesssim (t - s)^{-1}\|\omega_s\|_{B_zL^q_x},\label{CLBAD3-a}
\end{align}
\end{lem}
\begin{proof}
We define the operator
\[
K\omega = \int_s^t e^{(t - \sigma)\Delta}\Div\left(u^g(\sigma)\otimes \omega(\sigma) - \omega(\sigma)\otimes u^g(\sigma)\right)\,d\sigma,
\]
and the norm
\[
\|\omega\|_Z = \sup\limits_{t\in[s,(1 + \delta)s]}\left(\|\omega\|_{B_zL^q_x} + (t-s)^{\frac12}\|\nabla\omega\|_{B_zL^q_x} + (t - s)\|\nabla^2 \omega\|_{B_zL^q_x}\right).
\]

Using the estimates
\[
\|\nabla^k e^{(t - \sigma)\Delta}\Div\|_{B_zL^q_x\rightarrow B_zL^q_x}\lesssim (t - \sigma)^{-\frac12-\frac k 2}\quad\text{and}\quad\|\nabla^ku^g(\sigma)\|_{L^\infty}\lesssim \sigma^{-\frac12-\frac k2},
\]
for \(t\in[s,(1 + \delta)s]\) and \(k\in\{0,1,2\}\), we may bound
\begin{align*}
&\left\|\nabla^k \int_s^{\frac{t+s}2} e^{(t - \sigma)\Delta}\Div\left(u^g(\sigma)\otimes \omega(\sigma) - \omega(\sigma)\otimes u^g(\sigma)\right)\,d\sigma\right\|_{B_zL^q_x}\\
&\qquad\lesssim \int_s^{\frac{t+s}2} (t-\sigma)^{-\frac12-\frac k 2}\|u^g(\sigma)\|_{L^\infty_x}\|\omega(\sigma)\|_{B_zL^q_x}\,d\sigma\lesssim \sqrt\delta (t-s)^{-\frac{k}2}\|\omega\|_Z.
\end{align*}
Similarly, we have
\begin{align*}
&\left\|\nabla^k \int_{\frac{t+s}2}^t e^{(t - \sigma)\Delta}\Div\left(u^g(\sigma)\otimes \omega(\sigma) - \omega(\sigma)\otimes u^g(\sigma)\right)\,d\sigma\right\|_{B_zL^q_x}\\
&\qquad\lesssim \sum_{j=0}^k\int_{\frac{t+s}2}^t (t-\sigma)^{-\frac12}\|\nabla^j u^g(\sigma)\|_{L^\infty_x}\|\nabla^{k-j}\omega(\sigma)\|_{B_zL^q_x}\,d\sigma\lesssim \sqrt\delta (t-s)^{-\frac{k}2}\|\omega\|_Z.
\end{align*}
Combining these bounds, for  \(t\in[s,(1 + \delta)s]\) we obtain
\[
\|K\omega\|_{B_zL^q_x} + (t-s)^{\frac12}\|\nabla K\omega\|_{B_zL^q_x} + (t-s)\|\nabla ^2K\omega\|_{B_zL^q_x} \lesssim\sqrt \delta\|\omega\|_Z.
\]

Choosing \(0<\delta\ll_\alpha 1\) sufficiently small we may then apply the contraction principle to obtain a mild solution of \eqref{SmoothTD} satisfying the estimate
\[
\|\omega(t)\|_{B_zL^q_x} + (t - s)^{\frac12}\|\nabla \omega(t)\|_{B_zL^q_x} + (t - s)\|\omega(t)\|_{B_zL^q_x}\lesssim\|\omega_s\|_{B_zL^q_x}.
\]
By uniqueness of mild solutions in the space \(B_zL^q_x\) (when \(s>0\)) we obtain the estimate \eqref{CLBAD3}.
\end{proof}

To prove the estimate \eqref{LSmooth2}, we follow the argument of~\cite[Proposition~4.3]{MR2178064} and first consider the equation
\eq{Auxiliary}{
\pdet{
\partial_t f + \alpha [u^g\otimes \Div f - \Div f\otimes u^g ]= \Delta f
}{
\Div \Div f = 0
}
{
f(s) = f_s,
}
}
where \(f\) is a \(3\times 3\) tensor field, which is formally obtained by applying \(\Div^{-1}\) to the equation \eqref{SmoothTD}. Given \(0<s<t\leq T\) and \(f_s\in B_zL^1_x\) satisfying \(\Div\Div f_s = 0\) in the sense of distributions, there exists a unique mild solution \(f(t) = \bK(t,s)f_s\) of the equation \eqref{Auxiliary}. Furthermore, if \(\Div f_s\in B_zL^1_x\), then \(\bS(t,s)\Div f_s = \Div \bK(t,s)f_s\).

We then have the following short time smoothing estimate that is proved similarly to Lemma~\ref{lem:AnotherSmoothingBound}:
\begin{lem}\label{lem:ShortTimeSmoothing}
There exists \(0<\delta=\delta(\alpha)\ll1\) so that, for any \(0<s<t\leq (1 + \delta)s\) and any \(3\times 3\) tensor field  \(f_s\in B_zL^1_x\) satisfying \(\Div\Div f_s = 0\) in the sense of distributions, we have the estimate,
\eq{ShortTimeSmoothing}{
\|\bS(t,s)\Div f_s\|_{B_zL^1_x}\lesssim (t - s)^{-\frac12} \|f_s\|_{B_zL^1_x}.
}
\end{lem}
\bpf
Arguing as in Lemma~\ref{lem:AnotherSmoothingBound} we may apply a contraction mapping argument to find \(\delta = \delta(\alpha)>0\) and a mild solution of the equation \eqref{Auxiliary} on the time interval \([s,(1 + \delta)s]\) satisfying the estimate
\[
\|f\|_{B_zL^1_x} + (t - s)^{\frac12}\|\Div f\|_{B_zL^1_x}\lesssim \|f_s\|_{B_zL^1_x}.
\]
Using the identity \(\bS(t,s)\Div = \Div \bK(t,s)\) we obtain the estimate \eqref{ShortTimeSmoothing}.
\epf

To complete the proof of Proposition~\ref{prop:LinearSmoothing} we require a long time estimate for the operator \(\bK(t,s)\) that we prove similarly to Lemma~\ref{lem:LongTime}, again using several semigroup estimates proved in Appendix~\ref{app:Semigroups}:
\begin{lem}
For any \(0<s<t\leq T\), \(\gamma>0\) and \(3\times3\) tensor field \(f_s\in B_zL^1_x\) satisfying \(\Div\Div f_s = 0\) in the sense of distributions we have the estimate,
\eq{LongTimeT}{
\|\bK(t,s)f_s\|_{B_zL^1_x}\lesssim\left(\frac ts\right)^\gamma\|f_s\|_{B_zL^1_x},
}
where the implicit constant depends on \(\alpha,\gamma\).
\end{lem}
\begin{rem}
We note that although the estimate \eqref{LSmooth2} can be viewed as the analogue of the estimate \eqref{SmoothedBound}, we use a slightly different approach to prove it. The reason for this is most easily explained by considering the operator \(\cL\) on \(L^1_\xi\). In this case the condition \(\int f\,d\xi = 0\) is insufficient to ensure a spectral gap and hence we must impose the stronger condition \(f = \Div h\) for some \(h\in L^1_\xi\). (Contrast to the case that we consider \(\cL\) on \(L^2_\xi(m)\) for \(m>1\) where the the condition \(\int f\,d\xi = 0\) \emph{is} sufficient to ensure a spectral gap.)
\end{rem}
\begin{proof}
We start by writing the \(3\times 3\) tensor \(f\), considered as a matrix, in the form
\[
f = \begin{bmatrix}f^{xx}&f^{xz}\\(f^{zx})^T&f^{zz}\end{bmatrix},
\]
for a \(2\times 2\) tensor \(f^{xx}\), \(2\)-vectors \(f^{xz}\), \(f^{zx}\) and a scalar \(f^{zz}\). Recalling the convention that \((\Div f)^j = \partial_if^{ij}\) we see that the vector
\[
\Div f = \begin{bmatrix}\Div_xf^{xx} + \partial_zf^{zx}\\\nabla_x\cdot f^{xz} + \partial_zf^{zz}\end{bmatrix} = \Div_x f^x + \partial_zf^z,
\]
where
\[
f^x = \begin{bmatrix}f^{xx} & f^{xz}\end{bmatrix},\qquad f^z = \begin{bmatrix}f^{zx}\\f^{zz}\end{bmatrix},
\]
are respectively thought of as a \(2\times 3\) matrix and a \(3\)-vector.

Taking the Fourier transform in \(z\) we obtain the equation
\[
\partial_t \widehat f + \alpha\left[ u^g\otimes \left(\Div_x\widehat f^x + i\zeta \widehat f^z\right) - \left(\Div_x\widehat f^x + i\zeta \widehat f^z\right) \otimes u^g\right] = \left(\Delta_x - |\zeta|^2\right)\widehat f.
\]
We may then decompose this into a system of four equations
\[
\pdef{
\partial_t \widehat f^{xx} + \alpha\left[(u^g)^x\otimes \left(\Div_x\widehat f^{xx} + i\zeta \widehat f^{zx}\right) - \left(\Div_x\widehat f^{xx} + i\zeta \widehat f^{zx}\right)\otimes (u^g)^x \right] = \left(\Delta_x - |\zeta|^2\right) \widehat f^{xx} ,
}{
\partial_t \widehat f^{xz} + \alpha \left(\nabla_x\cdot f^{xz} + i\zeta \widehat f^{zz}\right)(u^g)^x = \left(\Delta_x - |\zeta|^2\right) \widehat f^{xz},
}{
\partial_t\widehat f^{zx} - \alpha \left(\nabla_x\cdot f^{xz} + i\zeta \widehat f^{zz}\right)(u^g)^x  = \left(\Delta_x - |\zeta|^2\right) \widehat f^{zx},
}{
\partial_t\widehat f^{zz} = \left(\Delta_x - |\zeta|^2\right) \widehat f^{zz}.
}
\]

Next we switch to self-similar variables, letting
\[
F(\tau,\xi,\zeta) = e^{\frac\tau 2} \widehat f(e^\tau,e^{\frac\tau 2}\xi,\zeta)
\]
to obtain the system
\[
\pdef{
\partial_\tau F^{\xi\xi} + \alpha g\otimes \left(\Div_\xi F^{\xi\xi} + ie^{\frac \tau 2}\zeta F^{z\xi}\right) - \left(\Div_\xi F^{\xi\xi} + ie^{\frac \tau 2}\zeta F^{z\xi}\right)\otimes\alpha g = \left(\cL - \frac12 - e^\tau|\zeta|^2\right) F^{\xi\xi},
}{
\partial_\tau F^{\xi z} + \left(\nabla_\xi\cdot F^{\xi z} + ie^{\frac \tau 2}\zeta F^{zz}\right)\alpha g = \left(\cL - \frac12 - e^\tau|\zeta|^2\right) F^{\xi z},
}{
\partial_\tau F^{z\xi} - \left(\nabla_\xi\cdot F^{\xi z} + ie^{\frac \tau 2}\zeta F^{zz}\right)\alpha g = \left(\cL - \frac12 - e^\tau|\zeta|^2\right) F^{z\xi},
}{
\partial_\tau F^{zz} = \left(\cL - \frac12 - e^\tau|\zeta|^2\right) F^{zz}.
}
\]

We then consider estimates for \(F^{\xi\xi},F^{\xi z},F^{z\xi},F^{zz}\) in turn:

\smallskip\noindent\underline{The \(zz\)-component.} From the expression \eqref{etcl} for the semigroup \(e^{\tau \cL}\) we obtain the estimate
\eq{Fzz}{
\|F^{zz}(\tau)\|_{L^1_\xi}\leq e^{-\frac12(\tau - \sigma)}\|F^{zz}(\sigma)\|_{L^1_\xi}.
}

\smallskip\noindent\underline{The \(\xi z\)-component.} For a \(2\)-vector \(f\) we define the operator
\[
\Xi f = (\nabla_\xi\cdot f) g,
\]
and from the estimate \eqref{SemigroupL1}, for any \(\gamma>0\) we have
\[
\|e^{\tau(\cL - \frac12 - \alpha\Xi)}\|_{L^1_\xi\rightarrow L^1_\xi}\lesssim e^{(\gamma - \frac12)\tau}.
\]
Using the Duhamel formula and the estimate \eqref{Fzz} for \(F^{\zeta\zeta}\), we then obtain the estimate
\eq{Fxz}{
\|F^{\xi z}(\tau)\|_{L^1_\xi}\lesssim e^{(\gamma - \frac12) (\tau - \sigma)}\left(\|F^{\xi z}(\sigma)\|_{L^1_\xi} + \|F^{zz}(\sigma)\|_{L^1_\xi}\right).
}

\smallskip\noindent\underline{The \(z\xi\)-component.} To bound \(F^{z\xi}\) we may simply use the mapping properties of \(e^{\tau\cL}\) with the estimates \eqref{Fzz}, \eqref{Fxz} to obtain
\eq{Fzx}{
\|F^{z\xi}(\tau)\|_{L^1_\xi}\lesssim e^{-\frac12(\tau - \sigma)}\|F^{z\xi}(\sigma)\|_{L^1_\xi} + e^{(\gamma - \frac12) (\tau - \sigma)}\left(\|F^{\xi z}(\sigma)\|_{L^1_\xi} + \|F^{zz}(\sigma)\|_{L^1_\xi}\right).
}

\smallskip\noindent\underline{The \(\xi\xi\)-component.} For a \(2\times 2\) tensor \(f\) we define the linear operator
\[
\Pi f = g\otimes \Div_\xi f - \Div_\xi f\otimes g,
\]
and from the estimate \eqref{SemigroupL1}, for all \(\gamma>0\) we have,
\[
\|e^{\tau(\cL - \frac12 - \alpha\Pi)}\|_{L^1_\xi\rightarrow L^1_\xi}\lesssim e^{(\gamma - \frac12)\tau}.
\]
Using the Duhamel formula with the estimate \eqref{Fzx}, we may then bound
\eq{Fxx}{
\|F^{\xi\xi}(\tau)\|_{L^1_\xi}\lesssim e^{(\gamma - \frac12)(\tau - \sigma)}\left(\|F^{\xi\xi}(\sigma)\|_{L^1_\xi} + \|F^{\xi z}(\sigma)\|_{L^1_\xi} +  \|F^{z\xi}(\sigma)\|_{L^1_\xi} + \|F^{zz}(\sigma)\|_{L^1_\xi} \right).
}
The estimate \eqref{LongTimeT} then follows from the bounds \eqref{Fzz}--\eqref{Fxx}.
\end{proof}

Using these lemmas we may complete the proof of Proposition~\ref{prop:LinearSmoothing}:

\bpf[Proof of Proposition~\ref{prop:LinearSmoothing}] It remains to prove the estimates \eqref{LSmooth1}--\eqref{LSmooth2a}.

\smallskip\noindent
\underline{The estimate \eqref{LSmooth1}.} For the \(z\)-component we apply the estimate \eqref{CLBAD2}. For the \(x\)-component we apply the estimates \eqref{CLBAD1} and \eqref{LongTime} to obtain
\begin{align*}
\|\bS^x(t,s)\omega_s^x\|_{B_zL^q_x} &= \|\mathbb{S}^x(t,\tfrac12(t + s))\mathbb{S}^x(\tfrac12(t + s),s)\omega_s^x\|_{B_zL^q_x}\\
&\lesssim (t - s)^{-\left(1 - \frac1q\right)}\|\bS^x(\tfrac12(t + s),s)\omega_s^x\|_{B_zL^1_x}\\
&\lesssim (t - s)^{-\left(1 - \frac1q\right)}\left(\frac ts\right)^\gamma \|\omega_s^x\|_{B_zL^1_x}.
\end{align*}

\smallskip\noindent
\underline{The estimate \eqref{LSmooth1a}.} We first choose sufficiently small \(\delta>0\) satisfying the hypothesis of Lemma~\ref{lem:AnotherSmoothingBound}. If \(0<s<t\leq (1 + \delta)s\) we may apply the short time estimate \eqref{CLBAD3} on the interval \([\frac12(t + s),t]\) with the estimate \eqref{LSmooth1} on the interval \([s,\frac12(t + s)]\) to obtain
\begin{align*}
\|\nabla\bS(t,s)\omega_s\|_{B_zL^q_x} &= \|\nabla\bS(t,\tfrac12(t+s))\bS(\tfrac12(t+s),s)\omega_s\|_{B_zL^q_x}\\
&\lesssim (t - s)^{-\frac12}\|\bS(\tfrac12(t+s),s)\omega_s\|_{B_zL^q_x}\\
&\lesssim (t - s)^{-\left(\frac32 - \frac1q\right)}\left(\frac ts\right)^\gamma\|\omega_s\|_{B_zL^1_x}.
\end{align*}
If instead we have \((1 + \delta)s<t\leq T\) we apply the short time estimate \eqref{CLBAD3} on the interval \([\frac2{2 + \delta}t,t]\) and the long time estimate \eqref{LSmooth1} on the interval \([s,\frac2{2+\delta}t]\) (noting that \(\frac1{1 + \delta}t<\frac 2{2 + \delta} t<t\)) to obtain
\begin{align*}
\|\nabla\bS(t,s)\omega_s\|_{B_zL^q_x} &= \|\nabla\bS(t,\tfrac2{2 + \delta}t)\bS(\tfrac2{2 + \delta}t,s)\omega_s\|_{B_zL^q_x}\\
&\lesssim t^{-\frac12}\|\bS(\tfrac2{2 + \delta}t,s)\omega_s\|_{B_zL^q_x}\\
&\lesssim (t - s)^{-\left(\frac32 - \frac1q\right)}\left(\frac ts\right)^\gamma\|\omega_s\|_{B_zL^1_x},
\end{align*}
where we have used the fact that \(\frac2{2 + \delta}t - s > \frac1{2 + \delta}(t - s)\) whenever \(\delta s<t - s\).

\smallskip\noindent
\underline{The estimate \eqref{LSmooth1b}.} This is proved in an identical manner to the estimate \eqref{LSmooth1a} using the estimate \eqref{CLBAD3-a}.

\smallskip\noindent
\underline{The estimate \eqref{LSmooth2}.} By exploiting a similar strategy to the proof of \eqref{LSmooth1} it suffices to show that
\[
\|\mathbb S(t,s)\Div f_s\|_{B_zL^1_x}\lesssim (t - s)^{-\frac12}\left(\frac ts\right)^\gamma\|f_s\|_{B_zL^1_x}.
\]
We take sufficiently small \(\delta >0\) as in Lemma~\ref{lem:ShortTimeSmoothing} and if \(s<t\leq (1 + \delta)s\) we may directly apply the estimate \eqref{ShortTimeSmoothing} to obtain
\[
\|\mathbb S(t,s)\Div f_s\|_{B_zL^1_x}\lesssim (t - s)^{-\frac12}\|f_s\|_{B_zL^1_x}.
\]
On the other hand, if \((1 + \delta)s< t\leq T\) we may combine the estimate \eqref{ShortTimeSmoothing} with the estimate \eqref{LongTimeT} to obtain
\begin{align*}
\|\mathbb S(t,s)\Div f_s\|_{B_zL^1_x}&\lesssim \|\mathbb S(t,\tfrac2{2 + \delta}t)\Div \bK(\tfrac2{2 + \delta} t,s) f_s\|_{B_zL^1_x}\\
&\lesssim t^{-\frac12}\|\bK(\tfrac2{2 + \delta}t,s) f_s\|_{B_zL^1_x}\\
&\lesssim (t - s)^{-\frac12}\left(\frac ts\right)^\gamma \|f_s\|_{B_zL^1_x}.
\end{align*}

\smallskip\noindent\underline{The estimate \eqref{LSmooth2a}.} This follows from the estimate \eqref{LSmooth2} in the same way that the estimate \eqref{LSmooth1a} follows from the estimate \eqref{LSmooth1}.

\epf

\subsection{Proof of Proposition~\ref{prop:MildBasic}}\label{sect:CriticalData}
We now turn to the proof of Proposition~\ref{prop:MildBasic}. We first consider the \(z\)-component and have the following lemma:
\begin{lem}\label{lem:z-comp}
For any \(1\leq p\leq \infty\) and \(\omega_s^z\in B_z L^p_x\) there exists a unique mild solution \(\omega^z\in \mc C_w([s,T];B_zL^p_x)\) of the equation \eqref{z-eqn} satisfying \(\omega^z(s) = \omega^z_s\). Further, for all \(0\leq s<t\leq T\) and \(1\leq p\leq q\leq \infty\) we have the estimate
\eq{Oz-p}{
\|\omega^z(t)\|_{B_z L^q_x} \lesssim (t - s)^{-\left(\frac1p - \frac1q\right)}\|\omega^z_s\|_{B_z L^p_x}
}
\end{lem}
\begin{proof}
Letting \(\widehat \omega(t) = e^{-(t - s)|\zeta|^2}b^z (t)\), we have
$$
\pde{
\partial_t b^z + \alpha (u^g)^x \cdot \nabla_x b^z = \Delta_x b^z
}{
b^z(s) = \widehat{\omega^z_s},
}
$$
which is precisely the scalar advection-diffusion equation \eqref{2dTD}. The estimate~\eqref{CarlenLoss} gives
$$
\| b^z(t) \|_{L^q_x} \lesssim (t-s)^{-\left(\frac{1}{p} - \frac{1}{q}\right)} \| b^z(s) \|_{L^p_x}
$$
for fixed $\zeta$, from which the estimate \eqref{Oz-p} follows.
\end{proof}

Next we consider \(\psi\) and have the following lemma:
\begin{lem}\label{lem:psi-comp}
For all \(\psi_s\in \ell^1B_zL^2_x\subset B_zL^2_x\) there exists a unique mild solution \(\psi\in \mc C_w([s,T];B_zL^2_x)\) of \eqref{psi-eqn} satisfying \(\psi(s) = \psi_s\). Further, we have the estimate,
\eq{psi-p}{
\|\psi\|_{\ell^1L^\infty([s,T];B_zL^2_x)} \lesssim \|\psi_s\|_{\ell^1B_zL^2_x}.
}
\end{lem}
\begin{proof}
Proceeding as in Lemma~\ref{lem:z-comp} we take \(\widehat\psi = e^{-(t - s)|\zeta|^2}w_\zeta\). We then apply the estimate \eqref{CarlenLoss2} and H\"older's inequality to obtain
\[
\|\chi_Mw_\zeta(t)\|_{L^2_x} \lesssim\sum\limits_{M'\approx M}\|\chi_{M'}w_\zeta(s)\|_{L^2_x} + \sum\limits_{M'\not\approx M}\frac{MM'}{t - s}e^{-C\frac{\max\{M^2,(M')^2\}}{t - s}}\|\chi_{M'}w_\zeta(s)\|_{L^2_x}.
\]
As a consequence,
\[
\|\chi_M\psi \|_{L^\infty([s,T];B_zL^2_x)} \lesssim \sum\limits_{M'\approx M}\|\chi_{M'}\psi_s\|_{B_zL^2_x} + \sum\limits_{M'\not\approx M}\frac{MM'}{\max\{M^2,(M')^2\}}\|\chi_{M'}\psi_s\|_{B_zL^2_x}.
\]
We may then sum over \(M\in 2^\Z\) to obtain the estimate \eqref{psi-p}.
\end{proof}

Using these estimates we may solve for \(\omega^x\) and complete the proof of proposition~\ref{prop:MildBasic}:

\begin{proof}[Proof of Proposition~\ref{prop:MildBasic}]
Using Lemmas~\ref{lem:z-comp},~\ref{lem:psi-comp} we may construct \(\omega^z\) and \(\psi\). Thus, it remains to prove that \(\omega^x\) exists, is unique and satisfies the estimate
\eq{x-ap}{
\|\omega^x(t)\|_{B_zL^1_x} \lesssim \|\omega_s\|_{Y_s}
}
The estimate \eqref{SBBound} on $[s,t]$ then follows from combining the estimate \eqref{YBound} on $[s,\frac{s+t}{2}]$ with the short time smoothing estimate \eqref{LSmooth1} on $[\frac{s+t}{2},t]$.
We start by writing \(\widehat{\omega^x} = e^{-(t - s)|\zeta|^2}b\) to obtain the \(2d\)-equation,
\eq{Hom-x-2d}{
\pde{
\partial_t b + \alpha (u^g)^x\cdot\nabla_x b - \alpha V J b = \Delta_x b + \alpha Wx\cdot b,
}{
b(s) = b_s.
}
}
Proceeding as in Lemma~\ref{lem:CLBAD}, for \(\delta>0\) we take,
\[
c = \left(\delta^2\varphi^2 + |b|^2\right)^{\frac12},
\]
where \(\varphi = e^{(t-s)\Delta_x}G\) to obtain the equation
%

%
\eq{Inhomogeneous2dTD}{
\partial_t c + \alpha (u^g)^x\cdot\nabla_x c = \Delta_x c + F,
}
where \(F_1\leq 0\) and the inhomogeneous term
\[
F = F_1 + F_2,
\]
with
\[
F_2 =  \frac1c\Re\left((\bar b\cdot W)(x\cdot b)\right).
\]

We now replace \(V\) by the mollification \(V^{(\epsilon)} = \rho_\epsilon*V\), so that \(W\) is replaced by \(W^{(\epsilon)} = \partial_rV^{(\epsilon)}\frac{x^\perp}{|x|}\), and construct a corresponding mild local solution \(b^{(\epsilon)} \in C([s,T];L^1_x)\).
The corresponding quantity \(c^{(\epsilon)} = (\delta^2\varphi^2 + |b^{(\epsilon)}|^2)^{\frac12}\) is a non-negative solution of the inhomogeneous equation \eqref{Inhomogeneous2dTD} (with \(b,V,W\) replaced by \(b^{(\epsilon)},V^{(\epsilon)},W^{(\epsilon)}\) respectively) and hence we may apply the estimate \eqref{CarlenLoss} with the fact that the solution operator of \eqref{2dTD} is sign-preserving to obtain,
\[
\|c^{(\epsilon)}(t)\|_{L^1_x}\lesssim \|c_s\|_{L^1_x} + \int_s^t \|F_2(\sigma)\|_{L^1_x}\,d\sigma.
\]
Multiplying by \(e^{-(t - s)|\zeta|^2} > 0\), integrating in \(\zeta\), and applying the estimates \eqref{WInf2Bound} for \(W\), \eqref{psi-p} for \(\psi\) and \eqref{Oz-p} for \(\omega^z\), we then obtain,
\begin{align*}
\|e^{-(t - s)|\zeta|^2}c^{(\epsilon)}(t)\|_{L^1_{x,\zeta}} &\lesssim \|e^{-(t - s)|\zeta|^2}c_s\|_{L^1_{x,\zeta}} + \int_s^t\|e^{-(t - \sigma)|\zeta|^2}W^{(\epsilon)} \hat \psi(\sigma)\|_{L^1_{x,\zeta}}\,d\sigma\\
&\qquad + \int_s^t\|e^{-(t - \sigma)|\zeta|^2}W^{(\epsilon)} \sigma \zeta \widehat{ \omega^z}(\sigma) \|_{L^1_{x,\zeta}}\,d\sigma\\
&\lesssim \|e^{-(t - s)|\zeta|^2}c_s\|_{L^1_{x,\zeta}} + \|W^{(\epsilon)}\|_{\ell^\infty L^1([s,T];L^2_x)}\|\psi\|_{\ell^1L^\infty([s,T];B_zL^2_x)}\\
&\qquad + \int_s^t(t - \sigma)^{-\frac12}\|W^{(\epsilon)}\sigma\|_{L^2_x}\|\omega^z(\sigma)\|_{B_zL^2_x}\,d\sigma\\
&\lesssim \|e^{-(t - s)|\zeta|^2}c_s\|_{L^1_{x,\zeta}} + \|\omega_s\|_{Y_s}.
\end{align*}
Taking \(\delta\rightarrow0\) we arrive at the a priori estimate,
\[
\|e^{-(t - s)|\zeta|^2}b^{(\epsilon)}(t)\|_{L^1_{x,\zeta}}\lesssim \|\omega_s\|_{Y_s}.
\]

We now pass to the limit as \(\epsilon\rightarrow 0\) to obtain \(\omega^x\in L^\infty([s,T];B_zL^1_x)\) satisfying the Duhamel formula,
\[
\omega^x(t) = e^{(t - s)\Delta}\omega^x_s - \alpha \int_s^t e^{(t - \sigma)\Delta}\Div\left((u^g)^x(\sigma)\otimes \omega^x(\sigma) - \omega^x(\sigma)\otimes (u^g)^x(\sigma)\right)\,d\sigma,
\]
as well as the estimate
\[
\|\omega^x(t)\|_{B_zL^1_x}\lesssim \|\omega_s\|_{Y_s}.
\]
Weak continuity in time then follows from an identical argument to the \(2d\) equation \eqref{2dTD}, so \(\omega^x\in \mc C_w([s,T];B_zL^1_x)\) is a mild solution of \eqref{x-eqn}.

It remains to prove that the mild solution of \eqref{x-eqn} we have constructed is unique. Suppose that \(\omega\in Y\) is a mild solution of \eqref{SmoothTD} with initial data \(\omega_s = 0\). From the uniqueness statement of Lemma~\ref{lem:z-comp} we then see that \(\omega^z = 0 = \psi\). As a consequence, the problem reduces to showing that there exists a unique mild solution of the equation \eqref{x-eqn} satisfying \(x\cdot \omega^x = 0 = \nabla_x\cdot \omega^x\) with initial data \(\omega^x_s = 0\). However, if $x \cdot \omega^x = \nabla_x \cdot \omega^x = 0$, then we may write $$\omega^x(t,x,z) = f(t,|x|,z) \frac{x^\perp}{|x|}.$$
But then $f$ is a solution of the equation,
$$
\partial_t f = \left(\partial_r^2 + \frac1r\partial_r - \frac1{r^2}\right)f,
$$
where \(r = |x|\), whose solution is clearly unique.
\end{proof}

\subsection{Proof of Proposition~\ref{subcriticallinear}} Recalling that \(1<p\leq \frac43\), by~\eqref{Oz-p},~\eqref{CLBAD3}, and~\eqref{CarlenLoss}, we have, for $p \leq q$,

\begin{align*}
\| \omega^z \|_{B_z L^{4/3}_x} &\lesssim t^{\frac34-\frac1p} \| \omega^z_0 \|_{B_z L^p_x}, \\
\| \partial_z \omega^z \|_{B_z L^{q}_x} &\lesssim t^{\frac{1}{q} - \frac{1}{p}- \frac{1}{2}} \| \omega^z_0 \|_{B_z L^p_x}, \\
\| \psi \|_{B_z L^q_x} &\lesssim t^{\frac{1}{q}-\frac{1}{p}} \| x\cdot\omega_0^x \|_{B_z L^p_x}.
\end{align*}
Arguing as in the proof of Proposition~\ref{prop:MildBasic} using~\eqref{CarlenLoss} we may bound
\begin{align*}
\|e^{-t|\zeta|^2}b(t)\|_{L^1_\zeta L^{4/3}_x} &\lesssim t^{\frac34-\frac1p} \| \omega^x_0 \|_{B_z L^p_x}  +\int_0^t(t - \sigma)^{-\frac14}\|e^{-(t-\sigma)|\zeta|^2}W\widehat\psi(\sigma)\|_{L^1_{x,\zeta}}\,d\sigma\\
&\quad + \int_0^t(t - \sigma)^{-\frac14} \|e^{-(t-\sigma)|\zeta|^2}W\sigma\zeta \widehat{\omega^z}(\sigma)\|_{L^1_{x,\zeta}}\,d\sigma\\
&\lesssim t^{\frac34-\frac1p} \| \omega^x_0 \|_{B_z L^p_x}\\
&\quad + \int_0^t(t - \sigma)^{-\frac14}\|W\|_{L^{\frac{2p}{3p-2}}_x}\left(\|\psi(\sigma)\|_{B_zL^{\frac{2p}{2-p}}_x} + \sigma\|\partial_z\omega^z(\sigma)\|_{B_zL^{\frac{2p}{2-p}}_x}\right)\,d\sigma\\
&\lesssim t^{\frac34-\frac1p} \| \omega^x_0 \|_{B_z L^p_x} + \int_0^t (t-\sigma)^{-\frac{1}{4}} \sigma^{-\frac{1}{p}} \left(\| x \cdot (\omega_0)^x \|_{B_z L^{\frac{2p}{2-p}}}  +  \| \omega_0 \|_{B_z L^p_x}\right)\,d\sigma \\
& \lesssim t^{\frac34-\frac1p} \| \omega^x_0 \|_{B_z L^p_x} +  t^{\frac{3}{4}-\frac{1}{p}} \left(\| x \cdot (\omega_0)^x \|_{B_z L^{\frac{2p}{2-p}}}  + \| \omega_0 \|_{B_z L^p_x}\right),
\end{align*}
where we have used that
\[
\|W\|_{L^{\frac{2p}{3p-2}}_x}\lesssim \sigma^{-\frac1p},
\]
which gives the desired result.\qed

\section{Nonlinear estimates for the straight filament}\label{sec:Straight}

This section is devoted to the proofs of Theorems~\ref{thm:StrtFil} and~\ref{thm:StrtFil2}. The proofs of both theorems are mostly identical, with a few differences which will be made clear.

\bigskip

We start by formally splitting the vorticity
\begin{align}
\omega(t,x,z) = \underbrace{\begin{bmatrix} 0 \\ \frac{\alpha}{t}G\left(\frac{x}{\sqrt{t}}\right) \end{bmatrix} + \omega^c(t,x,z)}_{ \displaystyle \widetilde{\omega}^c(t,x,z)}  + \omega^b(t,x,z) \label{eq:StrtSplitting}
\end{align}
with corresponding velocity field given by the Biot-Savart law
$$
u(t,x,z) =  \alpha u^g(t,x,z) + u^c(t,x,z) + u^b(t,x,z) \qquad \mbox{with} \qquad u^g(t,x,z) = \begin{bmatrix}\frac{1}{\sqrt{t}}g \left( \frac{x}{\sqrt{t}} \right) \\ 0 \end{bmatrix} .
$$
As usual, we capitalize $u$ and $\omega$ in self-similar variables: for $*=b$ or $c$,
$$
\Omega^{*}(\tau,\xi,z) = t \omega^*(t,x,z) \qquad \mbox{and} \qquad U^{*} (\tau,\xi,z) = \sqrt{t} u^*(t,x,z).
$$
It remains to define $\widetilde{\omega}^c$ and $\omega^b$: they are given by
\[
\pde{
\partial_t \widetilde{\omega}^c + u \cdot \nabla \widetilde{\omega}^c - \widetilde{\omega}^c \cdot \nabla u = \Delta \widetilde{\omega}^c}{ 
\widetilde{\omega}^c(t=0) = 
\begin{bmatrix}
0 \\ 
\alpha \delta_{x = 0}
\end{bmatrix}
}
\]
and 
\[
\pde{
\partial_t \omega^b + u \cdot \nabla \omega^b - \omega^b \cdot \nabla u = \Delta \omega^b}
{\omega^b(t=0) = \mu^b.}  
\]
Recall that $\mathbb{S}$ is the semigroup associated to the problem $\partial_t \omega + \alpha[u^g \cdot \nabla \omega - \omega \cdot\nabla u^g] = \Delta \omega$, while $S$ is the semigroup associated to $\partial_t \omega + \alpha[u^g \cdot \nabla \omega - \omega \cdot \nabla u^g + u \cdot \nabla \omega^g - \omega^g \cdot \nabla u] = \Delta \omega$ in self-similar variables.

Duhamel's formula then formally gives
\begin{equation}
\label{eqPhi}
\begin{split}
\omega^b(t) & = \mathbb{S}(t,0)\mu^b - \int_0^t \mathbb{S}(t,s) \nabla \cdot \left( u^b \otimes \omega^b - \omega^b \otimes u^b \right)\, ds \\ 
& \qquad \qquad - \int_0^t \mathbb{S}(t,s) \nabla \cdot \left( u^c \otimes \omega^b - \omega^b \otimes u^c \right)\, ds, \\
\Omega^c(\tau) & = -\alpha \int_{-\infty}^\tau S(\tau,s) \overline{\nabla}\cdot \left(U^b \otimes G e_3 - G e_3 \otimes U^b\right)\,  ds \\ 
& \quad - \int_{-\infty}^\tau S(\tau,s) \overline{\nabla} \cdot \left(U^b \otimes \Omega^c - \Omega^c \otimes U^b\right)\, ds \\
& \quad - \int_{-\infty}^\tau S(\tau,s) \overline{\nabla} \cdot \left(U^c \otimes \Omega^c - \Omega^c \otimes U^c\right)\, ds.
\end{split}
\end{equation}
Writing $\mc Q$ for the above right-hand side, we are looking for a solution of the equation
$$
(\omega^b,\Omega^c) = \left( \mathcal{Q}^b (\omega^b,\Omega^c), \mathcal{Q}^c(\omega^b,\Omega^c) \right) = \mathcal{Q}(\omega^b,\Omega^c) .
$$
Taking \(m\geq 2\), we will solve this fixed point problem by applying the Banach fixed point theorem in the following ball, for constants \(M,D\geq 1\) determined by the proof below, 

\begin{align*}
B_{\eps,T} & = \Big\{ (\omega^b,\Omega^c) \; \mbox{functions on $(0,T) \times \mathbb{R}^3$ such that} \; \nabla \cdot \omega^b = 0, \; \overline{\nabla} \cdot \Omega^c = 0, \\
&  \qquad \quad \mbox{and} \; \norm{(\omega^b,\Omega^c)}_X := M \sup_{t \in (0,T)} t^{\frac14}\|\omega^b(t)\|_{B_z L^{4/3}_x} + \sup_{\tau \in (-\infty,\log T)} \norm{\Omega^c(\tau)}_{B_z L^2_\xi(m)} \leq D\eps \Big\}. 
\end{align*}

To prove Theorem~\ref{thm:StrtFil} we will verify that, whenever the data satisfies \eqref{topologydata}, the map $\mathcal{Q}:B_{\eps,T} \to B_{\eps,T}$ is a contraction for any \(0<\epsilon\leq \epsilon_0\), any \(T>0\) and a judicious of the constants $M$, $D$, and $\epsilon_0$.

\bigskip

\noindent
\underline{Bound for the core}. We abbreviate the three summands in the definition of $\mathcal{Q}^c$ in~\eqref{eqPhi} by
$$
\mathcal{Q}^c(\omega^b,\Omega^c) = L^c + N^c_0 + N^c_1.
$$
By \eqref{SmoothedBound} and the rapid decay of \(G\),
\begin{align}
\norm{L^c(\tau)}_{B_z L^2_\xi(m)} & \lesssim \int_{-\infty}^{\tau} \frac{e^{-\mu(\tau-s)}}{a(\tau-s)^{\frac12}} \norm{G\,U^b(s)}_{B_z L^2_\xi(m)}\, ds \lesssim \int_{-\infty}^\tau \frac{e^{-\mu(\tau-s)}}{a(\tau-s)^{\frac12}} \norm{U^b(s)}_{B_z L^4_\xi}\, ds. 
\end{align}
Note that by scaling and the bound on the Biot-Savart formula in Lemma~\ref{BSx}
\begin{align}
\norm{U^b(\tau)}_{B_z L^4_\xi} = e^{\frac\tau4} \norm{u^b(e^{\tau})}_{B_z L^4_x} \lesssim \sup_{0 < t < T} t^{\frac14} \norm{\omega^b(t)}_{B_z L^{4/3}_x}. 
\end{align}
Hence, with an implicit constant independent of $T$ and $\eps$,
\begin{align}
\norm{L^c(\tau)}_{B_z L^2_\xi(m)} & \lesssim \frac{1}{M} \norm{(\omega^b,\Omega^c)}_{X}. 
\end{align}
Using again \eqref{SmoothedBound}, we have 
\begin{align*}
\norm{N^c_0(\tau)}_{B_z L^2_\xi(m)} &  \lesssim \int_{-\infty}^\tau \frac{e^{-\mu(\tau-s)}}{a(\tau-s)^{\frac34}} \norm{U^b \otimes \Omega^c(s)}_{B L^{4/3}_\xi(m)}\, ds \\ 
&  \lesssim \int_{-\infty}^\tau \frac{e^{-\mu(\tau-s)}}{a(\tau-s)^{\frac34}} \|U^b(s)\|_{B_z L^4_\xi} \norm{\Omega^c(s)}_{B_z L^{2}_\xi(m)}\, ds \\ 
& \lesssim \norm{(\omega^b,\Omega^c)}_{X}^2.
\end{align*}
The same proof applies to the nonlinear term $N^c_1$ by using Lemma~\ref{BSx}, 
$$
\| U^c(\tau) \|_{B_z L^4_\xi} = e^{\frac \tau4}\norm{u^c(e^\tau)}_{B_z L^4_x} \lesssim e^{\frac\tau4}\norm{\omega^c(e^\tau)}_{B_z L^{4/3}_x} \lesssim \| \Omega^c(\tau) \|_{B_zL^2_\xi(m)}.
$$

This completes the proof of the desired estimates near the core.

\bigskip

\noindent \underline{Bound for the background}. We abbreviate the three summands in the definition of $\mathcal{Q}^b$ in~\eqref{eqPhi} by
$$
\mathcal{Q}^b(\omega^b,\Omega^c) = L^b + N^b_0 + N^b_1.
$$
By Proposition~\ref{prop:LinearSmoothing}, choosing $\gamma \in (0,\frac{1}{2})$,
\begin{align*}
t^{\frac14}\norm{N^b_1(t)}_{B_z L^{4/3}_x} & = t^{\frac14}\int_0^t \norm{\mathbb{S}(t,s) \grad \cdot (u^c \otimes \omega^b - \omega^b \otimes u^c)}_{B_z L^{4/3}_x}\, ds \\
& \lesssim t^{\frac14}\int_0^t \frac{1}{(t-s)^{\frac34}} \left( \frac{t}{s} \right)^{\gamma} \norm{u^c(s)}_{B_z L^{4}_x} \|\omega^b(s)\|_{B_z L^{4/3}_x} \,ds \\ 
& \lesssim \frac 1M\norm{(\omega^b, \Omega^c)}_X^2  t^{\frac14} \int_0^t \frac{1}{(t-s)^{\frac34}} \left( \frac{t}{s} \right)^{\gamma} \frac{1}{s^{\frac12}}\, ds \\
& \lesssim \frac1M\norm{(\omega^b, \Omega^c)}_X^2 , 
\end{align*}
where we used Lemma~\ref{BSx}, scaling, and the inclusion $L^2_\xi(m) \subset L^{\frac43}_\xi$ to obtain
$$
\norm{u^c(t)}_{B_zL^4_x} \lesssim \norm{\omega^c(t)}_{B_zL^{4/3}_x} = t^{-\frac14} \| \Omega^c (\log t) \|_{B_zL^{4/3}_\xi} \lesssim t^{-\frac14} \| \Omega^c (\log t) \|_{B_zL^{2}_\xi(m)}\lesssim t^{-\frac14} \| (\omega^b,\Omega^c) \|_{X}.
$$
The term $N^b_0$ can be dealt with similarly. 

\bigskip

\noindent \underline{Proof of Theorem \ref{thm:StrtFil}} The above estimates imply that
$$
\| \mc Q (\omega^b,\Omega^c) \|_{X} \lesssim  M \sup_{t \in (0,T)} t^{\frac14} \| L^b \|_{B_z L^{4/3}_x} + \frac{1}{M}\| (\omega^b,\Omega^c) \|_{X} + \| (\omega^b,\Omega^c) \|_{X}^2. 
$$
By Proposition~\ref{prop:MildBasic} and the hypothesis \eqref{topologydata}, this implies that for some constant \(C_0\geq 1\) we have the estimate
\begin{align*}
\| \mc Q (\omega^b,\Omega^c) \|_{X} & \leq C_0 M  \left(\norm{\mu^b}_{B_z L^1_x} + \norm{x \cdot (\mu^b)^x}_{\ell^1 B_z L^2_x} \right)+ \frac{C_0}{M}\| (\omega^b,\Omega^c) \|_{X} + C_0 \| (\omega^b,\Omega^c) \|_{X}^2 \\
& \leq  C_0 M  \epsilon + \frac{C_0}{M}\| (\omega^b,\Omega^c) \|_{X} + C_0 \| (\omega^b,\Omega^c) \|_{X}^2.
\end{align*}
In order for $\mathcal{Q}$ to map $B_{\epsilon,T}$ to itself, it suffices that
$$
C_0 M \epsilon + \frac{C_0}{M} D \epsilon + C_0  ( D \epsilon)^2 \leq \frac{D}{2} \epsilon.
$$
This can be ensured by choosing $M = 10 C_0$, $D = 10 C_0 M$, and $\epsilon_0 \leq \frac{1}{10 C_0 D}$. A similar argument shows that $\mathcal{Q}$ is also a contraction on $B_{\epsilon,T}$.

\bigskip

\noindent \underline{Proof of Theorem \ref{thm:StrtFil2}} In this case, we take \(M,D,\epsilon_0\) to be chosen as in the proof of Theorem \ref{thm:StrtFil} and show that for sufficiently small \(T>0\) the map \(\mc Q\) is a contraction on \(B_{\epsilon_0,T}\).  We learn from Proposition~\ref{subcriticallinear} that
\begin{align}
\lim_{t \searrow 0} t^{\frac14} \| L^b \|_{B_z L^{4/3}_x} = 0. \label{ineq;Lbvan}
\end{align} 
Thus, it suffices to choose $T$ sufficiently small that the proof of Theorem \ref{thm:StrtFil} above applies.

\bigskip

\noindent \underline{Mild solution of Navier-Stokes}. 
Finally, we verify that the solution constructed by the fixed point argument above does indeed satisfy Definition \ref{milddef}. 
For this we need two things: (a) $\omega \in \mc C_w ([0,T];\cM^{\frac{3}{2}})$, specifically also $\lim_{t \searrow 0} \omega(t) = \alpha \delta_{x = 0} e_3 + \mu^b$ in the sense of $\Schwartz'$; (b) that the mild form of the equations \eqref{milddefDuhamel} is satisfied. 

Proposition~\ref{prop:LinearSmoothing}, the Duhamel formula \eqref{eqPhi}, and our contraction mapping argument ensure that
\[
\sup_{t \in (0,T)}\norm{\omega^b}_{B_z L^1_x}\lesssim \norm{\mu^b}_{B_zL^1_x} + \norm{(\omega^b,\Omega^c)}_X^2\lesssim 1.
\]
The embedding $B_z L^1_x\hookrightarrow \cM^{\frac{3}{2}}$ then yields the a priori estimate $\sup_{t \in (0,T)}\norm{\omega(t)}_{\cM^{\frac{3}{2}}} \lesssim 1$ and classical parabolic regularity ensures that $\omega \in \mc C( (0,T]; \cM^{\frac{3}{2}})$. 

It remains to verify that $\lim_{t \searrow 0} \omega(t) = \alpha \delta_{x = 0} e_3 + \mu^b$ in the sense of distributions. For this, first recall the decomposition \eqref{eq:StrtSplitting}. The term involving $G$ converges to $\alpha \delta_{x = 0} e_3$ weak$*$, and hence in the sense of distributions.

Next, observe that \eqref{ineq;Lbvan} holds not only in the setting of Theorem \ref{thm:StrtFil2}, but also (by approximation as $\mu^b \in B_z L^1_x$) in the setting of Theorem \ref{thm:StrtFil}. As a consequence, for any \(\epsilon'>0\) we may choose $T$ sufficiently small to ensure that the above contraction mapping argument closes with $\epsilon$ replaced by $\epsilon'$, regardless of the size of the initial data. This suffices to show that 
\begin{align}
\lim_{t \searrow 0} \left( \sup_{0 < s < t} s^{1/4} \norm{\omega^b(s)}_{B_z L^{4/3}_x} + \norm{\omega^c(t)}_{B_z L^1_x} \right) = 0.
\end{align}
Again appealing to the Duhamel formula \eqref{eqPhi}, for any $\phi \in \Schwartz$ we may apply \eqref{LongTimeT} (with $\bK$ defined as therein) to yield
\begin{align}
&\int \phi \cdot \mathbb{S}(t,s) \nabla \cdot \left( u^b (s)\otimes \omega^b(s) - \omega^b(s) \otimes u^b(s) \right)\, dx\\
&\qquad\qquad = -\int \grad \phi : \mathbb{K}(t,s)  \left( u^b(s) \otimes \omega^b(s) - \omega^b(s) \otimes u^b(s) \right)\, dx \\ 
&\qquad\qquad \lesssim_{\phi} \left(\frac{t}{s} \right)^\gamma \norm{u^b(s)}_{L^\infty_z L^4_x} \norm{\omega^b(s)}_{B_z L^{4/3}_x} \lesssim \left(\frac{t}{s} \right)^\gamma \frac{1}{s^{\frac12}}; 
\end{align}
where $ 0 < \gamma < \frac{1}{2}$, and similarly for the other nonlinear term in the $\omega^b$ equation in \eqref{eqPhi}. 
As a consequence, $\omega^b(s)$ converges to $\mu^b$ in $\Schwartz'$ and hence $\omega \in \mc C_w ([0,T];\cM^{\frac{3}{2}})$ attains the initial data. 

Finally, we verify that $\omega$ satisfies \eqref{milddefDuhamel}. 
It follows that $\omega$ defined via the reconstruction \eqref{def:destrt} is a classical solution of the 3D Navier-Stokes equations for $t > 0$. 
Hence, for all $0 < s < t$, 
\begin{align}
\omega(t) = e^{(t-s)\Delta}\omega(s) - \int_s^t e^{(t-t')\Delta} B[u(t'),\omega(t')]\, dt'. \label{mildprop_strt}
\end{align}
Due to the self-adjointness of the heat semigroup in $L^2$ and continuity of $\omega$ in $\Schwartz'$, we can pass to the limit in the first term: $\lim_{s \searrow 0} e^{(t-s)\Delta}\omega(s) = e^{t\Delta} \omega(0)$. 
By the Dominated Convergence Theorem, we may pass $s \searrow 0$ also in the nonlinear term; indeed 
\begin{align*}
\int_0^t \norm{e^{(t-t')\Delta} B[u(t'),\omega(t')]}_{\cM^{\frac{3}{2}}} dt' & \lesssim \int_0^t \frac{1}{(t-t')^{\frac12}} \norm{u \otimes \omega}_{B_z L^1_x}\, dt' \\ 
& \lesssim \left(1 + \norm{(\omega^b,\Omega^c)}_{X}\right)^2 \int_0^t \frac{1}{(t-t')^{\frac12} (t')^{\frac12}}\,  dt'. 
\end{align*}

\qed

\section{The curved filament} \label{sec:Curve1}

We now move to the proof of Theorem \ref{thm:curved}, the case of an arbitrary closed, non-self-intersecting, smooth filament $\Gamma$ of length \(2\pi\). 
In this section we set up the local change of coordinates, describe and motivate the decomposition of the corrections, and then outline the fixed point argument.  
In the following section, we carry out the technical details of this fixed point. 

\subsection{A local change of variables} \label{sec:ChangeVars}
We recall that \(\gamma\colon \T\rightarrow \R^3\) is a unit speed parameterization of \(\Gamma\) and \(\ft,\fn,\fb\colon \T\rightarrow\R^3\) is an orthonormal frame along \(\Gamma\) so that \(\ft = \gamma'\) and \(\fb = \ft\times \fn\). We recall that for \(R>0\) we define a tubular neighborhood of \(\Gamma\) of radius \(32R\) in the physical frame
\[
\Gamma_R = \left\{y\in \R^3:\dist(y,\Gamma)<32R\right\},
\]
and a corresponding set in the straightened frame
\[
\Sigma_R = \left\{(x,z)\in \R^2\times \T:|x|<32R\right\}.
\]
By choosing \(0<R_0\ll1\) sufficiently small we may define the map \(\Phi\colon \Sigma_{R_0}\rightarrow \Gamma_{R_0}\) by
\[
\Phi(x,z) = \gamma(z) + x_1\fn(z) + x_2\fb(z),
\]
so that \(\Phi(\Sigma_R) = \Gamma_R\) for all \(0<R\leq R_0\).

The Jacobian of \(\Phi\) is
\eq{Jacobian}{
\fJ = \nabla\Phi = \begin{bmatrix}\fn & \fb & \fD \ft + \fE \fn + \fF \fb\end{bmatrix},
}
where we define the Jacobian determinant,
\[
\fD = \det \nabla\Phi = 1 + x_1\fn'\cdot\ft + x_2\fb'\cdot\ft,
\]
and the remaining coefficients
\[
\fE = x_2 \fb'\cdot \fn,\qquad \fF = x_1\fn'\cdot\fb.
\]

Taking \(\Gamma_{R_0}\) to be endowed with the Euclidean metric \(e\), the corresponding pullback metric on \(\Sigma_{R_0}\) is then determined by the matrix,
\[
g = \fJ^T\fJ = \begin{bmatrix}1&0&\fE\\0&1&\fF\\\fE&\fF&\fD^2 + \fE^2 + \fF^2\end{bmatrix}
\]
and hence \(\Phi\colon (\Sigma_{R_0},g)\rightarrow(\Gamma_{R_0},e)\) is a smooth isometry. 
In particular, if we define the function \(\bd\colon\R^3\rightarrow [0,\infty)\) by
\[
\bd(y) = \dist(y,\Gamma),
\]
then for all \((x,z)\in \Sigma_{R_0}\) we have
\[
(\bd\circ\Phi)(x,z) = |x|.
\]

We define the pushforward map \(P_\Phi\colon T\Sigma_{R_0}\rightarrow T\Gamma_{R_0}\) mapping velocity fields defined in the straightened frame to velocity fields defined in the physical frame by
\[
(P_\Phi v)\circ\Phi = \fJ v.
\]
We also define a normalized pushforward map \(Q_\Phi \colon T\Sigma_{R_0}\rightarrow T\Gamma_{R_0}\), designed to preserve the divergence-free condition, that we use to map vorticities in the straightened frame to vorticities in the physical frame by
\[
(Q_\Phi \eta)\circ\Phi = \fD^{-1}\fJ\eta.
\]

Using these definitions we have the following identities:
\begin{lem}\label{lem:Geometry}
Let \(v,\eta\colon \Sigma_{R_0}\rightarrow \R^3\) be smooth vector fields defined in the straightened frame. Taking \(x_3 = z\) we have the identities:~

\begin{itemize}

\item[(i)] \ul{Divergence.} The divergence operator satisfies
\eq{divg}{
\nabla \cdot (Q_\Phi \eta) = \left(\fD^{-1}\nabla\cdot \eta\right)\circ\Phi^{-1}.
}

\item[(ii)] \ul{Curl.} The curl operator satisfies
\eq{Curlg}{
\nabla\times Q_\Phi \eta = Q_\Phi \Curl_\Phi \eta,
}
where the twisted curl operator is given by
\[
\Curl_\Phi \eta = \nabla\times \eta + E^j\partial_j \eta + F \eta,
\]
the matrices $E^j$ by
\begin{gather*}
E^1 = \begin{bmatrix}0&0&0\\-\frac\fE\fD&-\frac\fF\fD&1 - \fD - \frac{\fE^2 + \fF^2}\fD\\0&\frac1\fD - 1&\frac\fF\fD\end{bmatrix},\qquad
E^2 = \begin{bmatrix}\frac\fE\fD&\frac\fF\fD&\fD - 1 + \frac{\fE^2 + \fF^2}\fD\\0&0&0\\1 - \frac1\fD&0&-\frac\fE\fD\end{bmatrix},\\ 
E^3 = \begin{bmatrix}0&1 - \frac1\fD&-\frac\fF\fD\\\frac1\fD - 1&0&\frac\fE\fD\\0&0&0\end{bmatrix},
\end{gather*}
and the matrix \(F\) is smooth and bounded.  

\item[(iii)] \ul{Bilinear operator.} The bilinear operator \(B[v,\eta] = \Div(v\otimes \eta - \eta\otimes v)\) satisfies
\eq{Bilg}{
B[P_\Phi v,Q_\Phi \eta] = Q_\Phi B[v,\eta].
}

\item[(iv)] \ul{Laplacian.} The Laplacian satisfies
\eq{Deltag}{\Delta Q_\Phi \eta= Q_\Phi \Delta_\Phi \eta,
}
where the twisted Laplacian is given by
\[
\Delta_\Phi = \Delta + A^{ij}\partial_i\partial_j + B^j\partial_j + C,
\]
the matrix $A = (A^{ij})_{ij}$ by
\[
A = \begin{bmatrix}\frac{\fE^2}{\fD^2}&\frac{\fE\fF}{\fD^2}&-\frac\fE{\fD^2}\\\frac{\fE\fF}{\fD^2}&\frac{\fF^2}{\fD^2}&-\frac{\fF}{\fD^2}\\-\frac\fE{\fD^2}&-\frac\fF{\fD^2}&\frac1{\fD^2} - 1\end{bmatrix},
\]
and the matrices \(B^j,C\) are smooth and bounded.
\end{itemize}
\end{lem}

\begin{rem} As a consequence of the above formulas for coordinate changes, and imagining for a second that $\Phi$ is defined globally, we can write the Navier-Stokes equation for $u = P_\Phi v$ and $\omega = Q_\Phi \eta$:
$$
\partial_t \eta + B[v,\eta] = \Delta_\Phi \eta, \quad \mbox{where} \quad \nabla \cdot \eta = 0 \quad \mbox{and} \quad v = \fD (- \Delta_\Phi)^{-1} \operatorname{curl}_\Phi \eta.
$$
However, $\Phi$ is only defined locally, so the above only holds on its domain of definition. This will complicate the fixed point scheme that we are about to write.
\end{rem}

\begin{proof}
We define the Christoffel symbols
\[
\Gamma_j = \fJ^{-1}\partial_j\fJ,
\]
and the associated covariant derivative
\[
\bD_j = \partial_j + \Gamma_j,
\]
so that for \(\bD \varphi = \begin{bmatrix} \bD_1 \varphi&\bD_2\varphi&\bD_3\varphi\end{bmatrix}\) we have
\eq{CovariantGradient}{
\nabla P_\Phi \varphi = P_\Phi \left((\bD \varphi) J^{-1}\right).
}
One may verify by direct calculation that 
\eq{JmT}{
(J^{-1})^T = \begin{bmatrix}\fn - \frac\fE\fD\ft & \fb - \frac\fF\fD\ft & \frac1\fD\ft\end{bmatrix},
}
from which we obtain explicit expressions for the Christoffel symbols
\begin{gather*}
\Gamma_1 = \begin{bmatrix}0&0&-\frac\fE\fD\partial_1\fD\\0&0&\partial_1\fF - \frac\fF\fD\partial_1\fD\\0&0&\frac1\fD\partial_1\fD\end{bmatrix},\qquad
\Gamma_2 = \begin{bmatrix}0&0&\partial_2\fE - \frac\fE\fD\partial_2\fD\\0&0&-\frac\fF\fD\partial_2\fD\\0&0&\frac1\fD\partial_2\fD\end{bmatrix},\\
\Gamma_3 = \begin{bmatrix}-\frac\fE\fD\partial_1\fD&\partial_2\fE - \frac\fE\fD\partial_2\fD & - \fD\partial_1\fD - \frac{\fE^2}\fD\partial_1\fD + \fF\partial_2\fE - \frac{\fE\fF}\fD\partial_2\fD + \partial_3\fE - \frac\fE\fD\partial_3\fD\\\partial_1\fF - \frac\fF\fD\partial_1\fD & -\frac\fF\fD\partial_2\fD&-\fD\partial_2\fD + \fE\partial_1\fF - \frac{\fE\fF}\fD\partial_1\fD - \frac{\fF^2}\fD\partial_2\fD + \partial_3\fF - \frac\fF\fD\partial_3\fD\\\frac1\fD\partial_1\fD&\frac1\fD\partial_2\fD&\frac1\fD\partial_3\fD + \frac\fE\fD\partial_1\fD + \frac\fF\fD\partial_2\fD\end{bmatrix}.
\end{gather*}

\begin{itemize}
\item[(i)] \underline{Divergence.} Taking the trace of \eqref{CovariantGradient} yields the expression
\eq{divgg}{
\left(\nabla\cdot P_\Phi \varphi\right)\circ\Phi = \tr (\bD \varphi) = \fD^{-1}\nabla\cdot(\fD \varphi).
}
Setting \(\varphi = \fD^{-1}\eta\) then gives \eqref{divg}.

\item[(ii)] \underline{Curl.}
Using the expression \eqref{CovariantGradient} we may write
\[
(\nabla\times P_\Phi v)\circ\Phi = \begin{bmatrix}(J\bD v J^{-1})_2^3 - (J\bD v J^{-1})_3^2\\(J\bD v J^{-1})_3^1 - (J\bD v J^{-1})_1^3\\(J\bD v J^{-1})_1^2 - (J\bD v J^{-1})_2^1\end{bmatrix},
\]
where \((J\bD vJ^{-1})_j^i\) denotes the \((i,j)\)\textsuperscript{th} entry of the matrix \(J\bD vJ^{-1}\). Next we define (the matrix \(M = (M^i_j)_{ij}\) of column vectors) 
\[
M_j^i = J^{-1} [ (J^{-1})^i\times J_j],
\]
where we take \((J^{-1})^i\) to be the \(i\)\textsuperscript{th} row of \(J^{-1}\) and \(J_j\) to be the \(j\)\textsuperscript{th} column of \(J\). We then see that
\[
\nabla\times P_\Phi v = P_\Phi\left[\sum_{i,j} M_j^i(\bD v)^j_i\right],
\]
where \((\bD v)_i^j\) is the \((j,i)\)\textsuperscript{th} entry of the matrix \(\bD v\).
Using the expressions \eqref{Jacobian}, \eqref{JmT} we then compute
\[
M\!\! =\! J^{-1}\!\begin{bmatrix}-\frac\fE\fD\fb & \ft + \frac\fE\fD\fn & \fF\ft + \frac{\fE\fF}\fD\fn - (\fD + \frac{\fE^2}\fD)\fb\\-\ft - \frac\fF\fD\fb&\frac\fF\fD\fn&-\fE\ft + (\fD + \frac{\fF^2}\fD)\fn - \frac{\fE\fF}\fD\fb\\\frac1\fD\fb&-\frac1\fD\fn& - \frac\fF\fD\fn + \frac\fE\fD\fb\end{bmatrix}\!\! =\!\! \begin{bmatrix}\begin{bmatrix}0\\-\frac\fE\fD\\0\end{bmatrix}&\begin{bmatrix}0\\-\frac\fF\fD\\\frac1\fD\end{bmatrix}&\begin{bmatrix}0\\-\fD - \frac{\fE^2 + \fF^2}\fD\\\frac\fF\fD\end{bmatrix}\\\begin{bmatrix}\frac\fE\fD\\0\\-\frac1\fD\end{bmatrix}&\begin{bmatrix}\frac\fF\fD\\0\\0\end{bmatrix}&\begin{bmatrix}\fD + \frac{\fE^2 + \fF^2}\fD\\0\\-\frac\fE\fD\end{bmatrix}\\\begin{bmatrix}0\\\frac1\fD\\0\end{bmatrix}&\begin{bmatrix}-\frac1\fD\\0\\0\end{bmatrix}&\begin{bmatrix}-\frac\fF\fD\\\frac\fE\fD\\0\end{bmatrix}\end{bmatrix},
\]
from which we obtain the expression \eqref{Curlg}.

\item[(iii)] \underline{Bilinear operator.}
Denoting the \((i,j)\)\textsuperscript{th} entry of the matrix \(\Gamma_k\) by \(\Gamma_{kj}^i\) we see that the matrices satisfy the symmetry \(\Gamma_{kj}^i = \Gamma_{jk}^i\). As a consequence, for vector fields \(v,\varphi\) we may use the expression \eqref{CovariantGradient} to obtain,
\[
P_\Phi v\cdot\nabla P_\Phi \varphi - P_\Phi \varphi\cdot\nabla P_\Phi v = P_\Phi\left(v\cdot\nabla \varphi - \varphi\cdot\nabla v\right).
\]
Taking \(\varphi = \fD^{-1}\eta\) we then use the expression \eqref{divgg} to compute,
\begin{align*}
B[P_\Phi v,Q_\Phi\eta] &= (\nabla \cdot P_\Phi v)P_\Phi\varphi - (\nabla\cdot P_\Phi \varphi)P_\Phi v + P_\Phi v\cdot\nabla P_\Phi \varphi - P_\Phi \varphi\cdot\nabla P_\Phi v\\
&= P_\Phi\left(\fD^{-1}\nabla\cdot(\fD v)\, \varphi - \fD^{-1}\nabla\cdot(\fD \varphi)\, v + v\cdot\nabla\varphi - \varphi\cdot \nabla v\right)\\
&= Q_\Phi B[v,\eta].
\end{align*}

\item[(iv)] \underline{Laplacian.} A computation yields the covariant Laplacian (see for example~\cite[Lemma~4.8]{MR1468735}), 
\eq{CovariantLaplacian}{
\Delta P_\Phi \varphi =  P_\Phi\left(g^{ij}\left(\bD_i\bD_j\varphi - \Gamma^k_{ij} \bD_k\varphi\right)\right),
}
where \(g^{ij}\) is the \((i,j)\)\textsuperscript{th} entry of the matrix,
\[
g^{-1} = \begin{bmatrix}1 + \frac{\fE^2}{\fD^2}&\frac{\fE\fF}{\fD^2}&-\frac\fE{\fD^2}\\\frac{\fE\fF}{\fD^2}&1 + \frac{\fF^2}{\fD^2}&-\frac{\fF}{\fD^2}\\-\frac\fE{\fD^2}&-\frac\fF{\fD^2}&\frac1{\fD^2}\end{bmatrix}.
\]
Taking \(\varphi = \fD^{-1}\eta\) in the expression \eqref{CovariantLaplacian} we then obtain the expression \eqref{Deltag}.
\end{itemize}
\end{proof}

The next lemma is a consequence of the above calculations. It is used in multiple places below and, moreover, emphasizes why one expects the curvature of the filament to be subcritical.
\begin{lemma} \label{lem:Coefs}
Provided \(0<R_0\ll1\) is sufficiently small and \((x,z)\in \Gamma_{R_0}\), there holds the following for all $j\geq 0$ and multi-indices \(\alpha\in \N^3\),
\begin{align}
&\abs\fD \gtrsim 1,\\
&\abs{\partial_z^j(\fD - 1)}  \lesssim_j \abs{x} \\ 
&\abs{\nabla^\alpha_{x,z}\Phi} + \abs{\nabla_{x,z}^\alpha\fJ} + \abs{\nabla_{x,z}^\alpha\fD}  \lesssim_{\alpha} 1,
\end{align}
the coefficients \(E^i,A\) satisfy the estimates
\begin{align}
\abs{\partial_z^jE^1} + \abs{\partial_z^jE^2} + \abs{\partial_z^jE^3} + \abs{\partial_z^j A} \lesssim_j \abs{x},  
\end{align}
and similarly the coefficients $A,B^i,C,E^i,F$ satisfy the following estimates (where $X$ is any one of $A,B^i,C,E^i,F$)
\begin{align}
\abs{\grad_{x,z}^\alpha X} \lesssim_\alpha 1.  
\end{align}
\end{lemma}

As a simple application of this lemma we have the following estimate:
\begin{lem}\label{lem:PhiDiff}
Provided \(0<R_0\ll1\) is sufficiently small, whenever \((x,z),(x',z')\in \Sigma_{R_0}\) we have the estimate
\eq{PhiDiff}{
|\Phi(x,z) - \Phi(x',z')|\approx |x - x'| + |z - z'|.
}
\end{lem}
\bpf
Take \(\delta>0\) and suppose that \(|z - z'|<\delta\). Applying Taylor's Theorem and Lemma~\ref{lem:Coefs} we obtain
\[
\Phi(x,z) - \Phi(x',z') = \nabla\Phi(x',z')\begin{bmatrix}x - x'\\z - z'\end{bmatrix} + O\left(\delta\left(|x - x'| + |z - z'|\right)\right),
\]
where we note that we have used the fact that \(\Phi\) is linear in \(x\). From the definition of \(\Phi\) we may compute
\[
\left|\nabla\Phi(x',z')\begin{bmatrix}x - x'\\z - z'\end{bmatrix}\right|^2 = \left(|x - x'|^2 + |z - z'|^2\right)\left(1 + O(R_0)\right).
\]
In particular, provided \(0<\delta,R_0\ll1\) are sufficiently small, we obtain the estimate \eqref{PhiDiff}.

Conversely, if \(|z - z'|\geq \delta\) then from the definition of \(\Phi\) we see that
\[
\Phi(x,z) = \gamma(z) + O(R_0).
\]
As \(\Gamma\) is a simple smooth closed curve we have
\[
|\gamma(z) - \gamma(z')|\approx_\delta 1,
\]
whenever \(|z - z'|\geq\delta\). Thus, by choosing \(R_0\) sufficiently small (depending on \(\delta\)) we obtain the bound \eqref{PhiDiff}.
\epf

Finally, for each \(0<R\leq R_0\) we define the bump function \(\chi_R = \chi_R(x)\) to be a smooth, non-negative, radial bump function identically \(1\) on the set \(\{|x|\leq R\}\) and supported on the set \(\{|x|\leq 2R\}\). In particular, we will assume that \(\chi_R(x)>0\) for \(|x|<2R\).  We also define associated functions \(\widetilde\chi_R\) in the physical frame by \(\widetilde\chi_R\circ\Phi = \chi_R\) and observe that
\[
\widetilde\chi_RP_\Phi(\ \cdot\ ) = P_\Phi(\chi_R\ \cdot\ ),\qquad \widetilde\chi_RQ_\Phi(\ \cdot\ ) = Q_\Phi(\chi_R\ \cdot\ ),
\]
whenever \(0<R\leq R_0\).

\subsection{Decomposition of the solution} \label{sec:decomp}
As discussed after the statement of Theorem \ref{thm:curved} in Section \ref{sec:curvedIntro}, the nonlinear perturbation argument for Theorem \ref{thm:curved} is significantly more technical than for Theorems \ref{thm:StrtFil} and \ref{thm:StrtFil2}. In this section we describe the decomposition; in the following section we describe the norms and the fixed point scheme. 
We will construct our solution \(\omega\) in the physical frame by mimicking the straight filament and decomposing
\[
\omega = \widetilde\omega^c + \omega^b,
\]
where the core and background pieces satisfy the equations
\begin{align}
\partial_t \widetilde\omega^c + B[u,\widetilde\omega^c] &= \Delta\widetilde\omega^c,\label{corePhys}\\
\partial_t \omega^b + B[u,\omega^b] &= \Delta\omega^b\label{bgPhys},
\end{align}
the velocity is defined via the Biot-Savart law as
\eq{UsualBS}{
u = (-\Delta)^{-1}\nabla\times \omega,
}
and with initial data
\eq{InitData}{
\widetilde\omega^c(t = 0) = \alpha \delta_\Gamma,\qquad \omega^b(t = 0) = \mu^b.
}

To construct the core piece and background piece we will solve a system of 4 equations: 2 equations in the straightened frame to obtain vector fields \(\widetilde\eta^{c1}\), \(\eta^{b1}\) and two equations in the physical frame to obtain vector fields \(\omega^{c2}\), \(\omega^{b2}\). We will then construct the core and background pieces as
\eq{Reconstruction}{
\widetilde\omega^c = \underbrace{Q_\Phi(\chi_{2R}\,\widetilde\eta^{c1})}_{\widetilde\omega^{c1}} + \omega^{c2},\qquad \omega^b = \underbrace{Q_\Phi(\chi_{2R}\,\eta^{b1})}_{\omega^{b1}} + \omega^{b2},
}
for sufficiently small \(0<R\leq R_0\).

As in the case of the straight filament, we take
\[
\widetilde\eta^{c1} = \alpha \eta^g + \eta^{c1},
\]
where the Gaussian,
\[
\eta^g(t,x) = \begin{bmatrix}0\\\frac 1 tG(\frac x{\sqrt t})\end{bmatrix}.
\]
We may then write \(\widetilde \omega^c = \omega^g + \omega^c\), where the vector fields,
\[
\omega^g = Q_\Phi(\chi_{2R}\,\eta^g),\qquad \omega^c = \underbrace{Q_\Phi(\chi_{2R}\,\eta^{c1})}_{\omega^{c1}} + \omega^{c2}.
\]

We remark that, by an abuse of notation, in the curved case we will use \(\eta^g\) for the Gaussian vortex defined in the straightened frame. We will then take \(\omega^g\) to be corresponding vorticity in the physical frame. It is worth emphasizing a key difference from the case of the straight filament: with the current construction, \(\omega^g\) is no longer an exact solution of \eqref{NSE}.

The equations that define \(\eta^{c1},\omega^{c2},\eta^{b1},\omega^{b2}\) are given in \eqref{Core1Piece}, \eqref{Core2PiecePhys}, \eqref{bg1Piece}, \eqref{bg2PiecePhys} respectively.

\subsubsection{Some definitions}
Before proceeding, it will be useful to recall how we transform the various ingredients from the physical to straightened frame and back.

\begin{itemize}
\item Our basic ingredients in the straightened frame are:
\begin{itemize}
\item The Gaussian vortex \(\eta^g = \begin{bmatrix}0\\\frac1tG(\frac x{\sqrt t})\end{bmatrix}\)
\item The core-1 (c1) piece \(\eta^{c1}\) defined by the equation \eqref{Core1Piece}
\item The background-1 (b1) piece \(\eta^{b1}\) defined by the equation \eqref{bg1Piece}
\end{itemize}
and in the physical frame are:
\begin{itemize}
\item The core-2 (c2) piece \(\omega^{c2}\) defined by the equation \eqref{Core2PiecePhys}
\item The background-2 (b2) piece \(\omega^{b2}\) defined by the equation \eqref{bg2PiecePhys}
\end{itemize}

\item To transform vorticities from the straightened to physical frame we use the map \(\omega^* = Q_\Phi(\chi_{2R}\,\eta^*)\) and hence:
\[
\omega^g = Q_\Phi(\chi_{2R}\,\eta^g),\qquad
\omega^{c1} = Q_\Phi(\chi_{2R}\,\eta^{c1}),
\qquad \omega^{b1} = Q_\Phi(\chi_{2R}\,\eta^{b1}).
\]
\item To transform vorticities from the physical frame to the straightened frame we use the map \(\eta^* = Q_\Phi^{-1}(\widetilde\chi_{2R}\omega^*)\) and hence:
\[
\eta^{c2} = Q_\Phi^{-1}(\widetilde\chi_{2R}\,\omega^{c2}),\qquad \eta^{b2} = Q_\Phi^{-1}(\widetilde\chi_{2R}\,\omega^{b2}).
\]

\item To define velocities in the physical frame we use the Biot-Savart law:
\[
u^* = (-\Delta)^{-1}\nabla\times \omega^*.
\]
With this definition,
$$
u = \alpha u^g + \underbrace{u^{c1} + u^{c2}}_{u^c} + \underbrace{u^{b1} + u^{b2}}_{u^b}.
$$
\item To transform these velocities from the physical frame into the straightened frame we use the map:
\[
v^* = P_\Phi^{-1}(\widetilde\chi_R\,u^*)
\]
Therefore,
$$
v = \alpha v^g + \underbrace{v^{c1} + v^{c2}}_{v^c} + \underbrace{v^{b1} + v^{b2}}_{v^b}.
$$
\item We will use the explicit velocity field associated to the Gaussian vortex in the straightened frame and corresponding pushforward to the physical frame, which we denote by:
\[
\bar v^g = \begin{bmatrix}\frac 1 {\sqrt t}g(\frac x{\sqrt t})\\0\end{bmatrix},\qquad \bar u^g = P_\Phi(\chi_{4R}\, \bar v^g).
\]
\item We define the self-similar scalings of our variables in the straightened frame by
\begin{gather*}
H^*(\tau,\xi,z) = e^\tau \eta^*(e^\tau,e^{\frac\tau2}\xi,z),\qquad V^*(\tau,\xi,z) = e^{\frac\tau2}v^*(e^\tau,e^{\frac\tau2}\xi,z),\\
\bar V^g(\xi) = \begin{bmatrix}  g(\xi)\\0\end{bmatrix}.
\end{gather*}

\item Finally, the independent variable is denoted $(x,z)$ in the straightened frame; $(\xi,z)$ in the straightened frame after self-similar scaling; and $y$ in the physical frame. We try to keep track of the space where a function is defined by denoting $L^p_{x,z}$, $L^p_y$, etc... for the corresponding functional spaces.

\end{itemize}

\subsubsection{The approximate solution \(\eta^g\)}\label{sec:ugdef}
Due to the radial symmetry of \(\chi_R\) we have the equation:
\eq{ApproxSoln}{\boxed{
\partial_t \eta^g +  B[\alpha \chi_R\,\bar v^g,\eta^g] = \Delta \eta^g.
}
}
Applying the operator \(Q_\Phi( \chi_{2R}\,\cdot\,)\) to the equation \eqref{ApproxSoln} we obtain an equation in the physical frame (see Section \ref{sec:ChangeVars}),
\eq{ApproxSolnPhys}{
\partial_t \omega^g + B[\alpha \widetilde\chi_R\, \bar u^g,\omega^g] = \Delta \omega^g + \mc E^g,
}
where the error term is given by
\[
\mc E^g = Q_\Phi\left(\chi_{2R}\,\Delta \eta^g - \Delta_\Phi(\chi_{2R}\,\eta^g)\right).
\]
We note that crucially, $\bar u^g$ is not what is obtained by the Biot-Savart law, i.e.
$$\bar u^g \neq \chi_{4R}\, u^g \qquad \mbox{where} \qquad u^g = \grad \times (-\Delta)^{-1} \omega^g.$$
However, \(\bar u^g\) is ``close'' to $u^g$ in a reasonable sense, see Proposition~\ref{prop:ApproxBS}.

\subsubsection{The core piece \(\eta^{c1}\)}
We take the core piece \(\eta^{c1}\) (defined in the straightened frame) to satisfy the equation
\eq{Core1Piece}{\boxed{
\partial_t\eta^{c1} + B[v,\eta^{c1} + \eta^{c2}] + B[v - \alpha \chi_R\,\bar v^g,\alpha \eta^g] = \Delta \eta^{c1}
}
}
(recall that $v = P_\Phi^{-1} (\widetilde{\chi}_R\, u)$).
We note that this is an inhomogeneous equation (in the functions \((\eta^{c1},\omega^{c2},\eta^{b1},\omega^{b2})\)) with forcing term given by
\eq{c1InhomogeneousTerm}{
f^{c1}_I =\alpha^2 B[v^g - \chi_R\,\bar v^g,\eta^g].
}

Applying the operator \(Q_\Phi\chi_{2R}\) then yields an equation in the physical frame,
\eq{Core1PiecePhys}{
\partial_t\omega^{c1} + B[\widetilde\chi_R\, u,\omega^c] + B[\widetilde\chi_R\,(u - \alpha \bar u^g),\alpha \omega^g] = \Delta\omega^{c1} + \mc E^c,
}
for the error term
\[
\mc E^c = Q_\Phi\left(\chi_{2R}\,\Delta\eta^{c1} - \Delta_\Phi(\chi_{2R}\,\eta^{c1})\right).
\]

\subsubsection{The core piece \(\omega^{c2}\)}
The second core piece \(\omega^{c2}\) (defined in the physical frame) is then taken to satisfy the difference between the equation \eqref{corePhys} and the equations \eqref{ApproxSolnPhys}, \eqref{Core1PiecePhys},
\eq{Core2PiecePhys}{\boxed{
\partial_t\omega^{c2} + B[(1 - \widetilde\chi_R)u,\omega^c] +B[(1 - \widetilde\chi_R)(u - \alpha \bar u^g),\alpha \omega^g] = \Delta\omega^{c2} - \mc E^c- \alpha\mc E^g.
}
}
Once again we note that \eqref{Core2PiecePhys} is an inhomogeneous equation with forcing term
\eq{c2InhomogeneousTerm}{
f^{c2}_I = \alpha^2 B[(1 - \widetilde\chi_R)(u^g - \bar u^g),\omega^g] + \alpha \mc E^g.
}

\subsubsection{The background piece \(\eta^{b1}\)}
We take the background piece \(\eta^{b1}\) (defined in the straightened frame) to satisfy the equation
\eq{bg1Piece}{\boxed{
\partial_t\eta^{b1} + B[v,\eta^{b1} + \eta^{b2}] = \Delta\eta^{b1}.
}}
Applying the operator \(Q_\Phi(\chi_{2R}\,\cdot\,)\) we then obtain an equation in the physical frame,
\eq{bg1PiecePhys}{
\partial_t\omega^{b1} + B[\widetilde\chi_R\, u , \omega^b] = \Delta\omega^{b1} + \mc E^b,
}
where the error
\[
\mc E^b = Q_\Phi\left(\chi_{2R}\,\Delta\eta^{b1} - \Delta_\Phi(\chi_{2R}\,\eta^{b1})\right).
\]

\subsubsection{The background piece \(\omega^{b2}\)}
The second background piece \(\omega^{b2}\) (defined in the physical frame) is then taken to satisfy the difference between the equations \eqref{bgPhys} and \eqref{bg1PiecePhys},
\eq{bg2PiecePhys}{\boxed{
\partial_t\omega^{b2} + B[(1 - \widetilde\chi_R)u,\omega^b] = \Delta\omega^{b2} - \mc E^b.
}}

\subsection{Outline of the fixed point argument} 
As in the straight filament case in Section \ref{sec:Straight}, we will set up a fixed point in a suitable norm.
However, here the norms are more subtle. For estimating the core contributions $\omega^{c1}$ and $\omega^{c2}$ ``near'' and ``far'' from the filament: for $\beta \geq 0$,
\begin{align*}
\norm{\eta^*}_{\bN^\beta_c} & = \sup\limits_{0 < t \leq T} \sqrt{t}  \norm{ \Brak{\tfrac{x}{\sqrt{t}}}^{m}\brak{\sqrt{t} \grad}^\beta \eta^*}_{B_z L^2_x},  \\ 
\norm{\omega^*}_{\bF_c} &= \sup\limits_{0<t\leq T}\sqrt{t}\norm{(1-\widetilde\chi_{6R})\Brak{\tfrac{\bd}{\sqrt{t}} }^m \omega^*}_{L^3_y}.
\end{align*}
Note that the self-similar scaling yields the alternative expression for the \(\bN^\beta_c\) norm:
\[
\norm{\eta^*}_{\bN^\beta_c} = \sup\limits_{-\infty<\tau\leq \ln T}\norm{\brak{\onabla}^{\beta} H^*(\tau)}_{B_z L^2_\xi(m)}.
\]
For the background contributions, it turns out to be better to change the framework from that used with the straight filament.
This is due to difficulty \textbf{(b)} discussed after the statement of Theorem~\ref{thm:curved}.
We define the following norms
\begin{align*}
\norm{\eta^*}_{\bN_b^\beta} &= \sup\limits_{0 < t \leq T} t^{\frac{1}{4}}\left(R^{-\frac{3}{4}} \norm{ \brak{\sqrt{t} \nabla}^\beta \eta^* }_{L^{4/3}_{x,z}} + R^{\frac14}\norm{\brak{\sqrt{t} \nabla}^\beta \nabla \eta^*}_{L^{4/3}_{x,z}}\right),\\
\norm{\omega^*}_{\bF_b} &= \sup\limits_{0 < t \leq T} t^{\frac14}\left(R^{-\frac34}\norm{\omega^*}_{L^{4/3}_y} + R^{\frac14}\norm{\nabla \omega^*}_{L^{4/3}_y}\right).
\end{align*}
Notice both the $t$ scaling and the $R$ scaling.
In the absence of any filaments, one of the natural critical spaces to close a fixed point with initial data in the scale-invariant space $L_y^{\frac32}$ is $\sup_t t^{\frac14}\norm{\omega(t)}_{L^2_y}$.
By interpolation and Sobolev embedding, the following holds (with implicit constants independent of $R$ and $t$): 
\begin{align}
t^{\frac14}\norm{f}_{L^{2}_y} \leq t^{\frac14} \| f\|_{L^{4/3}_y}^{\frac14} \| f \|_{L^{12/5}_y}^{\frac34} \lesssim t^{\frac14} \| f\|_{L^{4/3}_y}^{\frac14} \| \nabla f\|_{L^{4/3}_y}^{\frac34} \lesssim \norm{f}_{\bF_b}; \label{ineq:FbL2}
\end{align}
thus the space $\bF_b$ is in fact a smaller critical space strictly contained in the natural $\sup_t t^{\frac14}\norm{f}_{L^2}$ space. 
Note that the $R$ dependence in $\bN_b^\beta$ and $\bF_b$ is set in the dimension-consistent manner so that embeddings such as \eqref{ineq:FbL2} are independent of $R$.

Roughly speaking, the parameter \(R\) can be thought of as controlling the error between the flat and curved Laplacians (see Lemma~\ref{lem:Coefs}). 
By choosing $R$ small, the curved Laplacian $\Delta_{\Phi}$ is well-approximated by the flat $\Delta$ (this is used to control errors in both the dissipation and in the Biot-Savart law). Indeed, there are a few terms where choosing $R$ small is necessary. However, choosing $R$ small makes other errors large. 
This detrimental $R$ dependence is eventually absorbed by choosing $T$ small, so favorable $T$ renders this issue harmless in many terms, but there are a few terms where it is imperative to track the dependence carefully (this also explains why it is convenient to define $\bF_b$ and $\bN_b^\beta$ with the dimension-consistent scaling in $R$).

We now choose \(\beta\in(0,\frac14)\) to be some fixed constant and for constants $M_c, M_{b1}, M_{b2}$ suitably chosen (depending on $R$, $\alpha$) we define the norm on the solution:
\begin{equation}\label{BigNorm}
\begin{aligned}
\norm{(\eta^{c1},\omega^{c2},\eta^{b1},\omega^{b2})}_{X_T} & =  \norm{\eta^{c1}}_{\bN^{\beta}_c} + M_{c}\norm{ Q_\Phi^{-1} [\widetilde{\chi}_{8R}\omega^{c2}] }_{\bN^0_c} + M_{c} \norm{\omega^{c2}}_{\bF_c} \\ 
& \quad + M_{b1}\norm{\eta^{b1}}_{\bN_b^{\beta}} + M_{b2} \norm{\omega^{b2}}_{\bF_b}. 
\end{aligned}
\end{equation}
We remark that we expect to be able to estimate the \(*1\)-pieces similarly to the straight filament case, and the \(b2\)-piece will be more straightforward thanks to the subcritical nature of the \(\bF_b,\bN_b\) spaces. Thus, the most challenging term to control will be the \(c2\)-piece. Here it is useful to note that the ``worst'' contributions to the \(c2\)-piece are from the \(\omega^g\) and \(\omega^{c1}\) pieces that are supported in the set \(\{\bd\leq 4R\}\). In particular, the \(X_T\)-norm is set up precisely so that we control \(\omega^{c2}\) in the anisotropic \(\bN_c^0\) norm in a neighborhood of these terms, \(\{\bd\leq 8R\}\), whereas we measure \(\omega^{c2}\) in the isotropic \(\bF_c\) norm in the region \(\{\bd\geq 6R\}\), separated from these terms by a distance of size \(O(R)\) (see Figure~\ref{fig:c2supports}). This separation of supports allows us to use the smoothing properties of the heat operator to switch from anisotropic to isotropic spaces, modulo bounded errors.


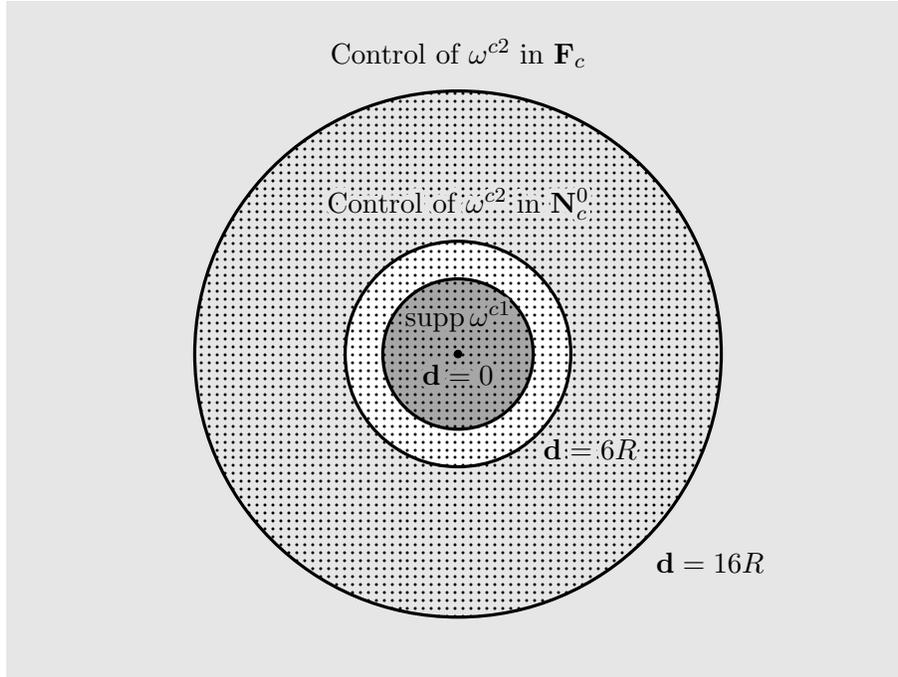
\begin{figure}[h]
\centering
\begin{tikzpicture}
\draw[fill=gray!20!white,draw=none] (-6,-4.3) -- (6,-4.3) -- (6,4.7) -- (-6,4.7)--(-6,-4.3);
\draw[fill=white, very thick] (0,0) circle (1.5);
\draw[fill=gray!70!white, very thick] (0,0) circle (1);
\draw[pattern=dots, very thick] (0,0) circle (3.5);

\node at (0,.5) {\contour{gray!70!white}{\(\supp \omega^{c1}\)}};
\node at (0,2) {\contour{gray!20!white}{Control of \(\omega^{c2}\) in \(\bN_c^0\)}};
\node at (0,4) {Control of \(\omega^{c2}\) in \(\bF_c\)};
\draw[fill=black] (0,0) circle (.05);
\node[below] at (0,0) {\contour{gray!70!white}{\(\bd = 0\)}};
\node[below right] at (1,-1) {\contour{gray!20!white}{$\bd=6R$}};
\node[below right] at (2.5,-2.5) {$\bd = 16R$};
\end{tikzpicture}
\caption{The regions of control of \(\omega^{c2}\). The dark gray area \(\{\bd\leq 4R\}\) is the support \(\omega^{c1}\). The dotted region \(\{\bd\leq 16R\}\) is the support of \(\widetilde\chi_{8R}\omega^{c2}\), controlled in the \(\bN_c^0\) norm. The light gray region \(\{\bd\geq 6R\}\) is the support of \((1 - \widetilde\chi_{6R})\omega^{c2}\), controlled in the \(\bF_c\) norm. Crucially, we will take advantage of the fact that there is a separation of size \(O(R)\) between the (dark gray) support of \(\omega^{c1}\) and the (light gray) support of \((1 - \widetilde\chi_{6R})\omega^{c2}\), which we control in the isotropic norm \(\bF_c\).}\label{fig:c2supports}
\end{figure}


For a judicious choice of the constants \(M_*\) and sufficiently small $T,R,\epsilon$, we will apply Banach's fixed point theorem in the ball 
\begin{align} 
B_{\eps,T,R,M_*} = \set{(\eta^{c1},\omega^{c2},\eta^{b1},\omega^{b2}): \norm{(\eta^{c1},\omega^{c2},\eta^{b1},\omega^{b2})}_{X_T} \leq \eps}
\end{align}
to the mapping $\mathcal{Q}:\omega \mapsto a$, where, by a slight abuse of notation, we take $\omega = (\eta^{c1},\omega^{c2},\eta^{b1},\omega^{b2})$, and $a = (a^{c1}, a^{c2},a^{b1},a^{b2})$ is given by
\begin{align}
A^{c1} & = -\int_{-\infty}^{\tau} S(\tau,s) F^{c1}(s) \, ds, \\ 
a^{c2} & = - \int_0^t e^{(t - s)\Delta}f^{c2}\, ds, \label{eq:wc2Dum} \\ 
a^{b1} & = - \int_0^t \bS(t,s) f^{b1}(s) ds, \\
a^{b2} & = e^{t\Delta}\mu^b - \int_0^t e^{(t-s)\Delta} f^{b2}\, ds, 
\end{align}
where \(A^{c_1}\), respectively \(F^{c1}\), is the self-similar scaling of \(a^{c_1}\), respectively \(f^{c1}\), and the perturbative terms:
\begin{align}
f^{c1} & = B[v - \alpha \bar v^{g},\eta^{c1}] + B[v, \eta^{c2}] + B[v - \alpha v^g - (-\Delta)^{-1} \nabla \times \eta^{c1}, \alpha \eta^g] + f^{c1}_I, \\ 
f^{c2} &= B[(1 - \widetilde\chi_R)u,\omega^c] + B[(1 - \widetilde\chi_R)(u - \alpha u^g), \alpha \omega^g] + \mc E^c + f^{c2}_I,\\
f^{b1} & = B[v - \alpha \bar v^g, \eta^{b1}] + B[v, \eta^{b2}],\\
f^{b2} &= B[(1-\widetilde{\chi}_R)u,\omega^b] + \cE^b,
\end{align}
with the inhomogeneous terms \(f^{c1}_I, f^{c2}_I\) defined as in \eqref{c1InhomogeneousTerm}, \eqref{c2InhomogeneousTerm} respectively.

The existence of a solution of \eqref{NSE} then follows as a consequence of the following Theorem:
\begin{thrm} \label{thm:FixedPt}
For $M_c$, $M_{b1}$, $M_{b2}$, R, $T$, $\eps$ suitably chosen,  $\mathcal{Q}:B_{\eps,T,R,M_\ast} \mapsto B_{\eps,T,R,M_\ast}$ and for $\rho,r \in B_{\eps,T,R,M_*}$
there holds the contraction property $\norm{\mathcal{Q}(\rho) - \mathcal{Q}(r)}_{X_T} \leq \frac{1}{2}\norm{\rho - r}_{X_T}$. 
It follows that there exists a unique fixed point $\mathcal{Q}(w) = w \in B_{\eps,T,R,M_\ast}$. More specifically, for all R sufficiently
small, there exists \(T\), \(\eps\), and \(M_*\) such that the fixed point holds.
\end{thrm} 

In the following section, we will only detail the estimates which give $\mathcal{Q}:B_{\eps,T,R,M_\ast} \mapsto B_{\eps,T,R,M_\ast}$; the extension to the contraction property is straightforward and is omitted for the sake of brevity.
\emph{All} implicit constants in the following will be \emph{independent of} $T$ (and any other time variables), $R$, $M_{c1}$, $M_{c2}$, $M_{b1}$, $M_{b2}$, and $\eps$ unless otherwise specified (but in general will depend on $\alpha$, $\beta$ and $m$). 

In order to simplify the exposition, it will be useful to assume initially that in addition to the assumption that \(0<R\leq R_0\) we have
\eq{Size}{
0<T\leq R^{12} \leq1 ,
}
and that the constants \(M_*\) are chosen so that
\eq{ConstantOrder}{
1\leq M_c\leq R^{\frac34}M_{b1}\leq R^{\frac34} M_{b2}.
}
In order to clarify the proof of the contraction, we will invoke the assumption \eqref{Size} in the statement of Proposition~\ref{prop:AllAs} below. However, in Section~\ref{sec:Curve2} we will largely avoid using assumption \eqref{Size} in order to elucidate the various bounds. The assumption \eqref{ConstantOrder} will be used throughout.

The proof that  $\mathcal{Q}:B_{\eps,T,R,M_\ast} \mapsto B_{\eps,T,R,M_\ast}$ is a consequence of the following a priori estimates for the terms \(a^*\):                                                  

\begin{prop}\label{prop:AllAs}~
Let \(\beta\in(0,\frac14)\) and \(m\geq 2\) be fixed. Provided \(0<R\leq R_0\) is sufficiently small and the constants \(T,M_*\) satisfy the assumptions \eqref{Size},~\eqref{ConstantOrder} the following a priori estimate holds for all $\omega \in B_{\eps,T,R,M_\ast}$:  
\begin{align}
\norm{a^{c1}}_{\bN_c^{\beta}} & \lesssim R\ln R^{-1} + \left(R\ln R^{-1} + \frac1{M_c}\right)\eps + \eps^2,\label{ineq:wc1bd}\\
M_c\norm{Q_\Phi^{-1} \left(\widetilde\chi_{8R}a^{c2}\right)}_{\bN_c^0} & \lesssim M_cR + \left(M_c R + R^{\frac34} + \frac{M_{b1}}{M_{b2}}\right) \epsilon + M_c\eps^2 \label{ineq:wc2-a},\\
M_c\norm{a^{c2}}_{\bF_c}&\lesssim M_c R + \left(M_cR + R^{\frac14}\right)\epsilon + M_c\epsilon^2 \label{ineq:wc2-b},\\
M_{b1}\norm{a^{b1}}_{\bN_b^\beta}&\lesssim \left(R\ln R^{-1} + \frac{M_{b1}}{M_{b2}}\right)\epsilon + \eps^2, \label{ineq:wb1}\\
M_{b2}\norm{a^{b2}}_{\bF_b}&\lesssim  M_{b2}\norm{e^{t \Delta}\mu^b}_{\bF_b} + \frac{M_{b2}}{M_{b1}}R\epsilon + \frac{M_{b2}}{M_{b1}}\eps^2. \label{ineq:wb2}
\end{align}

Further, if \(\omega,\omega'\in B_{\epsilon,T,R,M_*}\) have the same initial data, the differences satisfy the bounds 
\begin{align}
\|a^{c1}(\omega) - a^{c1}(\omega')\|_{\bN_c^{\beta}} &\lesssim \left(R\ln R^{-1} + \frac1{M_c} +  \eps \right) \|\omega - \omega'\|_{X_T}, \\
M_c\norm{Q_\Phi^{-1}\left(\widetilde\chi_{8R}\left(a^{c2}(\omega) - a^{c2}(\omega')\right)\right)}_{\bN_c^0} &\lesssim \left(M_c R + R^{\frac34} + \frac{M_{b1}}{M_{b2}} + M_c\epsilon\right)\|\omega - \omega'\|_{X_T},\\
M_c\norm{a^{c2}(\omega) - a^{c2}(\omega')}_{\bF_c} &\lesssim  \left(M_cR + R^{\frac14} + M_c\epsilon\right)\norm{\omega - \omega'}_{X_T}, \\
M_{b1}\norm{a^{b1}(\omega) - a^{b1}(\omega')}_{\bN_b^\beta}&\lesssim \left( R\ln R^{-1} + \frac{M_{b1}}{M_{b2}} +   \eps \right)\|\omega - \omega'\|_{X_T} , \\
M_{b2}\norm{a^{b2}(\omega) - a^{b2}(\omega')}_{\bF_b}&\lesssim \left(\frac{M_{b2}}{M_{b1}} R  +  \frac{M_{b2}}{M_{b1}}\eps\right)\|\omega - \omega'\|_{X_T}.
\end{align}
\end{prop} 

\begin{proof}[\textbf{Proof of Theorem \ref{thm:FixedPt}}]
Now let us explain how to set $R$, $T$, and $M_\ast$ so that the hypotheses of the Banach fixed point theorem holds. 
The most subtle point is choosing $M_\ast$ and $R$ consistently. 
First, we set $M_c$ such that the third term in \eqref{ineq:wc1bd} is consistent; note this requires choosing $M_c$ large relative only to $\alpha$. 
Next, set $M_{b1} = R^{-\frac34}M_c$ (as suggested by the size assumption \eqref{ConstantOrder}) and $M_{b2} =  K M_{b1}$ for $K$ sufficiently large to ensure that the fourth term in \eqref{ineq:wc2-a} and second term in \eqref{ineq:wb1} are consistent (it is important that $K$ is independent of $R$). 
Next, we set $\eps$ sufficiently small to ensure that the $O(\eps^2)$ terms are consistent. Next, we set $R$ sufficiently small to ensure that all remaining terms, except for the initial data term, are consistent. (Note this requires $R \ln R^{-1} \ll \eps$ and so $R$ is small relative to $\eps$, which is one of the reasons it is important to quantify $R$ dependence.)  
Finally, we set $T$ sufficiently small to ensure that both the hypothesis \eqref{Size} is satisfied and the initial data term is consistent (using Lemma~\ref{lem:Init}). This completes the proof that $\mathcal{Q}:B_{\eps,T,R,M_\ast} \mapsto B_{\eps,T,R,M_\ast}$. 

The same choices, possibly by adjusting \(M_c,K\) larger and $\eps,R,T$ smaller, also ensures that the mapping $\mathcal Q$ is a contraction. 
The theorem hence follows by the Banach fixed point. 
\end{proof} 

In Section~\ref{sec:MildlyIrritating} we prove that the solution obtained from Theorem~\ref{thm:FixedPt} is indeed a mild solution in the sense of Definition~\ref{milddef}. In Section \ref{sec:Unique} we prove the uniqueness result that follows from Theorem \ref{thm:FixedPt}; unlike in the straight filament case, this requires a short argument, as in order to apply the Banach fixed point, one must be able to decompose into $\omega^{c1},\omega^{c2},\omega^{b1},\omega^{b2}$ with the suitable a priori estimates. 
See therein for details. 

\section{Nonlinear estimates with curvature} \label{sec:Curve2}
In this section we prove Proposition~\ref{prop:AllAs}. Here we will only prove the a priori estimates as the estimates for the differences are similar.

Throughout this section we will assume that the hypothesis of Proposition~\ref{prop:AllAs}, i.e. \(\beta\in(0,\frac14)\), \(m \geq 2\) are fixed, \(0<R\ll1\) is sufficiently small and the constants \(T,R,M_*\) satisfy the inequalities \eqref{Size}, \eqref{ConstantOrder}. We will also allow constants throughout this section to implicitly depend on (admissible) values of \(\alpha,\beta,m\).

\begin{rem} \label{rem:Rsimp}
Note that obtaining the optimal powers of $R$ is not important in terms containing positive powers of $T$ and similarly, as long as the power is positive, the exact power of $T$ is not important either. Accordingly, we have not always endeavored to maintain the optimal scalings in $R$ or $T$.
\end{rem}

\subsection{Preliminary estimates}
Before starting the estimates of the contraction mapping $\mathcal{Q}$, we first outline some of the basic estimates on the initial data, the solution, and various estimates on the Biot-Savart law.

The first lemma provides the estimates coming from the initial data contribution:
\begin{lem}[Initial data bounds] \label{lem:Init}
If \(\mu^b\in W^{1,\frac{12}{11}}\), for all $R > 0$ fixed, we have the estimate
\begin{align}
\lim\limits_{T\searrow0}\norm{e^{t\Delta}\mu^b}_{\bF_b} & =0 \label{InitBound1}
\end{align}
\end{lem}
\begin{proof}
From the mapping properties of the heat operator in \(\R^3\) we have
\[
\|e^{t\Delta}\mu^b\|_{L^{4/3}_y}\lesssim t^{-\frac14}\|\mu^b\|_{L^{12/11}_y}.
\]
The estimate then follows from the density of \(\Schwartz\) in \(L^{\frac{12}{11}}_y\).
\end{proof}

Next we establish several auxiliary bounds for the vorticity in the straightened coordinates:
\begin{lemma}[Vorticity bounds in straightened coordinates] \label{lem:bEmbeds-1}
We have the following estimates for \(0<t\leq T\):
\begin{itemize} 

\item[(a)] \underline{\(B_zL^{4/3}_x\) bounds}
\begin{align}
t^{\frac14}\|\eta^g\|_{B_zL^{4/3}_x}&\lesssim 1,\label{i:Basic43bd}\\
t^{\frac14}\|\eta^{c1}\|_{B_zL^{4/3}_x} &\lesssim \|\omega\|_{X_T},\label{4/3-bnd}\\
t^{\frac14}\|\eta^{c2}\|_{B_zL^{4/3}_x} + t^{\frac14}\|Q_\Phi^{-1}(\widetilde\chi_{8R}\,\omega^{c2})\|_{B_zL^{4/3}_x} &\lesssim M_c^{-1}\|\omega\|_{X_T},\label{c2-I3-bnd}\\
t^{\frac14}\| \eta^{b1} \|_{B_zL^{\frac43}_x} &\lesssim M_{b_1}^{-1}\norm\omega_{X_T}, \label{Embedding3}\\
t^{\frac14}\|\eta^{b2}\|_{B_zL^{4/3}_x} + t^{\frac14}\|Q_\Phi^{-1}(\widetilde\chi_{4R}\,\omega^{b2})\|_{B_zL^{\frac43}_x} &\lesssim M_{b_2}^{-1}\norm\omega_{X_T}.\label{Embedding1}
\end{align}
\item[(b)] \underline{\(B_z L^1_{x}\) estimates near the core.}
\begin{align}
\norm{\chi_{\frac R4}\, \eta^{c1}}_{B_z L^1_{x}} & \lesssim \norm{\omega}_{X_T}\label{i:l1-a}\\
\norm{\chi_{\frac R4}\, \eta^{c2}}_{B_z L^1_{x}} & \lesssim M_c^{-1}\norm{\omega}_{X_T}.\label{i:l1-b}
\end{align}
\item[(c)] \underline{\(L^2_{x,z}\) bounds.}
\begin{align}
t^{\frac{1}{2}}\norm{\<t^{-\frac12}x\>^m\eta^{c1}}_{L^2_{x,z}} & \lesssim \norm{\omega}_{X_T}\label{i:eta-c1-l2-ugh}\\
t^{\frac{1}{2}}\norm{\<t^{-\frac12}x\>^m\eta^{c2}}_{L^2_{x,z}} & \lesssim M_c^{-1}\norm{\omega}_{X_T},\\
\norm{(1-\chi_{\frac R4}) \eta^{c1}}_{L^2_{x,z}} & \lesssim T^{\frac{m-1}2}R^{-m} \norm{\omega}_{X_T}.\label{i:c1-sep-l2}
\end{align}
\end{itemize}
\end{lemma}
\begin{proof}

\begin{itemize}
\item[(a)] The estimate \eqref{i:Basic43bd} follows from the explicit expression for \(\eta^g\). The estimates \eqref{4/3-bnd}, \eqref{c2-I3-bnd} follow from H\"older's inequality, using that \(m>1\). For the estimate \eqref{Embedding3} we apply the Sobolev embedding of Lemma \ref{Sobolev} to obtain
\[
t^{\frac14}\|\eta^{b1}\|_{B_zL^{\frac43}_x}\lesssim t^{\frac14}\|\eta^{b1}\|_{L^{\frac{4}{3}}}^{\frac14}\|\partial_z\eta^{b1}\|_{L^{\frac43}}^{\frac34}\lesssim \|\eta^{b1}\|_{\bN_b^\beta}.
\]
The estimate \eqref{Embedding1} is similar after using the estimates for the change of coordinates in Lemma~\ref{lem:Coefs}.
\item[(b)] From H\"older's inequality we have 
\[
\|\chi_{\frac R4}\eta^{c1}\|_{B_zL^1_x}\lesssim t^{\frac12}\|\<t^{-\frac12}x\>^m\eta^{c1}\|_{B_zL^2_x},
\]
where we have used that \(m\geq 2\).
\item[(c)] These follow from the embedding \(B_zL^2_x\subset L^\infty_zL^2_x\) and H\"older's inequality, where we recall that the curve has length \(2\pi\).
\end{itemize}
\end{proof}

Next we have estimates for the vorticity in the physical coordinates:
\begin{lemma}[Vorticity bounds in physical coordinates] \label{lem:bEmbeds-2}
We have the following estimates for \(0<t\leq T\):
\begin{itemize} 

\item[(a)] \underline{\(L^1_y\) estimates near the core.}
\begin{align}
\|\widetilde\chi_{\frac R4}\, \omega^{c1}\|_{L^1_y} &\lesssim \norm{\omega}_{X_T},\label{c2-I4-bnd-huh}\\
\|\widetilde\chi_{\frac R4}\, \omega^{c2}\|_{L^1_y} &\lesssim M_c^{-1}\norm{\omega}_{X_T}. \label{c2-I4-bnd-huh-1}
\end{align}
\item[(b)] \underline{\(L^2_y\) estimates for the core.}
\begin{align}
t^{\frac{1}{2}}\norm{\<t^{-\frac12}\bd\>^m \omega^{c1}}_{L^2_y} & \lesssim  \norm{\omega}_{X_T},\label{c1-l2y}\\
t^{\frac12}\norm{\<t^{-\frac12}\bd\>^m\widetilde\chi_{8R}\,\omega^{c2}}_{L^2_y} + t^{\frac12}\norm{\<t^{-\frac12}\bd\>^m(1 - \widetilde\chi_{6R})\omega^{c2}}_{L^3_y} & \lesssim M_c^{-1}\|\omega\|_{X_T}\label{c2-l2y},\\
\norm{(1-\widetilde{\chi}_{\frac R4}) \omega^{c1}}_{L^2_y} & \lesssim  T^{\frac{m-1}2}R^{-m}\norm{\omega}_{X_T},\label{c2-I4-bnd-also}\\
\norm{(1-\widetilde{\chi}_{\frac R4}) \omega^{c2}}_{L^2_y} & \lesssim  M_c^{-1} T^{\frac{m-1}{2}} R^{-m} \norm{\omega}_{X_T}.\label{c2-I4-bnd}
\end{align}
\item[(c)] \underline{\(L^2_y\) estimates for the background.} 
\begin{align}
t^{\frac14}\|\omega^{b1}\|_{L^2_y} &\lesssim M_{b_1}^{-1}\|\omega\|_{X_T}, \label{Embedding4}\\
t^{\frac14}\|\omega^{b2}\|_{L^2_y} &\lesssim M_{b_2}^{-1}\|\omega\|_{X_T}. \label{Embedding2}
\end{align}
\end{itemize}
\end{lemma}
\begin{proof}~
\begin{itemize}
\item[(a)] These follow from the estimates \eqref{i:l1-a}, \eqref{i:l1-b} using Lemma~\ref{lem:Coefs} to bound the change of variables.
\item[(b)] The estimate \eqref{c1-l2y} follows from the estimate \eqref{i:eta-c1-l2-ugh} using Lemma~\ref{lem:Coefs} to bound the change of coordinates. Similarly, the estimate \eqref{c2-I4-bnd-also} follows from the estimate \eqref{i:c1-sep-l2}.

For \eqref{c2-l2y}, notice first by the Lemma~\ref{lem:Coefs} and the definition of $\bN_c^0$, there holds 
\begin{align*}
t^{\frac{1}{2}}\norm{\brak{t^{-\frac{1}{2}}\bd}^{m} \widetilde{\chi}_{8R}\, \omega^{c2}}_{L^2_{y}} \lesssim M_c^{-1}\norm{\omega}_{X_T}. 
\end{align*}
The estimate then follows from the definition of the \(\bF_c\). 

Finally we consider the estimate \eqref{c2-I4-bnd}. 
For \(\widetilde\chi_{6R}\omega^{c2}\) the estimate follows as in \eqref{c2-I4-bnd-also}. Away from the filament, we have
\begin{align*}
\norm{(1-\widetilde\chi_{6R})\omega^{c2}}_{L^2_y} & \lesssim \norm{(1-\widetilde\chi_{6R})\omega^{c2} \brak{t^{-\frac{1}{2}} \bd}^m }_{L^3_y} \norm{(1-\widetilde\chi_{R})\brak{t^{-\frac{1}{2}} \bd}^{-m} }_{L^6_y} \\ 
& \lesssim M_c^{-1} t^{\frac{m-1}{2}} R^{\frac{1}{3}-m}  \norm{\omega}_{X_T},   
\end{align*}
and hence the estimate follows from the fact that \(0<R\leq1\).
\item[(c)] We observe that by Sobolev embedding and interpolation
\[
\|f\|_{L^2_y}\lesssim \|f\|_{L^{\frac43}_y}^{\frac14}\|\nabla f\|_{L^{\frac43}_y}^{\frac34}.
\]
The estimate \eqref{Embedding2} then follows from the definition of the \(\bF_b\) norm and the estimate \eqref{Embedding4} using Lemma~\ref{lem:Coefs} to bound the change of variables.
\end{itemize}
\end{proof}

Next, we prove estimates for the velocity:
\begin{lemma}[Velocity bounds in straightened coordinates] \label{lem:vgests-1}
The following estimates hold for all \(0<t\leq T\):
\begin{itemize}
\item[(a)] \underline{\(B_zL^4_x\cap B_z\dot W^{1,4/3}_x\) estimates.}
\begin{align}
t^{\frac14}\norm{v^g}_{B_z L^{4}_x} + t^{\frac14}\norm{\grad v^g}_{B_z L^{4/3}_x}& \lesssim 1,\label{i:vg-nobar}\\
t^{\frac14}\norm{v^{c1}}_{B_z L^{4}_x} + t^{\frac14}\norm{\grad v^{c1} }_{B_z L^{4/3}_x} &\lesssim \norm{\omega}_{X_T}, \label{ineq:uc1}\\ 
t^{\frac14}\norm{v^{c2}}_{B_z L^{4}_x} + t^{\frac14} \norm{\grad v^{c2} }_{B_z L^{4/3}_x} &\lesssim M_c^{-1}\|\omega\|_{X_T},  \label{ineq:uc2}\\
t^{\frac14}\norm{v^{b1}}_{B_z L^{4}_x} + t^{\frac14} \norm{\grad v^{b1}}_{B_z L^{4/3}_x} + t^{\frac14}\norm{P_\Phi^{-1}(\widetilde\chi_{7R}\, u^{b1})}_{B_zL^4_x}& \lesssim  M_{b1}^{-1}\norm{\omega}_{X_T}, \label{ineq:vb1} \\
t^{\frac14}\norm{v^{b2}}_{B_z L^{4}_x} + t^{\frac14} \norm{\grad v^{b2}}_{B_z L^{4/3}_x}  + t^{\frac14}\norm{P_\Phi^{-1}(\widetilde\chi_{7R}\, u^{b2})}_{B_zL^4_x} & \lesssim R^{-\frac34}M_{b2}^{-1}\norm{\omega}_{X_T}.  \label{ineq:vb2}
\end{align}

\item[(b)] \underline{Estimates for the difference between approximate and actual velocities.}
\begin{align}
 t^{\frac14} \norm{v^g - \chi_R\bar v^g}_{B_zL^4_x} + t^{\frac14} \norm{\grad \left(v^g - \chi_R\bar v^g\right)}_{B_zL^{4/3}_x}  \lesssim T^{\frac14}R^{-\frac12} + R\ln R^{-1},\label{i:vgdiff}  \\
t^{\frac14}\norm{v^{c1} - \chi_R(-\Delta)^{-1}\nabla\times \eta^{c1}}_{B_zL^4_x} \lesssim \left(T^{\frac14}R^{-\frac12} + R\ln R^{-1}\right)\|\omega\|_{X_T}.\label{i:vcdiff}
\end{align}

\item[(c)] \underline{Estimates away from the core.}
\begin{align}
\norm{(1-\chi_R) \bar v^g}_{B_z L^{4}_x} + \norm{\grad ((1-\chi_R) \bar v^g)}_{B_z L^{4/3}_x} & \lesssim R^{-\frac12} ,\label{i:vgsep}\\
\|(1-\chi_R)v^g\|_{B_zL^4_x} + \|P_\Phi^{-1}(\widetilde\chi_{7R}(1 - \widetilde\chi_R)u^g) \|_{B_zL^4_x} &\lesssim R^{-\frac12}, \label{VelocityBzL4-g}\\
\|(1-\chi_R)v^{c1}\|_{B_zL^4_x} + \|P_\Phi^{-1}(\widetilde\chi_{7R}(1 - \widetilde\chi_R)u^{c1}) \|_{B_zL^4_x} &\lesssim R^{-\frac12} \|\omega\|_{X_T},\label{VelocityBzL4}\\
\|(1-\chi_R)v^{c2}\|_{B_zL^4_x} + \|P_\Phi^{-1}(\widetilde\chi_{7R}(1 - \widetilde\chi_R)u^{c2}) \|_{B_zL^4_x} &\lesssim R^{-\frac12}M_c^{-1}\|\omega\|_{X_T}.\label{Spatula}
\end{align}
\end{itemize}
\end{lemma}
\begin{rem}
We note that from the assumptions \eqref{Size}, \eqref{ConstantOrder} we obtain from part (a):
\begin{align}
t^{\frac14}\|v - \alpha v^g\|_{B_zL^4_x} + t^{\frac14}\|\nabla ( v - \alpha v^g )\|_{B_zL^{4/3}_x} &\lesssim \|\omega\|_{X_T},\label{vDiffSummary}\\
t^{\frac14}\|v - \alpha v^g - v^{c1}\|_{B_zL^4_x} + t^{\frac14}\|\nabla ( v - \alpha v^g - v^{c1} )\|_{B_zL^{4/3}_x} &\lesssim M_c^{-1}\|\omega\|_{X_T}.\label{vDiffDiffSummary}
\end{align}
\end{rem}
\begin{proof}~

\begin{itemize}

\item[(a)] For the estimates \eqref{i:vg-nobar}, \eqref{ineq:uc1}, \eqref{ineq:vb1} we apply the estimate \eqref{ineq:BSbasic} together with the estimates \eqref{i:Basic43bd}, \eqref{4/3-bnd}, \eqref{Embedding3} for \(\eta^g\), \(\eta^{c1}\), \(\eta^{b1}\) respectively.

For the estimate \eqref{ineq:uc2} we first observe that by Sobolev embedding (in \(x\)) we have
\[
\|v^{c2}\|_{B_zL^4_x} \lesssim \|\nabla v^{c2}\|_{B_zL^{4/3}_x},
\]
so it suffices to prove the estimate for \(\nabla v^{c2}\).
Next, we decompose
\[
v^{c2}  = \underbrace{P_\Phi^{-1}(\widetilde\chi_R (-\Delta)^{-1}\grad \times ( \widetilde{\chi}_{4R}\, \omega^{c2}))}_{I_1} + \underbrace{P_\Phi^{-1}(\widetilde\chi_R (-\Delta)^{-1}\grad \times( (1-\widetilde{\chi}_{4R}) \omega^{c2})}_{I_2}. 
\]
For $I_1$ we apply the estimate \eqref{ineq:BSbasic} to obtain
\[
\|\nabla I_1\|_{B_zL^{4/3}_x}\lesssim \|Q_\Phi^{-1}(\widetilde\chi_{4R}\,\omega^{c2})\|_{B_zL^{4/3}_x},
\]
and the estimate follows from \eqref{c2-I3-bnd}. For $I_2$ we first apply Lemma \ref{lem:Sobolev} and H\"older's inequality to bound
\[
\norm{ \nabla I_2}_{B_z L^{4/3}_x} \lesssim \norm{\nabla I_2}_{L^{4/3}_{x,z}}^{\frac14} \norm{\partial_z\nabla I_2}_{L^{4/3}_{x,z}}^{\frac34} \lesssim R^{\frac32}\norm{\nabla I_2}_{L^\infty_{x,z}}^{\frac14} \norm{\partial_z\nabla I_2}_{L^\infty_{x,z}}^{\frac34}. 
\]
Applying Lemma~\ref{lem:Coefs} to control the change of variables we may then apply Young's inequality with the separation of supports, recalling that \(0<R\leq 1\), to bound, with the help of~\eqref{c2-I4-bnd}
\begin{align*}
\|\nabla I_2\|_{L^\infty_{x,z}} &\lesssim \left(R^{-1}\|\chi_{\{|y|\gtrsim R\}}K\|_{L^2_y} + \|\chi_{\{|y|\gtrsim R\}}\nabla K\|_{L^2_y}\right) \|(1 - \widetilde\chi_{4R})\omega^{c2}\|_{L^2_y}\\
& \lesssim R^{-\frac32}\|(1 - \widetilde\chi_{R/4})\omega^{c2}\|_{L^2_y}\\
&\lesssim T^{\frac{m-1}2}R^{-\frac32 - m}M_c^{-1} \|\omega\|_{X_T},
\end{align*}
where \(K(y) = \frac1{4\pi}\frac{y\times}{|y|^3}\) denotes the Biot-Savart kernel and \(\chi_{\{|y|\gtrsim R\}}\) is a smooth bump function adapted to the set \(\{|y|\gtrsim R\}\).
Similarly,
\[
\norm{\partial_z \nabla I_2}_{L^\infty_{x,z}}  \lesssim T^{\frac{m-1}2}R^{-\frac52 - m}M_c^{-1}\|\omega\|_{X_T}. 
\]
Combining these estimates we obtain the bound for \(I_2\),
\[
\|\nabla I_2\|_{B_zL^{4/3}_x}\lesssim T^{\frac{m-1}2}R^{-\frac34-m}M_c^{-1}\|\omega\|_{X_T}.
\]
Finally, using the assumption \eqref{Size} and that \(m\geq 2\) we see that \(T^{\frac{m-1}2}R^{-\frac34-m}\lesssim1\).

For the estimate \eqref{ineq:vb2} we decompose 
\begin{align*}
v^{b2} & = \underbrace{P_\Phi^{-1}(\widetilde\chi_R\grad \times (-\Delta)^{-1} \widetilde{\chi}_{4R} \omega^{b2})}_{I_3} + \underbrace{P_\Phi^{-1}(\widetilde\chi_R\grad \times (-\Delta)^{-1} (1-\widetilde{\chi}_{4R}) \omega^{b2})}_{I_4}.  
\end{align*}
The $I_3$ piece is treated as in \eqref{ineq:uc2} using the estimates \eqref{ineq:BSbasic}, \eqref{Embedding1}. For the \(I_4\) contribution, we use the same proof as in \eqref{ineq:uc2}, replacing the estimate \eqref{c2-I4-bnd} by the estimate \eqref{Embedding2}. For the remaining term on LHS\eqref{ineq:vb2} we introduce a non-negative, radial function \(\rho\) that is identically \(1\) on \(\{|x|\leq \frac{29}2R\}\), supported on \(\{|x|\leq \frac{31}2 R\}\), and again denote \(\widetilde\rho\circ\Phi = \rho\). We then decompose
\[
P_\Phi^{-1}(\widetilde\chi_{7R} u^{b2}) = P_\Phi^{-1}(\widetilde\chi_{7R}\grad \times (-\Delta)^{-1} \widetilde\rho\, \omega^{b2}) + P_\Phi^{-1}(\widetilde\chi_{7R}\grad \times (-\Delta)^{-1} (1-\widetilde\rho) \omega^{b2}),
\]
and observe that due to the construction of \(\rho\), an identical argument applies.
\item[(b)] Note that 
\[
v^g - \chi_R\bar v^g = \chi_R \left(P_\Phi^{-1}(-\Delta)^{-1}\nabla\times Q_\Phi - (-\Delta)^{-1}\nabla\times\right)\eta^g.
\]
The estimate \eqref{i:vgdiff} then follows from Proposition~\ref{prop:ApproxBS}. The estimate \eqref{i:vcdiff} is similar.

\item[(c)] The first two estimates follow from the explicit expression for \(\bar v^g\). For the remaining bounds, we first observe that by definition
\[
\|(1 - \chi_R)v^*\|_{B_zL^4_x}\lesssim \|P_\Phi^{-1}(\widetilde\chi_{7R}(1 - \widetilde\chi_R)u^*)\|_{B_zL^4_x}.
\]
For \eqref{VelocityBzL4}, we decompose
\begin{align*}
P_\Phi^{-1}\left(\widetilde\chi_{7R}(1 -\widetilde\chi_R)u^{c1}\right) &= \underbrace{(1 - \chi_R)P_\Phi^{-1}\left(\widetilde\chi_{7R}(-\Delta)^{-1}\nabla\times Q_\Phi((\chi_{2R } - \chi_{\frac R4})\eta^{c1})\right)}_{I_5}\\
&\quad + \underbrace{(1 - \chi_R)P_\Phi^{-1}\left(\widetilde\chi_{12R}(-\Delta)^{-1}\nabla\times Q_\Phi(\chi_{\frac R4}\eta^{c1})\right)}_{I_6}.
\end{align*}
For \(I_5\) we apply the estimate \eqref{ineq:BSbasic} with H\"older's inequality to obtain
\begin{align*}
\|I_5\|_{B_zL^4_x} &\lesssim \|(1 - \chi_{\frac R4})\eta^{c1} \|_{B_zL^{4/3}_x} \lesssim R^{-\frac12} \|\eta^{c1}\|_{\bN^{\beta}_c}.
\end{align*}
For \(I_6\) we instead use the separation of the supports to apply the estimate \eqref{ineq:BSbasicSeparation} and \eqref{i:l1-a} to obtain
\begin{align*}
\|I_6\|_{B_zL^4_x}&\lesssim R^{-\frac12}\|\eta^{c1} \|_{B_zL^1_x} \lesssim R^{-\frac12}\|\omega^{c1}\|_{\bN^{\beta}_c}. 
\end{align*}
The proof of the estimate \eqref{VelocityBzL4-g} is identical.

For the estimate \eqref{Spatula} we take \(\rho\) to be defined as in part (a). We then decompose
\begin{align*}
P_\Phi^{-1}\left(\widetilde\chi_{7R}(1 -\widetilde\chi_R)u^{c2}\right) &= \underbrace{P_\Phi^{-1}\Big(\widetilde\chi_{7R}(1 -\widetilde\chi_R)(-\Delta)^{-1}\nabla\times\left(\widetilde\rho\, \omega^{c2}\right)\Big)}_{I_7}\\
&\quad + \underbrace{P_\Phi^{-1}\Big(\widetilde\chi_{7R}(1 -\widetilde\chi_R)(-\Delta)^{-1}\nabla\times\left((1 - \widetilde\rho)\omega^{c2}\right)\Big)}_{I_8}.
\end{align*}
For \(I_7\) we argue as in the proof of \eqref{VelocityBzL4} using that
\[
\|P_\Phi^{-1}(\widetilde\rho\, \omega^{c2})\|_{\bN_c^0} \lesssim \|P_\Phi^{-1}(\widetilde\chi_{8R}\, \omega^{c2})\|_{\bN_c^0},
\]
to obtain
\[
\|I_7\|_{B_zL^4_x}\lesssim R^{-\frac12}M_c^{-1}\|\omega\|_{X_T}.
\]
For \(I_8\) we proceed identically to the bound for \(I_2\), using Sobolev embedding, Lemma~\ref{lem:Sobolev}, H\"older's inequality and the separation of supports of \(\widetilde\chi_{7R}\) and \((1 - \widetilde\rho)\) to bound
\begin{align*}
\|I_8\|_{B_zL^4_x}\lesssim \|\nabla I_8\|_{B_zL^{4/3}_x}\lesssim R^{-\frac32}\|(1 - \widetilde\rho)\omega^{c2}\|_{L^2_y}\lesssim T^{\frac{m-1}{2}} R^{-\frac32-m} M_c^{-1} \norm{\omega}_{X_T},
\end{align*}
where the final estimate follows from \eqref{c2-I4-bnd}. The estimate \eqref{Spatula} then follows from the assumption \eqref{Size} and that \(m\geq 2\).
\end{itemize}
\end{proof}

To conclude this subsection we prove estimates for the velocity in physical coordinates:

\begin{lemma}[Velocity bounds in physical coordinates] \label{lem:vgests-2}
The following estimates hold for all \(0<t\leq T\):
\begin{itemize}
\item[(a)] \underline{Estimates on the background.}
\begin{align}
R^{\frac{1}{4}}t^{\frac14}\norm{u^{b1}}_{L^{12}_y} + t^{\frac14}\norm{u^{b1}}_{L^6_y} + t^{\frac14}\norm{\grad u^{b1}}_{L^2_y} & \lesssim M_{b1}^{-1}\|\omega\|_{X_T}, \label{ineq:vbLp-1}\\
R^{\frac{1}{4}}t^{\frac14}\norm{u^{b2}}_{L^{12}_y} + t^{\frac14}\norm{u^{b2}}_{L^6_y} + t^{\frac14}\norm{\grad u^{b2}}_{L^2_y} & \lesssim M_{b2}^{-1}\|\omega\|_{X_T}. \label{ineq:vbLp-2}
\end{align}

\item[(b)] \underline{Estimates away from the core.}
\begin{align}
R^{\frac14}\norm{(1-\widetilde\chi_R)u^g}_{L^{12}_y} + \norm{(1-\widetilde\chi_R)u^g}_{L^6_y} + \norm{\grad ( (1-\widetilde\chi_R) u^g)}_{L^2_y}  &\lesssim R^{-\frac32} ,\label{velgL6}\\
R^{\frac14}\norm{(1-\widetilde\chi_R)u^{c1}}_{L^{12}_y} + \|(1 - \widetilde\chi_R)u^{c1}\|_{L^6_y} + \|\nabla\left((1 - \widetilde\chi_R)u^{c1}\right)\|_{L^2_y} &\lesssim R^{-\frac32}\|\omega\|_{X_T},\label{VelocityL6-1}\\
\hspace{1cm}R^{\frac14}\norm{(1-\widetilde\chi_R)u^{c2}}_{L^{12}_y} + \|(1 - \widetilde\chi_R)u^{c2}\|_{L^6_y} + \|\nabla\left((1 - \widetilde\chi_R)u^{c2}\right)\|_{L^2_y} &\lesssim R^{-\frac32} M_c^{-1}\|\omega\|_{X_T}. \label{VelocityL6-2}
\end{align}
\end{itemize}
\end{lemma}
\begin{rem}
Again we note that from the assumptions \eqref{Size}, \eqref{ConstantOrder} we obtain from part (a) and (b):
\begin{equation}
\begin{aligned}
&R^{\frac14}t^{\frac14}\|(1 - \widetilde\chi_R)(u - \alpha u^g)\|_{L^{12}_y} + t^{\frac14}\|(1 - \widetilde\chi_R)(u - \alpha u^g)\|_{L^6_y} + t^{\frac14}\|\nabla ((1 - \widetilde\chi_R)(u -\alpha u^g) )\|_{L^2_y}\\
&\qquad \lesssim \left(T^{\frac14}R^{-\frac32} + M_{b1}^{-1}\right)\|\omega\|_{X_T},
\end{aligned}
\label{uDiffSummary}\end{equation}
which we note has additional smallness over the estimate \eqref{vDiffSummary} for the velocity near the core, provided we choose \(T\ll_R1\) and \(M_{b1}\gg1 \).
\end{rem}
\begin{proof}~

\begin{itemize}
\item[(a)] We observe that by Sobolev embedding and the boundedness of Riesz transforms on \(L^2_y\) we have
\[
\|u^{b1}\|_{L^6_y}\lesssim \|\nabla u^{b_1}\|_{L^2_y}\lesssim \|\omega^{b1}\|_{L^2_y},
\]
and by the Hardy-Littlewood-Sobolev Lemma and Sobolev embedding,
\[
\|u^{b1}\|_{L^{12}_y}\lesssim \|\omega^{b1}\|_{L^{12/5}_y}\lesssim \|\nabla\omega^{b1}\|_{L^{4/3}_y}.
\]
The estimate \eqref{ineq:vbLp-1} then follows from the estimate \eqref{Embedding4}, the definition of the \(\bN_b^\beta\) norm and Lemma~\ref{lem:Coefs}. Similarly, the estimate \eqref{ineq:vbLp-2} follows from \eqref{Embedding2} and the definition of the \(\bF_b\) norm.
\item[(b)] For estimates \eqref{velgL6}--\eqref{VelocityL6-2}, we first note that by Sobolev embedding, the $L^6_y$ estimates follow from the \(L^2_y\) gradient bounds.

Let us first prove these $L^2_y$ gradient bounds and then return to the $L^{12}_y$ estimates. For \(*=g,c1,c2\) we decompose
\[
(1 - \widetilde\chi_R)u^* = \underbrace{(1 - \widetilde\chi_R)(-\Delta)^{-1}\nabla\times Q_\Phi((\chi_{2R} - \chi_{\frac R4})\eta^*)}_{I_1^*} + \underbrace{(1 - \widetilde\chi_R)(-\Delta)^{-1}\nabla\times Q_\Phi(\chi_{\frac R4}\eta^*)}_{I_2^*}.
\]
For the first of these we use H\"older's inequality followed by the Hardy-Littlewood-Sobolev lemma and boundedness of Riesz transforms on \(L^2\) with Lemma~\ref{lem:Coefs} to control the change of variables to obtain
\begin{align*}
\|\nabla I_1^*\|_{L^2_y} &\lesssim \|\nabla(-\Delta)^{-1}\nabla\times Q_\Phi((\chi_{2R} - \chi_{\frac R4})\eta^*)\|_{L^2_y} + R^{-1}\|(-\Delta)^{-1}\nabla\times Q_\Phi((\chi_{2R} - \chi_{\frac R4})\eta^*)\|_{L^2_y}\\
&\lesssim \|\nabla(-\Delta)^{-1}\nabla\times Q_\Phi((\chi_{2R} - \chi_{\frac R4})\eta^*)\|_{L^2_y} + \|(-\Delta)^{-1}\nabla\times Q_\Phi((\chi_{2R} - \chi_{\frac R4})\eta^*)\|_{L^6_y}\\
&\lesssim \|Q_\Phi((\chi_{2R} - \chi_{\frac R4})\eta^*)\|_{L^2_y}\\
&\lesssim\|(1 - \widetilde\chi_{\frac R4})\omega^*\|_{L^2_y}
\end{align*}
We may then apply the estimates \eqref{c2-I4-bnd-also}, \eqref{c2-I4-bnd} and
\[
\|(1 - \widetilde\chi_{\frac R4})\omega^g\|_{L^2_y}\lesssim_k T^{\frac{k-1}2}R^{-k}\quad\text{for any}\quad k\geq 0,
\]
in the cases \(*=c1,c2,g\) respectively. Using the hypothesis \eqref{Size} with the fact that \(m\geq 2\) it is then clear that \(T^{\frac{m-1}2}R^{-m}\leq R^{-\frac32}\).

For the second term we instead use Young's inequality to bound
\[
\|\nabla I_2^*\|_{L^2_y} \lesssim \left(R^{-1}\|\chi_{\{|y|\gtrsim R\}}K\|_{L^2_y} + \|\chi_{\{|y|\gtrsim R\}}\nabla K\|_{L^2_y}\right)\|Q_\Phi(\chi_{\frac R4}\eta^*)\|_{L^1_y}\lesssim R^{-\frac32}\|\widetilde \chi_{\frac R4}\omega^*\|_{L^1_y}.
\]
where \(K(y) = \frac1{4\pi}\frac{y\times}{|y|^3}\) is the \(3d\) Biot-Savart kernel and \(\chi_{\{|y|\gtrsim R\}}\) is a smooth bump function adapted to the set \(\{|y|\gtrsim R\}\). We may then apply the estimates \eqref{c2-I4-bnd-huh}, \eqref{c2-I4-bnd-huh-1} and
\[
\|\widetilde \chi_{\frac R4}\omega^g\|_{L^1_y}\lesssim 1,
\]
in the cases \(*=c1,c2,g\) respectively.

Next, we turn to the $L^{12}_y$ estimates. For the \(I_2^*\) piece we proceed similarly to before, using Young's inequality to estimate
\[
\|I_2^*\|_{L^{12}_y}\lesssim \|\chi_{\{|y|\gtrsim R\}}K\|_{L^{12}_y}\|Q_\Phi(\chi_{\frac R4}\eta^*)\|_{L^1_y}\lesssim R^{-\frac74}\|\widetilde \chi_{\frac R4}\omega^*\|_{L^1_y}.
\]

For the \(I_1^*\) piece we further decompose as
\[
I_1^* = (1 - \widetilde\chi_{6R})I_1^* + \widetilde\chi_{6R}I_1^*.
\]
For the first piece we proceed similarly to the \(I_2^*\) piece, using the separation of supports and H\"older's inequaity to bound
\[
\|(1 - \widetilde\chi_{6R})I_1^*\|_{L^{12}_y}\lesssim \|\chi_{\{|y|\gtrsim R\}}K\|_{L^{12/7}_y}\|Q_\Phi((\chi_{2R} - \chi_{\frac R4})\eta^*)\|_{L^2_y}\lesssim R^{-\frac14} \|(1 - \widetilde\chi_{\frac R4})\omega^*\|_{L^2_y}.
\]
For the second piece, we instead use Lemmas~\ref{lem:Coefs},~\ref{lem:KernelBounds} with the Hardy-Littlewood-Sobolev and H\"older inequalities to change variables and bound
\[
\|\widetilde\chi_{6R}I_1^*\|_{L^{12}_y}\lesssim \|P_\Phi^{-1}(\widetilde\chi_{6R}I_1^*)\|_{L^2_zL^{12}_x}\lesssim \|(\chi_{2R} - \chi_{\frac R4})\eta^*\|_{L^2_zL^{12/7}_x}\lesssim R^{\frac16} \|(1 - \widetilde\chi_{\frac R4})\omega^*\|_{L^2_y}.
\]
\end{itemize}
\end{proof}

\subsection{Estimates on the core corrections, $\omega^c$} 
\subsubsection{Estimates on $\eta^{c1}$}
In this subsection we prove the estimate \eqref{ineq:wc1bd}.
For convenience, we first decompose \(f^{c1}\) into linear, non-linear and inhomogeneous parts as
\begin{align*}
f^{c1} & = f^{c1}_L + f^{c1}_N + f^{c1}_I, 
\end{align*}
where 
\begin{align*}
f^{c1}_L & := \alpha B[v^g - \bar v^g, \eta^{c1}] + \alpha B[v^g, \eta^{c2}] + B[v - \alpha v^g - (-\Delta)^{-1}\nabla\times \eta^{c1}, \alpha\eta^g],\\
f^{c1}_N & := B[v - \alpha v^g, \eta^{c1} + \eta^{c2}],
\end{align*}
and we recall from \eqref{c1InhomogeneousTerm} that
\[
f^{c1}_I = \alpha^2 B[v^g - \chi_R\bar v^g,\eta^g]. 
\]
As in the straight filament case (see \eqref{Bilg} above) under the self-similar coordinate transform
\begin{align*}
B[f,g] = \Div\left( f \otimes g - g \otimes f \right) \mapsto \oDiv \left( F \otimes G - G \otimes F \right) =: \overline{B}[F,G]. 
\end{align*}

As a direct corollary of Theorem~\ref{prop:LinearEstimates} we obtain the following lemma: 
\begin{lemma}
Taking \(\mu = \mu(\alpha)\) as in Theorem~\ref{prop:LinearEstimates}, the following estimate holds for all \(-\infty<s,\tau\leq \ln T\):
\begin{align}
\norm{\brak{\onabla}^{\beta} S(\tau,s) \overline{B}[V,H]}_{B_z L^2_\xi (m)} & \lesssim_{\beta,\alpha}\frac{e^{-\mu(\tau-s)}}{a(\tau-s)^{\frac34 + \frac\beta2}} \norm{V}_{B_zL^4_\xi}\norm{H}_{B_z L^2_\xi(m)}. \label{ineq:SB}
\end{align}
\end{lemma}
\begin{proof}
We observe that for any vector fields \(F,G\) we have \(\oDiv\ \overline B[F,G] = 0\). As a consequence, we may interpolate the estimates \eqref{SmoothedBound}, \eqref{SmoothedBound1} with \(p=\frac43\) and apply the product law~\eqref{productlaw} to obtain the estimate \eqref{ineq:SB}.
\end{proof}

Applying this bilinear estimate together with Lemmas~\ref{lem:bEmbeds-1},~\ref{lem:vgests-1}, we obtain the following lemma, of which the estimate \eqref{ineq:wc1bd} is a direct consequence (after using \eqref{Size} to bound \(T^{\frac14}R^{-\frac12}\lesssim R\ln R^{-1}\)):
\begin{lemma} 
We have the estimates:
\begin{align}
\norm{\brak{\onabla}^{\beta} \int_{-\infty}^\tau S(\tau,s) F^{c1}_L(s) ds}_{B_z L^2_\xi(m)} & \lesssim \left(T^{\frac14}R^{-\frac12} + R\ln R^{-1} + M_{c}^{-1}\right)\|\omega\|_{X_T},\label{fc-1}\\
\norm{\brak{\onabla}^{\beta} \int_{-\infty}^\tau S(\tau,s) F^{c1}_N(s) ds }_{B_z L^2_\xi(m)} & \lesssim\|\omega\|_{X_T}^2\label{fc-2} \\ 
\norm{\brak{\onabla}^{\beta} \int_{-\infty}^\tau S(\tau,s) F^{c1}_I(s) ds }_{B_z L^2_\xi(m)} &\lesssim T^{\frac14}R^{-\frac12} + R\ln R^{-1}\label{fc-3}
\end{align}
\end{lemma}
\begin{proof}
We prove the estimate \eqref{fc-1}, the estimates \eqref{fc-2}, \eqref{fc-3} are similar. We will apply the linear estimate \eqref{ineq:SB} with the velocity estimates of Lemma~\ref{lem:vgests-1}, noting that by rescaling
\[
\|V^*(\tau)\|_{B_zL^4_\xi} = t^{\frac14}\|v^*(t)\|_{B_zL^4_x},
\]
and the estimate
\[
\|H^{c1}\|_{B_zL^2_\xi(m)} + M_{c}\|H^{c2}\|_{B_zL^2_\xi(m)}\leq \|\omega\|_{X_T},
\]
which follows directly from the definition of the norm.

In particular we may apply the estimates \eqref{i:vgdiff}, \eqref{i:vgsep} to bound
\[
\|V^g - \bar V^g\|_{B_zL^4_\xi} \lesssim T^{\frac14}R^{-\frac12} + R\ln R^{-1},
\]
the estimate \eqref{i:vg-nobar} to bound
\[
\|V^g\|_{B_zL^4_\xi}\lesssim 1
\]
and the estimates \eqref{i:vcdiff}, \eqref{vDiffDiffSummary} to bound
\[
\|V - \alpha V^g - (-\oDelta)^{-1}\onabla\times H^{c1}\|_{B_zL^4_\xi}\lesssim \left(T^{\frac14}R^{-\frac12} + R\ln R^{-1} + M_{c}^{-1} \right)\|\omega\|_{X_T}.
\]
As a consequence, there holds
\begin{align*}
&\norm{\<\onabla\>^{\beta} \int_{-\infty}^\tau S(\tau,s) F^{c1}_L(s) ds }_{B_z L^2_\xi(m)}\\
&\qquad \lesssim
\left(T^{\frac14}R^{-\frac12} + R\ln R^{-1} + M_{c}^{-1}\right)\|\omega\|_{X_T} \left(\sup_\tau\int_{-\infty}^\tau\frac{e^{ -\mu(\tau-s)}}{a(\tau-s)^{\frac 34 + \frac\beta2}}\,ds\right)\\
&\qquad\lesssim \left(T^{\frac14}R^{-\frac12} + R\ln R^{-1} + M_{c}^{-1}\right)\|\omega\|_{X_T},
\end{align*}
where we note that the integral converges because \(\mu>0\) and \(0<\beta<\frac14\).
\end{proof} 

\subsubsection{Estimates on $\omega^{c2}$}
In this section we prove the estimates \eqref{ineq:wc2-a}, \eqref{ineq:wc2-b}. Again we start by decomposing into a linear, nonlinear and inhomogeneous piece,
\[
f^{c2} = f^{c2}_L + f^{c2}_N + f^{c2}_I,
\]
where (recalling the inhomogeneous term from \eqref{c2InhomogeneousTerm}), 
\begin{align*}
f^{c2}_L&:= B[(1 - \widetilde\chi_R)\alpha u^g,\omega^{c1} + \omega^{c2}] + B[(1 - \widetilde\chi_R)(u - \alpha u^g),\alpha \omega^g] + \mc E^c,\\
f^{c2}_N &:= B[(1 - \widetilde\chi_R)(u - \alpha u^g),\omega^{c1} + \omega^{c2}], \\ 
f^{c2}_I & = B[(1 - \widetilde\chi_R)\alpha u^g,\alpha \omega^g] + \alpha\mc E^g.
\end{align*}

We begin with the estimates on $\bF_c$. 

\begin{lem}\label{lem:ErrorHandling}
We have the estimates:
\begin{align}
\norm{\int_0^t e^{(t - s)\Delta} f^{c2}_L(s)\,ds}_{\bF_c}& \lesssim \left(T^{\frac14}R^{-2} + R^{-\frac12}M_{b1}^{-1}\right)\|\omega\|_{X_T},\\
\norm{\int_0^t e^{(t - s)\Delta} f^{c2}_N(s)\,ds}_{\bF_c}
&\lesssim \left( T^{\frac14}R^{-2} + R^{-\frac12}M_{b_1}^{-1} \right)\|\omega\|_{X_T}^2,\\
\norm{\int_0^t e^{(t - s)\Delta} f^{c2}_I(s)\,ds}_{\bF_c}&\lesssim T^{\frac 14}R^{-2}.
\end{align}
\end{lem}
\begin{proof}
We start by noting the following: if $x,y \in \mathbb{R}^2$, $0<s<t$ we have for any $c > 0$,
$$\left\langle \tfrac{x}{\sqrt t} \right\rangle^m e^{-c \frac{|x-y|^2}{t-s}} \lesssim_c \left\langle \tfrac{y}{\sqrt s} \right\rangle^{m},$$
which, combined with Lemma~\ref{lem:PhiDiff}, will allow us to freely pass the spatial weights through the heat propagator.

Using this we consider the error terms $\mc E^c$ and $\mc E^g$ that are supported in the set \(\{\bd \leq 4R\}\). As they are supported near the filament, there will be a significant gain from the separation of supports (see Figure~\ref{fig:c2supports}). From the definition we observe that
\[
Q_\Phi^{-1}\mc E^c = - [\Delta_\Phi,\chi_{2R}]\eta^{c1} - \chi_{2R}\left(\Delta_\Phi - \Delta\right)\eta^{c1}.
\]
From Lemmas~\ref{lem:Geometry},~\ref{lem:Coefs} we then see that it is possible to write
\begin{align} 
Q_\Phi^{-1}\mc E^c = \nabla_{x,z}^\gamma\left(C_\gamma \eta^{c1}\right), \label{eq:EcDecom}
\end{align}
where the summation is taken over multi-indices \(\gamma\in \N^3\) satisfying \(|\gamma|\leq 2\) and the smooth matrix-valued functions \(C_\gamma\) are supported in \(\{|x|\leq 4R\}\) and satisfy
\eq{i:cg-1}{
|\nabla_{x,z}^\mu C_2|\lesssim R^{1-|\mu|},\qquad |\nabla_{x,z}^\mu C_1|\lesssim R^{-1-|\mu|},\qquad|\nabla_{x,z}^\mu C_0|\lesssim R^{-2-|\mu|},
}
for all multi-indices \(\mu\in \N^3\), where we use \(C_2\) to denote all terms with \(|\gamma|=2\), etc. Due to the separation of the support of \(\mc E^c\) and $(1-\widetilde\chi_{6R})$ as well as the assumption that \(0<R\leq 1\) we have the estimate
\begin{align*}
\norm{\int_0^t e^{(t - s)\Delta} \cE^{c} ds}_{\bF_c}&=\sup\limits_{t\in[0,T]}t^{\frac12}\norm{\<t^{-\frac12}\bd\>^m(1 - \widetilde\chi_{6R})\int_0^t e^{(t - s)\Delta} \cE^{c} ds}_{L^3_y}\\
& \lesssim \sup\limits_{t\in[0,T]}t^{\frac{1}{2}} \int_0^t R^{-2}(t-s)^{-\frac14} \norm{\brak{s^{-\frac{1}{2}}\bd}^{m} \omega^{c1}}_{L^2_y} ds  \\ 
& \lesssim T^{\frac 34}R^{-2} \norm{\omega}_{X_T}.
\end{align*}
where the last inequality follows from the estimate \eqref{c1-l2y}. For clarity, we illustrate this bound in more detail with the \(\gamma = 0\) term:
\begin{align*}
\norm{\<t^{-\frac12}\bd\>^m\int_0^t e^{(t - s)\Delta} Q_\Phi\left(C_0 \eta^{c1}\right)ds}_{L^3_y} &\lesssim \int_0^t(t - s)^{-\frac14}e^{-c\frac{R^2}{t-s}}\norm{\<s^{-\frac12}x\>^mC_0 \eta^{c1}}_{L^2_{x,z}}\,ds\\
&\lesssim \int_0^t R^{-2}(t - s)^{-\frac14}e^{-c\frac{R^2}{t-s}}\norm{\<s^{-\frac12}x\>^m\eta^{c1}}_{L^2_{x,z}}\,ds.
\end{align*}

Consider next the remaining terms in $f^{c2}_L$. First turn to the contribution of $\omega^{c2}$. 
Here we apply the estimate \eqref{velgL6} for \(u^g\) to bound
\begin{align*}
&\norm{\int_0^t e^{(t - s)\Delta} B[(1 - \widetilde\chi_R)u^g,\omega^{c2}]\,ds}_{\bF_c}\\
&\qquad \lesssim 
\sup\limits_{t\in[0,T]} t^{\frac 12}\int_0^t (t-s)^{-\frac{3}{4}} \norm{(1 - \widetilde\chi_R)u^g}_{L^{12}_y} \norm{\brak{s^{-\frac{1}{2}}x}^m Q_\Phi^{-1}( \widetilde\chi_{6R}\omega^{c2})}_{L^\infty_zL^2_x}\, ds \\ 
&\qquad \quad + \sup\limits_{t\in[0,T]}t^{\frac 12}\int_0^t(t-s)^{-\frac{3}{4}} \norm{(1 - \widetilde\chi_R)u^g}_{L^6_y} \norm{\brak{s^{-\frac{1}{2}}\bd}^m (1-\widetilde\chi_{6R})\omega^{c2}}_{L^3_y}\, ds \\ 
&\qquad  \lesssim \left(T^{\frac 14}R^{-\frac74} + T^{\frac14}R^{-\frac32}\right) M_c^{-1} \norm{\omega}_{X_T}, 
\end{align*}
where to treat the first term we used from Lemma \ref{lem:StrHeat} (and Lemma \ref{lem:Coefs}), followed by H\"older's inequality, 
\begin{align*}
&\norm{\brak{t^{-\frac12}\bd}^m e^{(t - s)\Delta} B[(1 - \widetilde\chi_R)u^g,\widetilde\chi_{6R}\omega^{c2}]}_{L^3_y}\\
& \qquad\lesssim (t-s)^{-\frac34} \norm{\brak{s^{-\frac12}x}^m P_\Phi^{-1}[(1-\chi_R) u^g] \otimes Q_\Phi^{-1}[\widetilde\chi_{6R}\omega^{c2}]}_{L^{3}_z L^{12/7}_x}\\
&\qquad \lesssim (t-s)^{-\frac34} \norm{P_\Phi^{-1}[(1-\chi_R) u^g]}_{L^{12}_{z,x}} \norm{\brak{s^{-\frac{1}{2}}x}^m Q_\Phi^{-1}[\widetilde\chi_{6R}\omega^{c2}]}_{L^\infty_z L^2_x}. 
\end{align*}
We note that this estimate is likely suboptimal, but it is immediate from the Biot-Savart law in physical variables, whereas the optimal estimate would likely require a more delicate argument.

Next consider the contribution of $\omega^{g}$. Recalling that $\omega^g$ is supported in the set \(\{\bd\leq 4R\}\), hence using  the separation of supports as above and the estimate \eqref{uDiffSummary} we obtain
\begin{align*}
&\norm{\int_0^t e^{(t - s)\Delta} B[(1 - \widetilde\chi_R)(u - \alpha u^g),\alpha \omega^{g}]\,ds}_{\bF_c}\\
&\qquad  \lesssim \sup\limits_{t\in[0,T]}t^{\frac12}\int_0^t R^{-\frac12}(t-s)^{-\frac34} \norm{(1 - \widetilde\chi_R)(u - \alpha u^g)}_{L^{6}_y} \norm{\brak{s^{-\frac12}\bd}^m\omega^{g}}_{L^2_y} ds \\
& \qquad \quad\lesssim \left(T^{\frac14}R^{-2}  + R^{-\frac12}M_{b1}^{-1}\right)\norm{\omega}_{X_T}.   
\end{align*}
Finally, the term involving $\omega^{c1}$  is similar, using the estimate \eqref{velgL6} and the separation of supports to obtain the bound
\[
\norm{\int_0^t e^{(t-s)\Delta} B[(1 - \widetilde\chi_R)u^g,\omega^{c1}]\,ds}_{\bF_c}\lesssim T^{\frac14}R^{-2}\|\omega\|_{X_T}.
\]
Using the assumption that \(0<T,R\leq 1\), this completes the treatment of $f^{c2}_L$.

Next turn to the nonlinear contributions $f^{c2}_N$. Similar to above, the $\omega^{c1}$ contribution is easier due to the separation of supports, hence, we only consider the $\omega^{c2}$ contribution.
Here we may argue as before and apply the estimate \eqref{uDiffSummary} to bound
\begin{align*}
&\norm{\int_0^t e^{(t - s)\Delta} B[(1 - \widetilde\chi_R)(u - \alpha u^g),\omega^{c2}]\,ds}_{\bF_c}\\
& \qquad \lesssim \sup\limits_{t\in[0,T]}t^{\frac12}\int_0^t (t-s)^{-\frac{3}{4}} \norm{(1 - \widetilde\chi_R)(u - \alpha u^g)}_{L^{12}_y} \norm{\brak{s^{-\frac12}x}^mQ_\Phi^{-1}(\widetilde \chi_{6R}\omega^{c2})}_{L^\infty_z L^2_x} ds \\
& \qquad\quad + \sup\limits_{t\in[0,T]}t^{\frac12}\int_0^t (t-s)^{-\frac{3}{4}} \norm{(1 - \widetilde\chi_R)(u -\alpha  u^g)}_{L^{6}_y} \norm{\brak{s^{-\frac12}\bd}^m(1-\widetilde\chi_{6R})\omega^{c2}}_{L^3_y} ds \\ 
& \qquad\lesssim \left(T^{\frac14}R^{-\frac74} + T^{\frac14}R^{-\frac32} +  M_{b1}^{-1}\right)M_c^{-1}\norm{\omega}_{X_T}^2, 
\end{align*}
Using the assumption \eqref{ConstantOrder} and that \(0<R\leq 1\), this completes the treatment of $f_{N}^{c2}$.

Finally, consider $f_{I}^{c2}$. The estimation of $\cE^g$ is similar to the treatment of $\cE^c$ and the $B[(1-\widetilde \chi_R) u^g,\omega^g]$ is estimated as the corresponding terms in $f^{c2}_L$. This completes all of the requisite estimates. 

\end{proof}

We next turn to the estimate of $\omega^{c2}$ in $\bN_c^0$. 
\begin{lem} We have the following estimates:

\begin{itemize}
\item[(i)] If \(f_1\) is supported in \(\{|x|\leq 16R\}\) then:
\begin{equation}
\begin{aligned}
&\left\|Q_\Phi^{-1}\left(\widetilde\chi_{8R}\int_0^te^{(t - s)\Delta}\big(Q_\Phi\Div f_1(s) + \Div f_2(s)\right)\,ds\big)\right\|_{\bN^0_c}\\ &\qquad\qquad \lesssim \sup\limits_{0<s\leq T} \left(s^{\frac34}\|\<s^{-\frac12}x\>^mf_1(s)\|_{B_zL^{4/3}_x} +  s^{\frac34}\|\<s^{-\frac12}\bd\>^m f_2(s)\|_{L^2_y}\right).\label{FarToNear}
\end{aligned}
\end{equation}

\item[(ii)]  If \(\gamma\in \N^3\) is a multi-index satisfying \(|\gamma|\leq 2\) and \(f\) is supported in \(\{|x|\leq 4R\}\) then: 
\begin{align}
\left\|Q_\Phi^{-1}\left(\widetilde\chi_{8R}\int_0^te^{(t - s)\Delta}Q_\Phi\nabla_{x,z}^\gamma  f(s)\,ds\right)\right\|_{\bN^0_c}\lesssim T^{1 - \frac{|\gamma|}2}\|f\|_{\bN^\beta_c}, \label{ErrorContribution-1}
\end{align}
\end{itemize}
\end{lem}
\begin{proof}~

\begin{itemize}

\item[(i)] We may directly apply Lemma \ref{lem:StrHeat} to bound \(f_1\), so it suffices to assume that \(f_1\equiv0\).

From Lemma \ref{lem:Sobolev} we have the estimate 
\[
\|g\|_{B_zL^2_x}\lesssim \|g\|_{L^2_{x,z}}^{\frac12}\|\partial_zg\|_{L^2_{x,z}}^{\frac12}.
\]

Using Lemma~\ref{lem:Coefs} to bound the change of coordinates, 
\begin{align*}
 \norm{\<t^{-\frac12}x\>^mQ_\Phi^{-1}\left(\widetilde\chi_{8R}e^{(t-s)\Delta}\Div f(s)\right)}_{L^2_{x,z}}&\lesssim(t - s)^{-\frac12}\|\<s^{-\frac12} \bd\>^mf(s)\|_{L^2_y}\\
\norm{\<t^{-\frac12}x\>^m\partial_zQ_\Phi^{-1}\left(\widetilde\chi_{8R}e^{(t-s)\Delta}\Div f(s)\right)}_{L^2_{x,z}} &\lesssim (t - s)^{-1}\|\<s^{-\frac12} \bd\>^mf(s)\|_{L^2_y},
\end{align*}
and hence
\[
\|\<t^{-\frac12}x\>^m\chi_{8R}Q_\Phi^{-1}e^{(t - s)\Delta}\Div f(s)\|_{B_zL^2_x}\lesssim (t - s)^{-\frac34} \norm{\<s^{-\frac12}\bd\>^mf(s)}_{L^2_y}.
\]
Integrating, we obtain the estimate \eqref{FarToNear}.

\item[(ii)] 
The estimate \eqref{ErrorContribution-1} is easily proved for $\gamma = 0$; for $|\gamma| \geq 1$, we apply Lemma~\ref{lem:StrHeat} and interpolation to obtain the bound
\begin{align}
\|\<t^{-\frac12}x\>^m Q_\Phi^{-1}\left(\widetilde\chi_{8R}e^{(t - s)\Delta}Q_\Phi \nabla_{x,z}^\gamma f(s)\right)\|_{B_zL^2_x}\lesssim(t-s)^{- \frac{|\gamma|}2}s^{ - \frac12}\left(\frac{ t-s}s\right)^{\frac\beta2}\|f\|_{\bN^{\beta}}, \label{ineq:hwWc2}
\end{align}
whenever \(0<s<t\leq T\), which we may then integrate in \(s\), where the integral converges using the fact that \(0<\beta<\frac14\).
\end{itemize} 
\end{proof}

We now apply these linear bounds to control each of the terms, using the estimate \eqref{ErrorContribution-1} to control the terms \(\mc E^g,\mc E^c\), and the estimates \eqref{FarToNear} to control the remaining terms.

The estimate \eqref{ineq:wc2-a} follows from the following lemma:
\begin{lem}\label{lem:wc2N}
We have the estimates:
\begin{align}
\norm{Q_\Phi^{-1}\left(\widetilde\chi_{8R}\int_0^t e^{(t - s)\Delta} f^{c2}_L(s)\,ds\right)}_{\bN_c^0}& \lesssim \left(R + T^{\frac14}R^{-2} + M_{b1}^{-1} + R^{-\frac34}M_{b2}^{-1}\right)\|\omega\|_{X_T},\\
\norm{Q_\Phi^{-1}\left(\widetilde\chi_{8R}\int_0^t e^{(t - s)\Delta} f^{c2}_N(s)\,ds\right)}_{\bN_c^0}
&\lesssim \left(T^{\frac14}R^{-\frac12} + M_{b1}^{-1} + R^{-\frac34}M_{b2}^{-1}\right)\|\omega\|_{X_T}^2,\\
\norm{Q_\Phi^{-1}\left(\widetilde\chi_{8R}\int_0^t e^{(t - s)\Delta} f^{c2}_I(s)\,ds\right)}_{\bN_c^0}&\lesssim R + T^{\frac12}R^{-2}.
\end{align}
\end{lem}
\begin{proof}
First consider the error terms $\cE^c$ and $\cE^g$. 
Consider only $\cE^c$; the $\cE^g$ term is treated similarly. Estimating using \eqref{ErrorContribution-1} together with \eqref{eq:EcDecom} (and \eqref{i:cg-1}), we see
\begin{align}
\norm{Q_\Phi^{-1}\left(\widetilde\chi_{8R}\int_0^t e^{(t - s)\Delta} \cE^c(s)\,ds\right)}_{\bN_c^0} & \lesssim \left(R +  T^{\frac12}R^{-1} + T R^{-2}\right) \norm{\omega}_{X_T}. 
\end{align}
This suffices to treat $\cE^c$. 

Next we consider \(B[(1 - \widetilde\chi_R)u^g,\omega^{c1} + \omega^{c2}]\). First recall that
\[
B[(1 - \widetilde\chi_R)u^g,\omega^{c1}] = Q_\Phi B[P_\Phi^{-1}(\widetilde\chi_{4R}(1 - \widetilde\chi_R)u^g),\eta^{c1}],
\]
and that by H\"older's inequality and the estimate~\eqref{VelocityBzL4-g} we have
\begin{align*}
t^{\frac34}\|\<t^{-\frac12}x\>^mP_\Phi^{-1}(\widetilde\chi_{4R}(1 - \widetilde\chi_R)u^g)\otimes \eta^{c1}\|_{B_zL^{4/3}_x} &\lesssim t^{\frac34}\|P_\Phi^{-1}(\widetilde\chi_{4R}(1 - \widetilde\chi_R)u^g)\|_{B_zL^4_x}\|\<t^{-\frac12}x\>^m\eta^{c1}\|_{B_zL^2_x}\\
&\lesssim T^{\frac 14}R^{-\frac12}\|\omega\|_{X_T}.
\end{align*}
Applying the estimate~\eqref{FarToNear} with \(f_2\equiv0\) we then obtain
\[
\norm{Q_\Phi^{-1}\left(\widetilde\chi_{8R}\int_0^t e^{(t -s)\Delta}B[(1 - \widetilde\chi_R)u^g,\omega^{c1}]\,ds\right)}_{\bN_c^0}\lesssim T^{\frac14}R^{-\frac12}\|\omega\|_{X_T}.
\]
The contribution of \(\widetilde\chi_{6R}\,\omega^{c2}\) is treated similarly, where we note that \(\widetilde\chi_{7R}\gtrsim 1\) on the support of \(\widetilde\chi_{6R}\omega^{c2}\) and hence the estimate \eqref{VelocityBzL4-g} may still be applied to bound the velocity. For the contribution of \((1 - \widetilde\chi_{6R})\omega^{c2}\) we instead apply the estimate~\eqref{FarToNear} with \(f_1\equiv0\), using the estimate \eqref{velgL6} to bound
\begin{align*}
t^{\frac34}\norm{\brak{t^{-\frac{1}{2}}\bd}^m (1-\widetilde\chi_R)u^g \otimes ((1 - \widetilde\chi_{6R})\omega^{c2})}_{L^2} &\leq t^{\frac34}\norm{(1-\widetilde\chi_R)u^g}_{L^6} \norm{\<t^{-\frac12}\bd\>^m(1 -\widetilde\chi_{6R})\omega^{c2}}_{L^3}\\
&\lesssim T^{\frac14} R^{-\frac32} M_c^{-1} \norm{\omega}_{X_T}.
\end{align*}
The remaining terms are treated similarly, where we note that by applying the estimates \eqref{VelocityBzL4}, \eqref{Spatula}, \eqref{ineq:vb1}, and \eqref{ineq:vb2} we may bound
\[
t^{\frac14}\|P_\Phi^{-1}(\widetilde\chi_{7R}(1 - \widetilde\chi_R)(u - \alpha u^g))\|_{B_zL^4_x}\lesssim \left(T^{\frac14}R^{-\frac12} + M_{b1}^{-1} + R^{-\frac34}M_{b2}^{-1}\right)\|\omega\|_{X_T}.
\]

Using the assumption \eqref{ConstantOrder} and that \(0<T,R\leq 1\), this completes the proof of Lemma \ref{lem:wc2N}. 
\end{proof}

The estimates \eqref{ineq:wc2-a}, \eqref{ineq:wc2-b} follow as a consequence of Lemmas~\ref{lem:ErrorHandling},~\ref{lem:wc2N} and the hypotheses \eqref{Size}, \eqref{ConstantOrder}.

\subsection{Estimates on the background, $\omega^b$}

\subsubsection{$\eta^{b1}$ contribution} 
In this section we prove the estimate \eqref{ineq:wb1}. Once again we decompose
\begin{align*}
f^{b1} = f^{b1}_L + f^{b1}_N
\end{align*}
where the linear and nonlinear terms are given by
\begin{align*}
f^{b1}_L & = \alpha B[v^g - \bar v^g,\eta^{b1}] + \alpha B[v^g,\eta^{b2}] \\ 
f^{b1}_N & = B[v - \alpha v^g,\eta^{b1} + \eta^{b2}]. 
\end{align*}

To control these terms we use the following corollary of Proposition~\ref{prop:LinearSmoothing}
\begin{lemma} \label{lem:NLbs}
We have the following estimates for all $\gamma \in (0,\frac{1}{2})$ and \(0<s<t\leq T\):
\begin{align}
\|\<t^{\frac12}\nabla\>^\beta\bS(t,s)B[f,g]\|_{L^{4/3}_{x,z}} &\lesssim (t - s)^{-\frac34} \left(\frac t{t-s}\right)^{\frac\beta2}\left(\frac ts\right)^\gamma \|f\|_{B_zL^4_x}\|g\|_{L^{4/3}_{x,z}},\label{NL1}\\
\|\<t^{\frac12}\nabla\>^\beta\nabla\bS(t,s)B[f,g]\|_{L^{4/3}_{x,z}} &\lesssim(t - s)^{-\frac34} \left(\frac t{t-s}\right)^{\frac\beta2}\left(\frac ts\right)^\gamma \|\nabla f\|_{B_zL^{4/3}_x}\|\nabla g\|_{L^{4/3}_{x,z}},\label{NL2}
\end{align}
\end{lemma}
\begin{proof}
We prove the estimate \eqref{NL2}; the estimate \eqref{NL1} is similar and is omitted for the sake of brevity. We observe that \(\Div B[f,g] = 0\) so we may interpolate the estimates \eqref{LSmooth1a} and \eqref{LSmooth1b} to bound,
\[
\|\<t^{\frac12}\nabla\>^\beta\nabla\bS(t,s)B[f,g]\|_{L^{4/3}_{x,z}} \lesssim_\gamma (t - s)^{-\frac34} \left(\frac t{t-s}\right)^{\frac\beta2}\left(\frac ts\right)^\gamma\|B[f,g]\|_{L^{4/3}_zL^1_x}.
\]
Applying H\"older's inequality we obtain,
\[
\|B[f,g]\|_{L^{4/3}_zL^1_x}\lesssim \|f\|_{B_zL^4_x}\|\nabla g\|_{L^{4/3}_{x,z}} + \|\nabla f\|_{B_zL^{4/3}_x}\|g\|_{L^{4/3}_zL^4_x}.
\]
The estimate \eqref{NL2} then follows from Sobolev embedding.
\end{proof}

We may then apply these estimates to obtain the following:
\begin{lemma} 
There holds 
\begin{align}
\norm{\int_0^t \bS(t,s) f^{b1}_L(s) ds}_{\bN^\beta_b} &\lesssim M_{b_1}^{-1}\left(T^{\frac14}R^{-\frac12} + R\ln R^{-1} + \frac{M_{b1}}{M_{b2}}\right)\|\omega\|_{X_T}, \label{ineq:fb1L}\\
\norm{\int_0^t \bS(t,s) f^{b1}_N(s) ds}_{\bN^\beta_b} & \lesssim M_{b1}^{-1}\|\omega\|_{X_T}^2.\label{ineq:fb1N}
\end{align}
\end{lemma} 
\begin{proof} 
We prove the estimate \eqref{ineq:fb1L}, the estimate \eqref{ineq:fb1N} is similar. Estimates \eqref{NL1} and \eqref{NL2} give
\begin{align*}
\norm{\int_0^t \bS(t,s) f^{b1}_L(s) ds}_{\bN^\beta_b} & \lesssim  \sup\limits_{0 < s < T} \left( s^{\frac 14} \norm{v^g - \bar v^g }_{B_z L^4_x} + s^{\frac 14}\norm{\grad\left(v^g - \bar v^g\right)}_{B_z L^{4/3}_x}\right) \norm{\eta^{b1}}_{\bN^0_b} \\ 
& \quad +  \sup\limits_{0 <s < T} \left(s^{\frac 14}\norm{v^g}_{B_z L^4_x} + s^{\frac 14}\norm{\grad v^g}_{B_z L^{4/3}_x}\right) \norm{\omega^{b2}}_{\bF_b}.
\end{align*}
Then \eqref{ineq:fb1L} follows from the estimates~\eqref{i:vg-nobar},~\eqref{i:vgsep},~\eqref{i:vgdiff}  .
\end{proof}

\subsubsection{$\omega^{b2}$ contribution}
In this section we prove the estimate \eqref{ineq:wb2}. We start with a set of estimates for the heat propagator:
\begin{lemma} 
For \(0<s<t\leq T\) we have:
\begin{align}
\|e^{(t - s)\Delta}B[f,g]\|_{L^{4/3}_y} &\lesssim (t - s)^{-\frac34}\|f\|_{L^6_y}\|g\|_{L^{4/3}_y},\label{NL3}\\
\|\nabla e^{(t - s)\Delta}B[f,g]\|_{L^{4/3}_y} &\lesssim (t - s)^{-\frac34}\left(\|f\|_{L^6_y} + \|\nabla f\|_{L^2_y}\right)\|\nabla g\|_{L^{4/3}_y}.\label{NL4}
\end{align}
Further, for multi-indices \(\gamma\in \N^3\) satisfying \(|\gamma|\leq 2\), $0 \leq \beta \leq \abs{\gamma}$, and \(f\) supported in \(\{|x|\leq 4R\}\),
\begin{equation}
\|e^{(t - s)\Delta}\widetilde\chi_{4R}Q_\Phi\nabla_{x,z}^\gamma f\|_{L^{4/3}_y}\lesssim_\beta (t - s)^{-\frac{|\gamma|}2}\left(\frac{t - s}s\right)^{\frac\beta 2}\|\<s^{\frac12}\nabla\>^\beta f\|_{L^{4/3}_{x,z}}. \label{ineq:Ebheat}
\end{equation}
\end{lemma}
\begin{proof}
The estimates \eqref{NL3}, \eqref{NL4} follow from the explicit expression for the heat kernel. The estimate \eqref{ineq:Ebheat} follows from properties of the heat kernel using Lemma~\ref{lem:Coefs} to bound the change of variables.
\end{proof}

Again we will decompose the error term into linear and nonlinear parts as
\[
f^{b2} = f^{b2}_L + f^{b2}_N,
\]
where
\begin{align*}
f^{b2}_L &= \alpha B[(1 - \widetilde\chi_R)u^g,\omega^{b1} + \omega^{b2}] + \mc E^b,\\
f^{b2}_N &= B[(1 - \widetilde\chi_R)(u - \alpha u^g),\omega^{b1} + \omega^{b2}].
\end{align*}
We then have the following lemma:
\begin{lemma} \label{lem:ErrorHandlingAgain}
There holds
\begin{align}
\norm{\int_0^t e^{(t-s)\Delta} f^{b2}_{L}(s) ds}_{\bF_b} & \lesssim  M_{b1}^{-1}\left(R + T^{\frac14}R^{-2}\right)\|\omega\|_{X_T},\label{ineq:fb2l}\\
\norm{\int_0^t e^{(t-s)\Delta} f^{b2}_{N}(s) ds}_{\bF_b} & \lesssim M_{b1}^{-1}\left(T^{\frac14}R^{-\frac32} + M_{b1}^{-1}\right) \|\omega\|_{X_T}^2. \label{Whisk}
\end{align} 
\end{lemma}
\begin{proof}
Again we will just prove the linear estimate \eqref{ineq:fb2l} as the nonlinear estimate \eqref{Whisk} is similar.

First, \eqref{NL3} and \eqref{NL4} together with the change of variables Lemma~\ref{lem:Coefs} give
\begin{align*}
\norm{\int_0^t e^{(t-s)\Delta} B[(1 - \widetilde\chi_R)u^g,\omega^{b1}] ds}_{\bF_b} & \lesssim \sup_{0< s \leq T}s^{\frac14} \big(\norm{(1-\widetilde\chi_R)u^g}_{L^6} + \norm{\grad (1-\widetilde\chi_R)u^g}_{L^2}\big) \norm{\eta^{b1}}_{\bN^0_b}, 
\end{align*}
which we may then bound using the estimate \eqref{velgL6}. The corresponding \(\omega^{b2}\) term is similar.

It remains to bound the error \(\mc E^b\). However, as in Lemma~\ref{lem:ErrorHandling} we may write
\[
Q_\Phi^{-1}\mc E^b = \nabla_{x,z}^\gamma\left(C_\gamma \eta^{b1}\right),
\]
where the coefficients \(C_\gamma\) are compactly supported in the set \(\{|x|\leq 4R\}\) and satisfy the estimate \eqref{i:cg-1}. The desired bound then follows from the estimate \eqref{ineq:Ebheat} and the assumption that \(0<T,R\leq 1\).
\end{proof}

\subsection{Mild solution}\label{sec:MildlyIrritating}

With the above a priori estimates established, we are now in a position to verify that the solution constructed in Theorem \ref{thm:FixedPt} satisfies Definition \ref{milddef}.

We first recall the definition of the space \(M^p_q\) ($1 \leq q \leq p$) of functions \(f\in L^q_{\mr{loc}}\) with 
\[
\|\omega\|_{M^p_q}:=\sup_{r>0,\;y\in \R^3}\left\{r^{\frac3p - \frac 3q}\left(\int_{B(y,r)}|\omega|^q\,dy\right)^{\frac1q}\right\}<\infty,
\]
where we note that \(\cM^{\frac32}_1 \subset \cM^{\frac32}\), where \(\cM^{\frac 32}\) is defined as in Definition~\ref{milddef} (recall that \(\cM^{\frac 32}\) contains measures, not just locally integrable functions). Further, we recall the mapping properties of the heat operator (see~\cite[Proposition~3.2]{MR993821})
\eq{HeatMorrey}{
\|\nabla^\alpha e^{t\Delta}\|_{\cM^{3/2}\rightarrow \cM^{3/2}}\lesssim t^{-\frac{|\alpha|}2}. 
}

Similarly, we point out the following useful embeddings: (recall $1 \leq q \leq p$), 
\begin{align*}
L^p \subset \cM^{p}_q, \quad B_z L_x^{q} \subset \cM^{\frac{3q}{2}}_{q}, 
\end{align*}
which are used several times below. We then have the following Proposition:
\begin{prop} \label{prop:MildSoln}
The solution constructed in Theorem \ref{thm:FixedPt} is a mild solution of the 3D Navier-Stokes equations as in Definition \ref{milddef} (with $\omega$ constructed from $\omega^{c\ast}, \omega^{b\ast}$ as described above in Section \ref{sec:decomp}) and satisfies the estimates 
\begin{align}\label{APostiori}
\norm{\omega(t)}_{\cM^{3/2}} + t^{\frac14}\norm{\omega(t)}_{\cM^2_{4/3}} + t^{\frac14}\norm{u(t)}_{\cM^6_4} \lesssim 1. 
\end{align}
\end{prop}
\begin{proof}

We will prove the estimate \eqref{APostiori}. The fact that \(\omega\) is a mild solution in the sense of Definition~\ref{milddef} may then be proved by a variant of the argument used for the straight filament in Section~\ref{sec:Straight}. Note that by H\"older's inequality and the estimate \eqref{HeatMorrey} we may bound the nonlinear terms by
\eq{NonlinearHandling}{
\begin{aligned}
\int_s^t \norm{ e^{(t - \sigma)\Delta} B[u(\sigma),\omega(\sigma)]}_{\cM^{3/2}} \,d\sigma &\lesssim \int_s^t (t - \sigma)^{-\frac12} \|u(\sigma)\|_{\cM^6_4}\|\omega(\sigma)\|_{\cM^2_{4/3}}\,d\sigma\\
&\lesssim \left(\sup\limits_tt^{\frac14}\|u(t)\|_{\cM^6_4}\right)\left(\sup\limits_t t^{\frac14}\|\omega(t)\|_{\cM^2_{4/3}}\right).
\end{aligned}
}

To obtain the estimate on the velocity in \eqref{APostiori} we apply H\"older's inequality and Lemmas~\ref{lem:Coefs},~\ref{lem:PhiDiff} to bound the change of coordinates to obtain
\[
\|u\|_{\cM^6_4}\lesssim \|\widetilde\chi_R u\|_{\cM^6_4} + \|(1 - \widetilde\chi_R)u\|_{\cM^6_4}\lesssim \|v\|_{B_zL^4_x} + \|(1 - \widetilde\chi_R)u\|_{L^6_y}.
\]
The \(\cM^6_4\) estimate for \(u\) then follows from Lemmas~\ref{lem:vgests-1},~\ref{lem:vgests-2}. Similarly, we may bound
\[
\|\omega\|_{\cM^2_{4/3}}\leq \|\widetilde\chi_{\frac R4}\omega\|_{\cM^2_{4/3}} + \|(1 - \widetilde\chi_{\frac R4})\omega\|_{\cM^2_{4/3}}\lesssim \|\chi_{\frac R4}\eta\|_{B_zL^{4/3}_x} + \|(1 - \widetilde\chi_{\frac R4})\omega\|_{L^2_y},
\]
and the \(\cM^2_{\frac43}\) estimate for \(\omega\) follows from Lemmas~\ref{lem:bEmbeds-1},~\ref{lem:bEmbeds-2}.

It remains to establish the \(\cM^{\frac32}\) estimate for \(\omega\). First we observe that by arguing as in the straight filament case in Section~\ref{sec:Straight}, using estimates for \(S(\tau,\sigma)\), \(\bS(t,s)\) in \(B_zL^1_\xi\), \(B_zL^1_x\) respectively, we obtain the estimates
\[
\|\<\onabla\>^\beta H^{c1}\|_{B_zL^1_\xi} + \|\<t^{1/2}\nabla\>^\beta\eta^{b1}\|_{B_zL^1_x}\lesssim 1.
\]
Using the embedding \(B_zL^1_x\subset \cM^{\frac 32}\) and Lemmas~\ref{lem:Coefs},~\ref{lem:PhiDiff} to bound the change of coordinates we obtain the estimate
\eq{M32Frac}{
\|\<t^{\frac 12}\nabla\>^\beta\omega^g\|_{\cM^{3/2}} + \|\<t^{\frac 12}\nabla\>^\beta\omega^{c1}\|_{\cM^{3/2}} + \|\<t^{\frac 12}\nabla\>^\beta\omega^{b1}\|_{\cM^{3/2}}\lesssim 1.
}
Next we note that (see~\cite[Theorem~3.8]{MR1187618}) that the estimate \eqref{HeatMorrey} extends to fractional derivatives and hence we may argue as in Lemma~\ref{lem:ErrorHandling} using the estimate \eqref{M32Frac} to obtain
\[
\norm{\int_0^t e^{(t-s)\Delta}\mc E^*(s)}_{\cM^{3/2}}\lesssim \int_0^t(t - s)^{-1 + \frac\beta 2}s^{-\frac\beta2}\,ds \lesssim 1,
\]
for \(* = c,g,b\). Finally we use the estimate \eqref{NonlinearHandling} to bound the remaining nonlinear contributions to \(\omega^{c2}\), \(\omega^{b2}\), and the estimate
\[
\|e^{t\Delta }\mu^b\|_{\cM^{3/2}}\lesssim \|\mu^b\|_{L^{3/2}}\lesssim \|\mu^b\|_{W^{1,12/11}},
\]
to bound the initial data. Thus, we obtain the estimates
\[
\|\omega^{c2}\|_{\cM^{3/2}} + \|\omega^{b2}\|_{\cM^{3/2}}\lesssim 1,
\]
as required.
\end{proof}

\subsection{Uniqueness} \label{sec:Unique}
In this section we prove the following uniqueness criterion, which is a natural corollary of our proof.

\begin{thrm}
Let $\omega \in \mc C^2((0,T_\ast)\times \mathbb R^3)$ be another mild solution. Define $\widetilde{\omega}^c$ and $\omega^b$ as the solutions 
\begin{align}
&\pde{
\partial_t \widetilde{\omega}^c + B[u,\widetilde{\omega}^c] = \Delta \widetilde{\omega}^c,
}{ 
\widetilde{\omega}^c(t=0) = \alpha \delta_{\Gamma},
}
\\
&\pde{
\partial_t \omega^b +  B[u,\omega^b] = \Delta \omega^b,
}{
\omega^b(t=0) = \mu^b.  
}
\end{align}
Suppose that for sufficiently small $R > 0$ we have
\begin{align}
\lim_{T \searrow 0} \left( \norm{Q_{\Phi}^{-1} \left(\widetilde\chi_{8R}\left(\widetilde{\omega}^c - \alpha \omega^g\right)\right)}_{\bN^0_c} + \norm{\widetilde{\omega}^c -  \alpha \omega^g}_{\bF_c} + \norm{Q_\Phi^{-1}\left(\widetilde\chi_{8R}\omega^b\right)}_{\bN^0_b} + \norm{\omega^b}_{\bF_b}\right)  = 0, \label{uniqVan} 
\end{align}
where, as before, we take \(\omega^g = Q_\Phi^{-1}(\widetilde\chi_{2R}\eta^g)\). Then $\omega$ is necessarily the solution constructed above in Theorem \ref{thm:FixedPt}. 
\end{thrm}
\begin{proof}
We start by choosing \(\beta,m,\epsilon,M_*\) as in the proof of Theorem~\ref{thm:FixedPt}. Take \(\omega^c := \widetilde\omega^c - \omega^g\) and let
\[
\norm\omega:= \norm{Q_{\Phi} \left(\widetilde\chi_{8R}\omega^c\right)}_{\bN^0_c} + \norm{\omega^c}_{\bF_c} + \norm{Q_\Phi^{-1}\left(\widetilde\chi_{8R}\omega^b\right)}_{\bN^0_b} + \norm{\omega^b}_{\bF_b}
\]
so that from \eqref{uniqVan} we have \(\lim_{T \searrow 0} \norm\omega = 0\).

Our goal is to find a suitable decomposition of \(\omega^c,\omega^b\) so that by choosing sufficiently small \(R,T>0\) sufficiently small, the decomposed solution lies in the ball \(B_{\epsilon,T,R,M_*}\), in which the contraction mapping argument guarantees uniqueness. We are able to do this thanks to two key observations. First, the equations \eqref{Core1Piece}, respectively \eqref{bg1Piece}, for \(\eta^{c1}\), respectively \(\eta^{b1}\), depend on \(\omega^c,\omega^b\) rather than their decomposition. This will enable us to make a decomposition that matches the one outlined in Section~\ref{sec:Curve1}. Second, in the proof of Theorem~\ref{thm:FixedPt} we first fix \(M_*,\epsilon\) (recalling that the \(M_{b*}\approx R^{-\frac34}M_c\)), then choose \(R>0\) sufficiently small, and finally choose \(T>0\) sufficiently small. Thus, provided we choose \(R>0\) sufficiently small before choosing \(T>0\) sufficiently small, the contraction mapping argument still applies.

We now proceed to carry out this approach. We first define $\eta^{c1}$ to be a mild solution of the heat equation
\begin{align*}
\partial_t \eta^{c1} - \Delta \eta^{c1} = -B[v, \underline\eta^{c}] - B[v-\alpha \chi_R \bar v^g,\alpha \eta^g],
\end{align*}
with initial data \(\eta^{c1}(0) = 0\), where \(\underline\eta^c = Q_\Phi^{-1}\left(\widetilde\chi_{2R}\omega^c\right)\) and \(v = P_\Phi^{-1}\left(\widetilde\chi_R (-\Delta)^{-1}\nabla\times \omega\right)\). Switching to self-similar variables we obtain
\begin{align*}
H^{c1} = -\int_{-\infty}^{\tau} e^{(\tau-s)\mathcal{L} + (e^\tau-e^s)\partial_z^2} \left(\bar{B}[V, \underline H^{c}] + \bar{B}[V - \alpha \chi_{Re^{-s/2}}\bar V^g,\alpha H^g] \right) ds.
\end{align*}
By definition we may bound
\[
\|\underline H^c\|_{B_zL^2_\xi(m)}\lesssim \norm{\omega}.
\]
Decomposing \(V = \alpha V^g + V^c + V^b\), where \(V^*\) is the self-similar scaling of \(v^* = P_\Phi^{-1}\left(\chi_R\omega^*\right)\), we may apply the estimate \eqref{ineq:BSbasic} to bound
\[
\|V^c\|_{B_zL^4_\xi}\lesssim  \norm{\omega},
\]
and similarly, using the Sobolev estimate \eqref{Sobolev},
\[
\|V^b\|_{B_zL^4_\xi}\lesssim  \norm{\omega}.
\]
Further, we may apply the estimate \eqref{i:vgdiff} to bound
\[
\|V^g - \chi_{Re^{-s/2}}\bar V^g\|_{B_zL^4_\xi}\lesssim T^{\frac14}R^{-\frac12} + R\ln R^{-1}.
\]
Consequently, we apply \eqref{estimatesFP} to obtain the estimate
\begin{align*}
\|\eta^{c1}\|_{\bN_c^\beta} = \sup\limits_{-\infty<\tau\leq \ln T}\|\langle \overline\nabla\rangle^\beta H^{c1}\|_{B_zL^2_\xi(m)} &\lesssim \norm{\omega} + \norm{\omega}^2 + T^{\frac14}R^{-\frac12} + R\ln R^{-1}.
\end{align*}

We define \(\eta^{b1}\) to be the mild solution of the heat equation
\[
\partial_t\eta^{b_1} - \Delta\eta^{b1} = - B[v,\underline \eta^b],
\]
with initial data \(\eta^{b1}(t=0) = 0\), where \(\underline\eta^b = Q_\Phi^{-1}\left(\widetilde\chi_{2R}\omega^b\right)\). Using the self-similar scaling we may write
\[
H^{b1} = -\int_{-\infty}^\tau e^{(\tau - s)\cL + (e^\tau - e^s)\partial_z^2}\overline B[V,\underline H^b]\,ds.
\]
By definition we may bound
\[
R^{-\frac34}\|\underline H^b\|_{L^{4/3}_{\xi,z}} + R^{\frac14}\|\overline\nabla \underline H^b\|_{L^{4/3}_{\xi,z}}\lesssim \norm{\omega}.
\]
Further, applying the estimate \eqref{i:vg-nobar} for \(V^g\) together with the estimate \eqref{ineq:BSbasic} for \(V^c\), \(V^b\) we may bound
\[
\|V\|_{B_zL^4_\xi} + \|\nabla V\|_{B_zL^{4/3}_\xi}\lesssim 1 + \norm\omega.
\]
As a consequence, we obtain the estimate
\[
M_{b1}\|\eta^{b1}\|_{\bN_b^\beta} =  M_{b1}\sup\limits_{-\infty<\tau\leq \ln T} \left(R^{-\frac34}\|\langle \overline\nabla\rangle^\beta H^{b1}\|_{L^{4/3}_{\xi,z}} + R^{\frac14}\|\langle \overline\nabla\rangle^\beta \overline \nabla H^{b1}\|_{L^{4/3}_{\xi,z}}\right) \lesssim M_{b1}\norm{\omega} + M_{b1}\norm{\omega}^2.
\]

Next we define the remainders
\[
\omega^{c2} := \omega^c - Q_\Phi\left(\chi_{2R} \eta^{c1}\right),\qquad \omega^{b2} := \omega^b - Q_\Phi\left(\chi_{2R} \eta^{b1}\right),
\]
and observe that as \(v\) is supported on the region on which \(\chi_{2R} = 1\), for \(*=c,b\) the corresponding \(\eta^* = Q_\Phi^{-1}(\widetilde\chi_{2R} \omega^{*2}) + \eta^{*1}\) satisfies
\(
B[v,\underline\eta^*] = B[v,\eta^*].
\)
In particular, we may replace \(\underline \eta^*\) by \(\eta^*\) to see that \(\eta^{c1}\), respectively \(\eta^{b1}\), satisfy \eqref{Core1Piece}, respectively \eqref{bg1Piece}. It is then clear that $\omega^{c2}$ is a mild solution of \eqref{Core2PiecePhys} and that \(\omega^{b2}\) is a mild solution of \eqref{bg2PiecePhys}.

To complete the proof we use the triangle inequality to bound
\[
M_c\norm{Q_\Phi^{-1}\left(\widetilde\chi_{8R}\omega^{c2}\right)}_{\bN_c^0} \leq M_c\norm{Q_\Phi^{-1}\left(\widetilde\chi_{8R}\omega^c\right)}_{\bN_c^0} + M_c\norm{\eta^{c1}}_{\bN_c^0}\leq M_c\|\omega\| + M_c\|\eta^{c1}\|_{\bN_c^\beta},
\]
Further, by definition we have
\[
M_c\norm{\omega^{c2}}_{\bF_c} = M_c\norm{\omega^c}_{\bF_c}\leq M_c\norm\omega,
\]
Similarly, we may bound
\[
M_{b2}\|\omega^{c2}\|_{\bF_b}\lesssim M_{b2}\|\omega\| + M_{b2}\norm{\eta^{b1}}_{\bN_b^\beta}.
\]
Combining these bounds we obtain
\[
\|(\eta^{c1},\omega^{c2},\eta^{b1},\omega^{b2})\|_{X_T}\lesssim (1 + M_c + M_{b1} + M_{b2}) \left(\norm\omega + \norm\omega^2\right) + (1 + M_c)\left(T^{\frac14}R^{-\frac12} + R\ln R^{-1}\right).
\]
To complete the proof, we first recall that \(M_c\) is chosen independently of \(R>0\). Thus we may choose \(R>0\) sufficiently small to ensure the contribution of \((1 + M_c)R\ln R^{-1}\) is sufficiently small. The remaining terms are all \(o(1)\) as \(T\searrow 0\) and hence we may choose \(T>0\) sufficiently small to ensure that
\[
\|(\eta^{c1},\omega^{c2},\eta^{b1},\omega^{b2})\|_{X_T}\leq \epsilon.
\]
\end{proof}

\begin{appendix}

\section{Analysis in $B_zL^p_x$ spaces}

\subsection{Basic inequalities} 
\begin{lem}
If $1 \leq p,q,r \leq \infty$ and $\frac{1}{r} = \frac{1}{p} + \frac{1}{q}$, then
\begin{align}
\label{productlaw}
\norm{fg}_{B_z L^r_x} \lesssim \norm{f}_{B_z L^p_x} \norm{g}_{B_z L^q_x}.  
\end{align}
\end{lem}
\begin{proof} We consider the case that \(z\in \R\), the case that \(z\in \T\) is similar. Applying Minkowski's and H\"older's inequalities,
\begin{align*}
\| \widehat{fg}(x,\zeta) \|_{L^1_\zeta L^r_x}
& \leq \iint \left\| \widehat{f}(x,\zeta') \right\|_{L^p_x } \left\| \widehat{f}(x,\zeta-\zeta') \right\|_{ L^q_x } \,d\zeta' \,d\zeta \leq \| \widehat{f}(x,\zeta) \|_{L^1_\zeta L^p_x } \| \widehat{f}(x,\zeta) \|_{L^1_\zeta L^q_x } .
\end{align*}
\end{proof}

\begin{lem}\label{lem:Sobolev}
If \(1 < p\leq 2\) we have the estimate,
\eq{Sobolev}{
\|f\|_{B_zL^p_x}\lesssim \|f\|_{L^p_{x,z}}^{1 - \frac1p}\|\partial_zf\|_{L^p_{x,z}}^{\frac1p}.
}
\end{lem}
\begin{proof}
We consider the case that \(z\in \R\), the case that \(z\in\T\) is similar and assume \(f\not\equiv 0\). We take \(\frac1p + \frac 1{p'} = 1\) and for a real number \(A>0\) to be chosen shortly we apply H\"older's inequality to obtain,
\[
\int_{|\zeta|\leq A}\|\widehat f(\zeta)\|_{L^p_x}\,d\zeta\lesssim A^{\frac1p}\|\widehat f\|_{L^{p'}_\zeta L^p_x}\lesssim A^{\frac1p}\|\widehat f\|_{L^p_xL^{p'}_\zeta }\lesssim A^{\frac1p}\|f\|_{L^p_{x,z}},
\]
where we note that reversing the order of integration is justified as \(p'\geq p\) and the final inequality is proved using the Hausdorff-Young inequality for a.e. \(x\in \R^2\).

Similarly we may bound,
\[
\int_{|\zeta|>A}\|\widehat f(\zeta)\|_{L^p_x}\,d\zeta \lesssim A^{\frac1p - 1}\|\zeta \widehat f\|_{L^{p'}_\zeta L^p_x}\lesssim A^{\frac1p - 1}\|\partial_z f\|_{L^p_{x,z}}.
\]
We then choose \(A = \frac{\|\partial_z f\|_{L^p}}{\|f\|_{L^p}}\) to obtain the estimate \eqref{Sobolev}.
\end{proof}

\begin{lem}[Boundedness of the Riesz transforms] \label{lem:Riesz}
There holds for all $1<p<\infty$, 
\begin{align}
\norm{\grad^2 f}_{B_z L^p_x} & \lesssim_p \norm{\Delta f}_{B_z L^p_x}. 
\end{align}
\end{lem}
\begin{proof} 
Let $a,b \geq 0$ with $a + b = 2$; we claim the following holds independently of $\zeta$, 
\begin{align}
\norm{\zeta^{a} \grad_x^b (\zeta^{2} - \Delta_x)^{-1} \hat{g}(\cdot,\zeta)}_{L^p_x} \lesssim_p \norm{\hat{g}(\cdot,\zeta)}_{L^p_x}. 
\end{align}
The case $b = 2$ follows from the Calder\'on-Zygmund theorem applied to the Bessel potential. 
The cases $a = 1,2$ follows by scaling and that the kernels of $(1-\Delta_x)^{-1}$ and $\grad_x (1-\Delta_x)^{-1}$ are in $L^1$. 
\end{proof}

\subsection{Biot-Savart law in physical coordinates}
\begin{lem}
\label{lem:BSx}
\eq{BSx}{
\norm{\grad \times (-\Delta)^{-1} f}_{B_z L^{4}_x} \lesssim \norm{f}_{B_z L^{4/3}_x}.
}
\end{lem}
\begin{proof}
By scaling it suffices to prove that
\[
\|\nabla_x(1-\Delta_x)^{-1}f\|_{L^4_x}\lesssim \|f\|_{L^{4/3}_x},\qquad \|(1-\Delta_x)^{-1}f\|_{L^4_x}\lesssim \|f\|_{L^{4/3}_x}.
\]
This follows from the the Hardy-Littlewood-Sobolev, Young's inequality, and fact that the kernel of \(\nabla_x(1 - \Delta_x)^{-1}\) is in \(L^{2,\infty}\) and the kernel of \((1 - \Delta_x)^{-1}\) is in \(L^2\).
\end{proof}

\subsection{Biot-Savart law in straightened coordinates}

Our main goal in this section is to prove the following pair of propositions:

\begin{prop}\label{prop:StreamlinedBS}
Let $\eta$ be supported in $\{|x|\leq 16R\}$. Then, provided \(0<R\ll1\) is sufficiently small, the following holds with constants independent of \(R\):
\begin{align}
\norm{\chi_{8R}P_\Phi^{-1}(-\Delta)^{-1}\nabla\times Q_\Phi\eta}_{B_zL^4_x} + \norm{\nabla(\chi_{R}P_\Phi^{-1}(-\Delta)^{-1}\nabla\times Q_\Phi\eta)}_{B_zL^{4/3}_x} \lesssim \|\eta\|_{B_zL^{4/3}_x},\label{ineq:BSbasic}\\
\norm{ (1-\chi_R)\chi_{8R}P_{\Phi}^{-1} (-\Delta)^{-1} \grad \times  Q_{\Phi} (\chi_{\frac{R}{4}}\eta)}_{B_z L_x^{4}} \lesssim R^{-\frac{1}{2}}\norm{\eta}_{B_z L^{1}_x}. \label{ineq:BSbasicSeparation}
\end{align}
\end{prop}

\begin{prop} \label{prop:ApproxBS}
Let \(m>1\), \(0<R\ll1\) be sufficiently small, \(\eta\) be supported in $\{|x|\leq 16R\}$ with corresponding self-similar scaling \(H(\tau,\xi,z) = e^\tau\eta(e^\tau,e^{\frac\tau2}\xi,z)\) and
\[
v = \chi_RP_{\Phi}^{-1} \grad \times (-\Delta)^{-1} Q_{\Phi} \eta.
\]
Then, whenever \(0<\sqrt t\lesssim R\ll1\) we have the estimates,
\begin{align}
t^{\frac14}\norm{v - \nabla\times(-\Delta)^{-1}\eta}_{B_zL^4_x} &\lesssim \left( e^{\frac \tau4}R^{-\frac 12} + R\ln R^{-1} \right) \norm{H(\tau)}_{B_zL^2_\xi(m)},\label{appBS1}\\
t^{\frac14}\norm{\nabla\left(v - \nabla\times(-\Delta)^{-1}\eta\right)}_{B_zL^{4/3}_x} &\lesssim \left( e^{\frac \tau4}R^{-\frac12} + R\ln R^{-1} \right) \norm{H(\tau)}_{B_zL^2_\xi(m)}\label{appBS2}.
\end{align}
\end{prop}

The proof of both propositions rely on the following lemma:
\begin{lem}\label{lem:KernelBounds}
Let \(T\) be a translation-invariant operator with kernel \(k\in C^\infty(\R^3\backslash\{0\})\) satisfying the estimate
\[
|\partial_y^\alpha k(y)|\lesssim |y|^{-n - |\alpha|},
\]
for some \(1\leq n\leq 3\) and all multi-indices \(\alpha\in \N^3\). Then, the corresponding kernel \(K\) of the operator \(\mc FQ_\Phi^{-1}\widetilde\chi_{16R}TQ_\Phi \chi_{16R}\mc F^{-1}\), where \(\mc F\) is the Fourier transform in the \(z\)-variable, satisfies the estimate
\eq{KernelEst}{
|K(x,\zeta,x',\zeta')|\lesssim \begin{cases}
|x - x'|^{1-n} \<\zeta - \zeta'\>^{-2},&\quad n=2,3,\\
\left(1 + \ln |x - x'|\right) \<\zeta - \zeta'\>^{-2},&\quad n = 1
\end{cases}
}
\end{lem}
\begin{proof}
From Lemmas~\ref{lem:Coefs},~\ref{lem:PhiDiff} it suffices to prove that the approximate kernel
\eq{estimatetK}{
\widetilde K(x,\zeta,x',\zeta') = \iint \chi_{16R}(x)\chi_{16R}(x')k(\Phi(x,z) - \Phi(x',z'))  e^{i(z'\zeta' - z \zeta)}\,dz\,dz',
}
satisfies the estimate \eqref{KernelEst}. To prove this bound we integrate by parts repeatedly with the help of the formula
$$
\frac{1}{i(\zeta-\zeta')} (\partial_z + \partial_{z'}) e^{i(z'\zeta' - z \zeta)} = e^{i(z'\zeta' - z \zeta)},
$$
and observe that by Lemmas~\ref{lem:Coefs},~\ref{lem:PhiDiff} there holds 
$$
\left| (\partial_z + \partial_{z'})^N K(\Phi(x,z) - \Phi(x',z')) \right| \lesssim_N \frac{1}{|x-x'|^n + |z-z'|^n},
$$
Recalling that \(z,z'\in \T\), we may dyadically decompose the integral into sets \(\{2^{k}\leq |z - z'|<2^{2k}\}\), for all $N \geq 0$, and estimate
\begin{align}
\left| \widetilde K(x,\zeta,x',\zeta') \right| \lesssim_N \sum_{k} \frac{1}{|\zeta-\zeta'|^N} \frac{2^k}{|x-x'|^n + 2^{nk}}.
\end{align}
Choosing $N=0$ for $\zeta = \zeta'$ and $N=2$ for $\zeta \neq \zeta'$ then leads to the estimate~\eqref{estimatetK}.
\end{proof}

\begin{proof}[Proof of Proposition~\ref{prop:StreamlinedBS}]
From Lemma~\ref{lem:Coefs}, the Hardy-Littlewood-Sobolev and Young inequalities, to prove the first inequality in \eqref{ineq:BSbasic} and \eqref{ineq:BSbasicSeparation}, it suffices to show that the kernel \(K\) of the operator
\[
\mc FQ_\Phi^{-1}\widetilde\chi_{16R} (-\Delta)^{-1}\nabla\times Q_\Phi\chi_{16R}\mc F,
\]
satisfies the bounds
\[
|K(x,\zeta,x',\zeta')|\lesssim |x - x'|^{- 1}\<\zeta - \zeta'\>^{-2},
\]
which follows directly from Lemma~\ref{lem:KernelBounds}.

Next, consider the second inequality in \eqref{ineq:BSbasic}. 
Write 
\[
\psi = \chi_{8R}Q_\Phi^{-1}(-\Delta)^{-1}Q_\Phi \eta.
\]
We observe that, from Lemma~\ref{lem:Coefs} and Lemma~\ref{lem:KernelBounds} we have
\begin{align}
\frac{1}{R\ln R^{-1}} \|\psi\|_{B_zL^4_x} + \|\nabla\psi\|_{B_zL^4_x} + \frac{1}{R} \| \nabla \psi \|_{B_z L^{4/3}_x} \lesssim \|\eta\|_{B_zL^{4/3}_x}, \label{ineq:psibds}
\end{align}
and from the identity \eqref{Deltag} we have
\[
\chi_{2R}(-\Delta_\Phi) \psi = \chi_{2R}\eta. 
\]
Computing further gives
\begin{align}
-\Delta \grad^2 \chi_{2R} \psi = \grad^2(\chi_{2R}\eta) - \grad^2(\Delta - \Delta_{\Phi})\chi_{2R} \psi - \grad^2 [\chi_{2R}, (-\Delta_{\Phi})] \psi. 
\end{align}
Lemma \ref{lem:Riesz} gives 
\begin{align}
\norm{\grad^2 (\chi_{2R}\psi)}_{B_z L^{4/3}_x} & \lesssim \norm{\chi_{2R}\eta}_{B_z L^{4/3}_x} + \norm{(\Delta - \Delta_{\Phi})\chi_{2R} \psi}_{B_z L^{4/3}_x} + \norm{[\chi_R, (-\Delta_{\Phi})] \psi}_{B_z L^{4/3}_x}. 
\end{align}
By Lemma \ref{lem:Coefs}, \eqref{ineq:psibds} and H\"older's inequality,
\begin{align}
\norm{(\Delta - \Delta_{\Phi})\chi_{2R} \psi}_{B_z L^{4/3}_x} & \lesssim R\norm{\grad^2 (\chi_{2R}\psi)}_{B_z L^{4/3}_x} + \norm{\grad (\chi_{2R} \psi)}_{B_z L^{4/3}_x} + \norm{\chi_{2R} \psi}_{B_z L^{4/3}_x} \\ 
& \lesssim R\norm{\grad^2 (\chi_{2R}\psi)}_{B_z L^{4/3}_x} + \norm{\eta}_{B_z L^{4/3}_x}. 
\end{align}
Similarly, using Lemma \ref{lem:Coefs} and \eqref{ineq:psibds}, 
\begin{align}
\norm{[\chi_R, (-\Delta)_{\Phi}] \psi}_{B_z L^{4/3}} & \lesssim \frac{1}{R} \norm{\psi}_{B_zL^{4/3}_x}  \lesssim \norm{\eta}_{B_z L^{4/3}_x}. 
\end{align}
Therefore, we conclude, for all $R$ sufficiently small, that 
\begin{align}
\norm{\grad^2 (\chi_{2R}\psi)}_{B_z L^{4/3}_x} \lesssim \norm{\eta}_{B_z L^{4/3}_x}. 
\end{align}
This estimate (and the first inequality in \eqref{ineq:BSbasic} together with Lemma \ref{lem:Coefs}) then implies the second inequality in \eqref{ineq:BSbasic}. 
\end{proof}

The proof of Proposition~\ref{prop:ApproxBS} is a tiny bit more involved, but again essentially follows from Lemma~\ref{lem:KernelBounds}:

\begin{proof}[Proof of Proposition~\ref{prop:ApproxBS}]
We prove the estimate \eqref{appBS1}; the proof of the estimate \eqref{appBS2} is similar. As in the proof of \eqref{ineq:BSbasic}, write
\[
\psi = \chi_{8R}Q_\Phi^{-1}(-\Delta)^{-1}Q_\Phi \eta,
\]
and recall that, applying Lemma~\ref{lem:Coefs} to the operator \(\mc F_zQ_\Phi^{-1}\widetilde\chi_{16R}\nabla^\alpha(-\Delta)^{-1}Q_\Phi\chi_{16R}\mc F\), we have (using also Proposition \ref{prop:StreamlinedBS}),
\[
\frac{1}{R\ln R^{-1}} \|\psi\|_{B_zL^4_x} + \|\nabla\psi\|_{B_zL^4_x} + \frac{1}{R} \| \nabla \psi \|_{B_z L^{4/3}_x} + \|\nabla^2\psi\|_{B_zL^{4/3}_x}\lesssim \|\eta\|_{B_zL^{4/3}_x},
\]
and from the identity \eqref{Deltag} we have
\[
\chi_{2R}(-\Delta_\Phi) \psi = \chi_{2R}\eta,\qquad \chi_{2R}\fD^{-1}\Curl_\Phi \psi = v.
\]
Next, we decompose the difference
\begin{align*}
v - \nabla\times(-\Delta)^{-1}\eta &= \underbrace{\chi_R\left(\fD^{-1}\Curl_\Phi - \nabla\times\right)\psi}_{T_1} + \underbrace{\chi_R\nabla\times\left(\psi - (-\Delta)^{-1}\eta\right)}_{T_2}\\
&\quad + \underbrace{(1 - \chi_R)\nabla\times(-\Delta)^{-1}\eta}_{T_3}.
\end{align*}
To bound the first term we use the expression \eqref{Curlg} and Lemma~\ref{lem:Coefs} to bound
\[
t^{\frac14}\|T_1\|_{B_zL^4_x} \lesssim Rt^{\frac14}\|\nabla\psi\|_{B_zL^4_x} + t^{\frac14}\|\psi\|_{B_zL^4_x}\lesssim R\ln R^{-1} t^{\frac14}\|\eta\|_{B_zL^{4/3}_x}.
\]
To bound the second term, we apply Lemma~\ref{lem:BSx} followed by the identity \eqref{Deltag}, and the estimates on $\psi$ and its derivatives above to obtain
\begin{align*}
t^{\frac14}\|T_2\|_{B_zL^4_x} &\lesssim t^{\frac14}\|\chi_{2R}\left(\Delta - \Delta_\Phi\right)\psi\|_{B_zL^{4/3}_x}\\ 
& \lesssim Rt^{\frac14}\|\nabla^2\psi\|_{B_zL^{4/3}_x} + t^{\frac14} \|\nabla \psi\|_{B_zL^{4/3}_x} + t^{\frac14}\|\psi\|_{B_zL^{4/3}_x} \\
&\lesssim R\ln R^{-1}t^{\frac14}\|\eta\|_{B_zL^{4/3}_x}.
\end{align*}
To bound the final term we estimate
\begin{align*}
t^{\frac14}\|T_3\|_{B_zL^4_x} &\lesssim t^{\frac14}\|(1 - \chi_R)\nabla\times(-\Delta)^{-1}(\chi_{\frac R4}\eta)\|_{B_zL^4_x} + t^{\frac14}\|\nabla\times(-\Delta)^{-1}((1 - \chi_{\frac R4})\eta)\|_{B_zL^4_x}\\
&\lesssim \frac{e^{\frac\tau4}}{\sqrt R}\|H\|_{B_zL^2_\xi(m)},
\end{align*}
where we have used the separation of the supports and Young's inequality to bound the first term, and the estimate \eqref{BSx} with the fact that \(m>1\) and \(0<\sqrt t\lesssim R\) to bound the second.
\end{proof}

\subsection{The heat equation}

\begin{lem} \label{lem:StrHeat}
Let \(f\) be supported in $\{|x|\leq 16R\}$ and $1 \leq r \leq 2$. Suppose further that \(0<R\ll1\) and \(0<s<t\leq R^2\). Then for all multi-indices \(\gamma,\gamma'\in \N^3\) there holds 
\begin{align}
\norm{\brak{t^{-\frac{1}{2}}x}^{m}Q_\Phi^{-1}\left(\widetilde\chi_{8R}e^{(t-s)\Delta}\nabla^\gamma Q_\Phi \nabla^{\gamma'} f\right)}_{B_z L^2_x} \lesssim (t - s)^{\frac{1}{2} - \frac{1}{r} - \frac{|\gamma| + |\gamma'|}2}\norm{\brak{s^{-\frac{1}{2}}x}^{m} f}_{B_z L^r_x}, \label{ineq:heatBas}
\end{align}
and $B_z$ can also be replaced by $L^r_z$. 

Furthermore, there holds: for $1 \leq p \leq q \leq \infty$, 
\begin{align}
\norm{\brak{t^{-\frac{1}{2}}\bd}^m \widetilde{\chi}_{8R} e^{(t-s)\Delta} Q_{\Phi} \Div f}_{L^q_y} \lesssim \frac{1}{(t-s)^{\frac{1}{2} + \frac{1}{p} - \frac{1}{q}}}\norm{\brak{s^{-\frac{1}{2}}x}^m f}_{L^q_z L_x^{p}}. \label{ineq:aniso} 
\end{align}
\end{lem}
\begin{proof}
We consider \eqref{ineq:heatBas}; the proof of \eqref{ineq:aniso} is a straightforward variant. We consider only the case $\gamma,\gamma' = 0$ the case $|\gamma|,|\gamma'| > 0$ follows similarly, using Lemma~\ref{lem:Coefs} to replace the heat kernel by a similar, rapidly decaying kernel.

Consider the operator 
\begin{align}
T f = t^{-\frac32} \iint \chi_{16R}(x)\chi_{16R}(x') e^{\frac{|\Phi(x,z) - \Phi(x',z')|^2}{4t}}  f(x',z')\,dx' \,dz'.
\end{align}
The change of variable estimates of Lemmas~\ref{lem:Coefs},~\ref{lem:PhiDiff} and Young's inequality imply 
\begin{align}
\norm{\brak{t^{-\frac{1}{2}}x}^{m} Tf}_{L^r_z L^2_x} \lesssim t^{\frac{1}{2} - \frac{1}{r}}\norm{\brak{t^{-\frac{1}{2}}x}^{m} f}_{L^r_x}. 
\end{align}
The corresponding estimate follows for the heat operator (again from Lemma \ref{lem:Coefs}). 

To obtain the estimates in the case $\widehat{L^1}$, it suffices to prove that (where $C >0$ denotes a fixed constant independent of the parameters of interest) 
\begin{equation}
\label{estimateHK}
\left| \iint \chi_{16R}(\sqrt{t}\xi)\chi_{16R}(\sqrt{t}\xi') e^{-\frac{|\Phi(\sqrt{t} \xi,z) - \Phi(\sqrt{t} \xi',z')|^2}{4t}} e^{i(z'\zeta' - z \zeta)}\,dz\,dz' \right| \lesssim \sqrt{t} \phi(\sqrt{t}(\zeta - \zeta')) e^{-C|\xi-\xi'|^2},
\end{equation} 
with $\phi \in L^1$. We integrate by parts repeatedly using the formula
$$
\frac{1}{i(\zeta-\zeta')} (\partial_z + \partial_{z'}) e^{i(z'\zeta' - z \zeta)} = e^{i(z'\zeta' - z \zeta)},
$$
and keep in mind that (for constants $C > 0$, not necessarily the same in each inequality or each term), 
$$
\left| (\partial_z + \partial_{z'})^N e^{-\frac{|\Phi(\sqrt{t}\xi,z) - \Phi(\sqrt{t} \xi',z')|^2}{4t}} \right| \lesssim t^{-\frac N2} e^{-C \frac{|\Phi(\sqrt t \xi,z) - \Phi(\sqrt t \xi',z')|^2}{t}}\lesssim t^{-\frac N2}e^{-C|\xi - \xi'|^2}e^{-C\frac{|z - z'|^2}{\sqrt t}}
$$
Dyadically decomposing the integral into sets $\{2^{k}\leq|z-z'|<2^{k+1}\}$, the integral above can then be bounded by (possibly different constants) 
\begin{align*}
 &\left| \iint \chi_{16R}(\sqrt{t}\xi)\chi_{16R}(\sqrt{t}\xi')e^{-\frac{|\Phi(\sqrt t \xi,z) - \Phi(\sqrt{t} \xi',z')|^2}{4t}} e^{i(z'\zeta' - z \zeta)}\,dz\,dz' \right|\\
 & \qquad\lesssim \sum_{k} \frac{1}{(\sqrt{t}|\zeta-\zeta'|)^N} 2^k e^{-C \frac{2^{2k}}t} e^{-C|\xi-\xi'|^2}  \\ &\qquad \lesssim \frac{\sqrt t}{(\sqrt{t}|\zeta-\zeta'|)^N} e^{-C|\xi-\xi'|^2}.
\end{align*}
Choosing $N=0$ for $\zeta = \zeta'$ and $N=2$ for $\zeta \neq \zeta'$ leads to the desired estimate.
\end{proof} 

\section{Two-dimensional semigroup estimates}\label{app:Semigroups}

\subsection{Statement of the estimates}

In this appendix we prove several estimates for compact perturbations of the semigroup \(e^{\tau\cL}\). For concreteness we recall several linear operators defined in the main body of the article:
\begin{itemize}
\item If \(f\) is a \(2\)-vector we define
\[
\Gamma f = g\cdot\nabla_\xi f - f\cdot\nabla_\xi g
\]
\item If \(f\) is a scalar we define
\[
\Lambda f = g\cdot\nabla_\xi f - (-\Delta_\xi)^{-1}\nabla_\xi^\perp f\cdot\nabla_\xi G
\]
\item If \(f\) is a \(2\)-vector we define
\[
\Xi f = (\nabla_\xi\cdot f)g
\]
\item If \(f\) is a \(2\times 2\) tensor we define
\[
\Pi f = g\otimes \Div_\xi f - \Div_\xi f\otimes g.
\]
\end{itemize}
Our main results in this section are as follows.

\begin{prop}[Semigroups on \(L^2_\xi(m)\)]\label{prop:SemigroupL2m}
For all \(\alpha\in \R\) and \(m>1\) the operators \(\cL - \alpha \Gamma\), \(\cL - \alpha \Lambda\) define strongly continuous semigroups \(e^{\tau(\cL - \alpha \Gamma)}\), \(e^{\tau(\cL - \alpha\Lambda)}\) on \(L^2_\xi(m)\) so that for all \(\gamma>0\) we have
\eq{SemigroupL2m}{
\|e^{\tau(\cL - \alpha\Gamma)}\|_{L^2_\xi(m)\rightarrow L^2_\xi(m)}\lesssim e^{\gamma \tau},\qquad \|e^{\tau(\cL - \alpha \Lambda)}\|_{L^2_\xi(m)\rightarrow L^2_\xi(m)}\lesssim e^{\gamma \tau}.
}

If \(\alpha\neq 0\) there exists some \(0<\mu = \mu(\alpha)<\frac12\) so that, whenever \(m>1 + 2\mu\), we have the estimate
\eq{SemigroupL2ma}{
\|e^{\tau(\cL - \alpha\Gamma)}\|_{L^2_\xi(m) \rightarrow L^2_\xi(m)}\lesssim e^{-\mu\tau}.
}

If \(L^2_{\xi,0}(m) = \{f\in L^2_\xi(m):\int f\,d\xi = 0\}\) and \(m>1 + 2\mu\) for some \(0<\mu< \frac12\), then we have the estimate
\eq{SemigroupL2mb}{
\|e^{\tau(\cL - \alpha\Lambda)}\|_{L^2_{\xi,0}(m)\rightarrow L^2_{\xi,0}(m)}\lesssim e^{-\mu \tau}.
}
\end{prop}

\begin{prop}[Semigroups on \(L^1_\xi\)]\label{prop:SemigroupL1}
For all \(\alpha\in \R\) the operators \(\cL - \alpha\Gamma\), \(\cL - \alpha \Xi\), \(\cL - \alpha \Pi\) define strongly continuous semigroups \(e^{\tau(\cL - \alpha \Gamma)}\), \(e^{\tau(\cL - \alpha\Xi)}\), \(e^{\tau(\cL - \alpha\Pi)}\) on \(L^1_\xi\) so that for all \(\gamma>0\) we have the estimates
\eq{SemigroupL1}{
\|e^{\tau(\cL - \alpha \Gamma)}\|_{L^1_\xi\rightarrow L^1_\xi}\lesssim e^{\gamma\tau},\qquad \|e^{\tau(\cL - \alpha \Xi)}\|_{L^1_\xi\rightarrow L^1_\xi}\lesssim e^{\gamma\tau},\qquad \|e^{\tau(\cL - \alpha \Pi)}\|_{L^1_\xi\rightarrow L^1_\xi}\lesssim e^{\gamma\tau}.
}
\end{prop}

The proof of Propositions~\ref{prop:SemigroupL2m}~\ref{prop:SemigroupL1} essentially follows the argument of~\cite{MR1912106}, using additional estimates from~\cite{MR2770021},~\cite{MR2178064}. The general strategy of the proof for \(* = \Gamma,\Lambda,\Xi,\Pi\) is summarized as follows:
\begin{itemize}
\item \textbf{Step 1:} Prove that all eigenvalues of \(\cL - \alpha*\) are in the Gaussian-weighted space \(L^2_\xi(\infty)\).
\item \textbf{Step 2:} Use symmetries of the operators in \(L^2_\xi(\infty)\) to obtain bounds on the eigenvalues.
\item \textbf{Step 3:} Use the fact that \(e^{\tau(\cL - \alpha *)}\) is a compact perturbation of \(e^{\tau\cL}\) to deduce the corresponding spectral radii.
\end{itemize}

\subsection{Properties of the Fokker-Plank operator}
Before proceeding, it will be useful to collect some properties of the Fokker-Planck operator \(\cL\). We first recall (see e.g.~\cite[Theorem~A.1]{MR1912106}) the explicit expression
\eq{etcl}{
e^{\tau\cL}f = \frac{e^\tau}{4\pi a(\tau)}\int e^{-\frac{|\xi - \xi'|^2}{4a(\tau)}}f(e^{\frac\tau2}\xi'  )\,d\xi',
}
which immediately implies, if $1 \leq p \leq q \leq \infty$, 
\begin{equation}
\label{estimatesFP}
\begin{split}
& \| \nabla^\beta e^{\tau\cL}f \|_{L^q_\xi(m)} \lesssim \frac{e^{\tau(1-\frac{1}{p})}}{a(\tau)^{\frac{1}{p}-\frac{1}{q}+ \frac{|\beta|}{2}}} \| f \|_{L^p_{\xi}(m)}
\\
& \| e^{\tau\cL}\nabla^\beta f \|_{L^q_\xi(m)} \lesssim \frac{e^{\tau(1-\frac{1}{p})} e^{-\frac{|\beta|\tau}{2}}}{a(\tau)^{\frac{1}{p}-\frac{1}{q}+ \frac{|\beta|}{2}}} \| f \|_{L^p_{\xi}(m)}.
\end{split}
\end{equation}
Further, on \(L^2_\xi(m)\) we have the following proposition:
\begin{prop}[{\cite[Theorem~A.1]{MR1912106}}]\label{prop:FPProps}
Let \(m\in[0,\infty]\) and \(\cL\) be considered to be an unbounded operator on \(L^2_\xi(m)\) defined on its maximal domain \(\cD(m)\).
\begin{enumerate}
\item The spectrum of \(\cL\) is given by
\[
\sigma(\cL) = \left\{\lambda\in \C:\Re(\lambda)\leq \frac{1 - m}2\right\}\cup\left\{-\frac k2:k\in \N\right\}.
\]
\smallskip
\item If \(m>1\) then \(0\) is an isolated eigenvalue of \(\cL\), the corresponding eigenspace is spanned by the Gaussian \(G\). Further, the corresponding spectral projection on \(L^2_\xi(\infty)\) is given by \[P_0f = G\int_{\R^2}f\,d\xi.\]
\smallskip
\item If \(m>2\) then \(-\frac12\) is an isolated eigenvalue of \(\cL\) and the corresponding eigenspace is spanned by \(\{\frac12\xi_1 G,\frac12\xi_2G\}\). Further, the corresponding spectral projection on \(L^2_\xi(\infty)\) is given by
\[
P_1f = \frac12\xi_1G\int_{\R^2}\xi_1 f\,d\xi + \frac12\xi_2G\int \xi_2f\,d\xi.
\]
\end{enumerate}
\end{prop}

\subsection{Construction of the semigroups}

The semigroups that are the object of this section, namely
on the one hand $e^{\tau(\cL - \alpha\Gamma)}$ and $e^{\tau(\cL - \alpha \Lambda)}$ on $L^2_\xi(m)$, and on the other hand $e^{\tau(\cL - \alpha\Gamma)}$, $e^{\tau(\cL - \alpha \Xi)}$ and $e^{\tau(\cL - \alpha \Pi)}$ on $L^1_\xi$, can be constructed by a straightforward fixed point procedure.

We illustrate this for $e^{\tau(\cL - \alpha\Gamma)}$ on $L^2_\xi(m)$. For $f_0 \in L^2_\xi(m)$, let $\mathcal{T}$ be the operator
$$
\mathcal{T} f = e^{\tau \mathcal{L}} f_0 - \alpha \int_0^\tau e^{(\tau-\sigma) \mathcal{L}} \Gamma f(\sigma)\,d\sigma.
$$
Using the estimates~\eqref{estimatesFP}, it is easy to check that, for $\delta$ sufficiently small, $\mathcal{T}$ is a contraction on $\mathcal{C}([0,\delta],L^2_\xi(m))$.

This procedure gives strongly continuous semigroups; it also specifies a domain for the operators $\cL - \alpha\Gamma$ and $\cL - \alpha \Lambda$ on $L^2_\xi(m)$ (abusing notations, we denote them indistinctly by $\mathcal{D}(m)$), and similarly for the operators on $L^1_\xi$ (which we denote by $\mc D_1$). Whenever the spectrum (or essential spectrum, etc...) of these operators is discussed below, it is with respect to this domain.

\subsection{Gaussian decay of eigenfunctions}
In order to control the eigenvalues, we first show that they have Gaussian decay. We first consider the operator \(\cL - \alpha\Gamma\):
\begin{lem}\label{lem:EV-G-Inf}
Let \(\alpha\in \R\) and \(f\) be an eigenfunction of \(\cL - \alpha \Gamma\) with eigenvalue \(\lambda\). If one of:
\begin{enumerate}
\item \(f\in \cD(m)\), $\lambda \notin \sigma_{\operatorname{ess}}(\cL - \alpha \Gamma)$ and \(\Re\lambda>\frac{1 - m}2\) for \(m>1\),
\item \(f\in \cD_1\), $\lambda \notin \sigma_{\operatorname{ess}}(\cL - \alpha \Gamma)$ and \(\Re\lambda>0\),
\end{enumerate}
hold, then \(\xi^k \nabla_\xi^\ell f\in L^2_\xi(\infty)\) for all integers $k,\ell$.
\end{lem}
\begin{proof}
We follow the argument given in~\cite[Lemma~4.5]{MR2123378}, although we note that an alternative proof for case 1 is given in in~\cite[Proposition~3.4]{MR2770021}.

Suppose that \((\cL - \alpha\Gamma )f = \lambda f\) and switch to polar coordinates, taking
\[
\xi = \begin{bmatrix}r\cos\theta\\r\sin\theta\end{bmatrix},\qquad \begin{bmatrix}
f^r\\f^\theta
\end{bmatrix}
=
\begin{bmatrix}
\cos\theta&\sin\theta\\-\sin\theta&\cos\theta
\end{bmatrix}
f.
\]
Using the formulas $g \cdot \nabla f = V(\partial_\theta f^r - f^\theta) e_r + V ( \partial_\theta f^\theta + f^r) e_\theta$ and $f \cdot \nabla g = f^r \partial_r (rV) e_\theta - V f^\theta e_r$, we obtain the system
\[
\pde{
\lambda f^r + \alpha V\partial_\theta f^r = \left(\partial_r^2 + \dfrac1r\partial_r + \dfrac1{r^2}\partial_\theta^2 + \dfrac12 r\partial_r + 1 - \dfrac1{r^2}\right)f^r - \dfrac2{r^2}\partial_\theta f^\theta,
}{
\lambda f^\theta + \alpha V\partial_\theta f^\theta - \alpha r \partial_r V f^r  = \left(\partial_r^2 + \dfrac1r\partial_r + \dfrac1{r^2}\partial_\theta^2 + \dfrac12 r\partial_r + 1 - \dfrac1{r^2}\right)f^\theta + \dfrac2{r^2}\partial_\theta f^r,
}
\]
where
\[
V(r) = \frac1{2\pi r^2}\left(1 - e^{-\frac14 r^2}\right). 
\]

The decomposition of $f$ in angular harmonics is given by
$f = \sum_{n\in\Z} \widehat{f_n} e^{in\theta}$; the crucial observation is that the projectors on angular harmonics commute with $\cL - \alpha \Gamma$. If $\widehat{f_n}$ was non-zero for infinitely many $n$, then the kernel of $\mathcal{L}-\alpha \Gamma$ would contain all the linear combinations of the $\widehat{f_n} e^{in\theta}$, and hence be infinite-dimensional, which would contradict $\lambda \notin \sigma_{\operatorname{ess}} (\mathcal{L}-\alpha \Gamma)$. Therefore, only a finite number of $\widehat{f_n}$ are non-zero, and thus it suffices to prove the result when 
$f$ contains a single angular harmonic.

Taking \(\Phi = \Phi_n = \begin{bmatrix}\widehat f^r_n \\\widehat f^\theta_n \end{bmatrix}\), we obtain the equation
\[
\Phi'' + \left(\frac r2 + \frac1r\right)\Phi' + \left(1 - \lambda - \frac{1 + n^2}{r^2} + \frac{2in}{r^2}\begin{bmatrix}0&-1\\1&0\end{bmatrix} - i\alpha n V + \begin{bmatrix}0& 0 \\\alpha r \partial_r V  &0\end{bmatrix}\right)\Phi = 0.
\]
Changing variables by taking \(\rho = \frac14 r^2\) we obtain the ODE
\[
\Phi'' + \left(1 + \frac1\rho\right)\Phi' + \left(\frac{1 - \lambda}{\rho} - A(\rho)\right)\Phi = 0,
\]
where the coefficient
\[
A(\rho) = \frac{1 + n^2}{4\rho^2} - \frac{in}{2\rho^2}\begin{bmatrix}0&-1\\1&0\end{bmatrix} + \frac{i\alpha n}{8\pi\rho^2}\left(1 - e^{-\rho}\right) - \alpha
\begin{bmatrix}0& 0 \\  \frac{-1 + e^{-\rho}  + \rho e^{-\rho}}{4\pi \rho^2}  &0\end{bmatrix}
\]
is such that $|A(\rho)| \lesssim \frac{1}{\rho^2}$.

Taking \(\Psi = \begin{bmatrix}\Phi\\\Phi'\end{bmatrix}\) we obtain the first order system
\[
\Psi' = \left(\begin{bmatrix}0&I\\0&-I\end{bmatrix} + \begin{bmatrix}0&0\\\frac{\lambda-1}\rho I&-\frac1\rho I \end{bmatrix} + \begin{bmatrix}0&0\\A(\rho)&0\end{bmatrix}\right)\Psi,
\]
The sum of the two first matrices above has the eigenvalues $\mu_+ = \frac{\lambda-1}{\rho} + O \left( \frac{1}{\rho^2} \right)$ and $\mu_- = -1 - \frac{\lambda}{\rho} + O \left( \frac{1}{\rho^2} \right)$, with eigenvectors $\begin{bmatrix}1 & 0 & \mu_{\pm} & 0\end{bmatrix}^T$ and $\begin{bmatrix}0 & 1 & 0 & \mu_{\pm}\end{bmatrix}^T$.
Therefore, it is possible to write
$$
\begin{bmatrix}0&I\\0&-I\end{bmatrix} + \begin{bmatrix}0&0\\\frac{\lambda-1}\rho I&-\frac1\rho I \end{bmatrix} = S \operatorname{diag}(\mu_+,\mu_+,\mu_-,\mu_-) S^{-1},
$$
where the matrix $S = S(\rho)$ has a nonsingular limit and satisfies $|S'(\rho)| \lesssim \frac{1}{\rho^2}$. Thus, we may apply (the proof of)~\cite[Theorem~III.8.1]{MR0069338} to obtain linearly independent solutions satisfying
\begin{alignat}{2}
&\Psi \sim \begin{bmatrix}1\\0\\0\\0\end{bmatrix} \rho^{\lambda - 1}, &&\qquad\Psi \sim  \begin{bmatrix}0\\1\\0\\0\end{bmatrix} \rho^{\lambda - 1}, \\
&\Psi\sim \begin{bmatrix}1\\0\\-1\\0\end{bmatrix} \rho^{-\lambda}e^{-\rho},&&\qquad\Psi\sim \begin{bmatrix}0\\1\\0\\-1\end{bmatrix} \rho^{-\lambda}e^{-\rho},
\end{alignat}
as \(\rho\rightarrow+\infty\). In particular, as \(r\rightarrow+\infty\) eigenfunctions must satisfy
\[
\widehat f \sim r^{2(\lambda - 1)} \text{ or }\widehat f\sim r^{-2\lambda}e^{-\frac14 r^2},
\]
and we readily see that \( r^{2(\lambda - 1)}\not\in L^2_\xi(m)\) if \(\Re\lambda>\frac{1 - m}2\) and \(r^{2(\lambda - 1)}\not\in L^1_\xi\) if \(\Re\lambda>0\).
\end{proof}

For the operator \(\cL - \alpha\Lambda\) a similar argument yields the following:

\begin{lem}[{\cite[Lemma~4.5]{MR2123378}}]\label{lem:EV-Lambda-Inf}
If \(\alpha\in \R\), \(m>1\) and \(f\in \cD(m)\) is an eigenfunction of \(\cL - \alpha\Lambda\) with eigenvalue \(\lambda\) satisfying \(\Re\lambda>\frac{1 - m}2\) and $\lambda \notin \sigma_{\operatorname{ess}} (\cL - \alpha\Lambda)$, then \(\xi^k \nabla_\xi^\ell f\in L^2_\xi(\infty)\) for any integers $k,\ell$.
\end{lem}

For the operator \(\cL - \alpha\Xi\) we have the following result, which follows from~\cite[Proposition~4.3]{MR2178064}:

\begin{lem}[{\cite[Proposition~4.3]{MR2178064}}]\label{lem:EV-ag-Inf}
If \(\alpha\in \R\) and \(f\in \cD_1\) is an eigenfunction of \(\cL - \alpha \Xi\) with eigenvalue \(\lambda\) satisfying \(\Re\lambda>0\) and \(\lambda\not\in \sigma_{\operatorname{ess}}(\cL - \alpha\Xi)\), then \(\xi^k \nabla_\xi^\ell (\nabla_\xi\cdot f)\in L^2_\xi(\infty)\) for all $k$, $\ell$.
\end{lem}

\begin{proof}
If \(f\in L^1_\xi\) satisfies
\[
(\cL - \alpha\Xi)f = \lambda f 
\]
then we claim that it belongs to the Schwartz class. Indeed, write
\[ 
f = -\alpha \int_0^\infty e^{-t \lambda} e^{t\cL} \Xi f dt. 
\]
Note $\Xi f = \Div_\xi (f \otimes g) - f \cdot \grad_\xi g$, hence by the decay of $g$, we have
\[
\norm{f}_{L^{1}(1)} \lesssim \int_0^\infty e^{-t \Re \lambda} \frac{1}{a(t)^{\frac{1}{2}}} dt \norm{f}_{L^1} \lesssim \frac{1}{\abs{\Re \lambda}}\norm{f}_{L^1}.
\]
By bootstrap it follows that $f \in L^1(m)$ for all $m \geq 0$. From there a bootstrap argument (going say, $1/2$ a derivative at a time) gives that $\grad_\xi^k f \in L^1(m)$ for all $k \geq 0$ and all $m \geq 0$; this implies that $f$ belongs to the Schwartz class.  
Hence, taking the divergence, 
\[
(\cL - \alpha g\cdot\nabla_\xi)(\nabla_\xi\cdot f) = (\lambda - \tfrac12)(\nabla_\xi\cdot f).
\]
In particular, \(\nabla_\xi\cdot f\in L^1_\xi\) is an eigenfunction of \(\cL - \alpha g\cdot\nabla_\xi\) with corresponding eigenvalue \(\lambda - \frac12\). From~\cite[Proposition~4.3]{MR2178064} we then see that \(\nabla_\xi\cdot f\in L^2_\xi(\infty)\).
\end{proof}

Finally, a similar argument yields the following:
\begin{lem}\label{lem:EV-Pi-Inf}
If \(\alpha\in \R\) and \(f\in\cD_1\) is an eigenfunction of \(\cL - \alpha\Pi\) with eigenvalue \(\lambda \notin \sigma_{\operatorname{ess}} (\mathcal{L} - \alpha \Pi)\) satisfying \(\Re\lambda>0\) then \(\xi^k \nabla_\xi^\ell \Div_\xi f\in L^2_\xi(\infty)\) for any integers $k,\ell$.

\end{lem}

\begin{proof}
If \(f\in \cD_1\) satisfies
\[
(\cL - \alpha\Pi)f = \lambda f,
\]
then its divergence \(F = \Div_\xi f\) satisfies
\[
(\cL - \alpha \Gamma)F + \alpha(\nabla_\xi\cdot F) g  = (\lambda - \frac12)F,
\]
and following a similar argument to Lemma~\ref{lem:EV-ag-Inf} we see that \(\nabla_\xi^k F\in L^2_\xi(m)\) for all \(k,m\geq 0\).

Following the argument of Lemma~\ref{lem:EV-G-Inf}, we switch to polar coordinates and expand in angular harmonics to obtain and equation for \(\Phi = \Phi_n = \begin{bmatrix}\widehat F^r_n \\\widehat F^\theta_n \end{bmatrix}\). Using that 
$$
\nabla \cdot (F^r e_r + F^\theta e_\theta) = \frac{1}{r} \partial_r (rF^r) + \frac{1}{r} \partial_\theta F^\theta,
$$
we obtain the equation
\[
\Phi'' + \left(\frac r2 + \frac1r + \begin{bmatrix}0&0 \\ \alpha rV&0\end{bmatrix}\right)\Phi' + \left(\frac32 - \lambda - \frac{1 + n^2}{r^2} + \frac{2in}{r^2}\begin{bmatrix}0&-1\\1&0\end{bmatrix} + \begin{bmatrix} -i\alpha nV &0\\\alpha \partial_r(rV)& 0 \end{bmatrix}\right)\Phi = 0.
\]

Switching to the variable $\rho = \frac{r^2}{4}$ and letting $\Psi(\rho) = \begin{bmatrix} \Phi \\ \Phi' \end{bmatrix}$, we find the equation
$$
\Psi' + \left( M(\rho) + B(\rho) \right) \Psi = 0, \qquad \mbox{where} \qquad M(\rho) = \begin{bmatrix} 0 & I \\ 0 & -I \end{bmatrix} + \frac{1}{\rho} \begin{bmatrix} 0 & 0 \\ (\lambda - \frac{3}{2}) I & \begin{bmatrix} -1 & 0 \\ \alpha/4\pi & -1 \end{bmatrix} \end{bmatrix} 
$$
and $|B(\rho)| \lesssim \frac{1}{\rho^2}$. Theorem III.8.1 in~\cite{MR0069338} does not quite apply, since the matrix $M(\rho)$ is not diagonalizable. However, it can be put in Jordan normal form: denoting its eigenvalues
$$
\mu_{\pm}(\rho) = \frac{1}{2}\left( - 1 - \frac{1}{\rho} \pm \sqrt{\left( 1+\frac{1}{\rho} \right)^2 + \frac{4\lambda - 6}{\rho}} \right),
$$
and switching to the basis
$$
e_1 = \begin{bmatrix} 0 \\ 1 \\ 0 \\ \mu_+ \end{bmatrix}, \qquad e_2 = \begin{bmatrix} 1 \\ 0 \\ \mu_+ \\ \frac{\alpha \mu_+}{4\pi(1 + \rho + 2\mu \rho)} \end{bmatrix}, \qquad e_3 = \begin{bmatrix} 0 \\ 1 \\ 0 \\ \mu_- \end{bmatrix}, \qquad e_4 = \begin{bmatrix} 1 \\ 0 \\ \mu_- \\ \frac{\alpha \mu_+}{4\pi(1 + \rho + 2\mu_- \rho)} \end{bmatrix}
$$
it becomes
$$
\begin{bmatrix} \mu_+ & \frac{\alpha \mu_+}{4\pi(1+\rho + 2\mu_+ \rho)} & 0 & 0 \\ 0& \mu_+ & 0 & 0 \\ 0 & 0 & \mu_- & \frac{\alpha \mu_-}{4\pi(1+\rho + 2\mu_- \rho)} \\ 0 & 0 & 0 & \mu_- \end{bmatrix}.
$$
Noting that $\mu_+(\rho) = \frac{\lambda - \frac{3}{2}}{\rho}+ O \left( \frac{1}{\rho^2} \right)$ and $\mu_-(\rho) = - 1 + \frac{\frac{1}{2}-\lambda}{\rho} + O \left( \frac{1}{\rho^2} \right)$, the Jordan blocks above, up to errors of order $O \left( \frac{1}{\rho^2} \right)$, can be written
$$
\begin{bmatrix} \frac{\lambda - \frac{3}{2}}{\rho} & \frac{\alpha (\lambda - \frac{3}{2})}{4\pi \rho} \\ 0 & \frac{\lambda - \frac{3}{2}}{\rho}\end{bmatrix} \qquad \mbox{and} \qquad \begin{bmatrix} - 1 + \frac{\frac{1}{2} - \lambda}{\rho} & \frac{\alpha}{4\pi \rho} \\ 0 & - 1 + \frac{\frac{1}{2} - \lambda}{\rho} \end{bmatrix}.
$$
Solving $f' = P f$ for each of these two matrices, one obtains solutions with the asymptotic behavior $\rho^{\lambda - \frac{3}{2}} (A+B \log \rho)$ and $\rho^{\frac{1}{2} - \lambda} e^{-\rho} (C+D \log \rho)$ (for constants $C$ and $D$), respectively. Since the matrix $S = \begin{bmatrix} e_1 & e_2 & e_3 & e_4 \end{bmatrix}$ has a nonsingular limit, and satisfies $S'(\rho) \lesssim \frac{1}{\rho^2}$, this conclusion can be transferred to the full system. In $r$ coordinates, this means that a basis of solutions has the asymptotics
$$
(e_1,e_2) r^{2\lambda -3}(A+B \log r) \quad \mbox{and} \quad (e_3,e_4) r^{1-2\lambda} e^{-\frac{r^2}4}(C+D \log r)
$$

However, as \(F = \Div_\xi f\) for \(f\in L^1_\xi\) we see that a behavior $\sim  r^{2\lambda -3}(A+B \log r)$ is excluded whenever \(\Re\lambda>0\).  Indeed, for \(R>0\) take \(\chi_R\in C^\infty_c(\R^2)\) to be a bump function supported on \(\{|\xi|\leq 2R\}\) and identically \(1\) on \(\{|\xi|\leq 1\}\). Then,
\begin{align*}
\abs{\int \brak{\xi} \chi_R(\xi) F d\xi} = \abs{\int \brak{\xi} \chi_R(\xi) \Div_\xi f d\xi} \lesssim \int \abs{f} d\xi,  
\end{align*}
hence $F \sim  r^{2\lambda -3}(A+B \log r)$ is excluded whenever \(\Re\lambda>0\) as the left hand side would diverge as \(R\to \infty\), leading to a contradiction. Therefore, \(\xi^k \nabla_\xi^\ell F\in L^2_\xi(\infty)\) as required.
\end{proof}

\subsection{Upper bounds on the eigenvalues}
As the eigenfunctions of the operators \(\cL - \alpha \Gamma\), \(\cL - \alpha\Lambda\), \(\cL - \alpha\Xi\), \(\cL - \alpha \Pi\) have Gaussian decay whenever \(\Re\lambda\) is sufficiently large, we may now use the fact that these operators exhibit certain symmetries with respect to the natural inner product on \(L^2_\xi(\infty)\) to obtain bounds on the eigenvalues.

We first consider the operator \(\cL - \alpha\Gamma\):
\begin{lem}[{\cite[Proposition~3.5]{MR2770021}}]\label{lem:GammaEV}
If \(\alpha\neq 0\), then there exists \(0<\mu(\alpha) \leq\frac12\) so that: if \(f\) such that $\xi^k \nabla_\xi^\ell f \in L^2_\xi(\infty)$ for all $k,\ell$ is an eigenfunction of \(\cL - \alpha\Gamma\) with eigenvalue \(\lambda\), then \(\Re\lambda\leq-\mu\).
\end{lem}
\bpf
Let $f$ satisfy the hypotheses of the theorem, in particular
\[
(\cL - \alpha\Gamma)f = \lambda f.
\]
We then compute
\begin{align}
\cL(\xi\cdot f) - \frac12(\xi\cdot f) - 2\nabla_\xi\cdot f - \alpha g\cdot\nabla_\xi(\xi\cdot f) &= \lambda\xi\cdot f,\label{DotEV}\\
\cL(\nabla_\xi\cdot f) + \frac12(\nabla_\xi\cdot f) - \alpha g\cdot\nabla_\xi(\nabla_\xi\cdot f) &= \lambda \nabla_\xi\cdot f.\label{DivEV}
\end{align}
Integrating by parts then yields the identities,
\begin{align}
\Re\lambda\|f\|_{L^2_\xi(\infty)}^2 &= \<\cL f,f\>_{L^2_\xi(\infty)} + \alpha \Re\<p\ \xi\cdot f,\xi^\perp \cdot f\>_{L^2_\xi(\infty)},\label{ID1}\\
\Re\lambda \|\xi\cdot f\|_{L^2_\xi(\infty)}^2 &= \<\cL(\xi\cdot f),\xi\cdot f\>_{L^2_\xi(\infty)} - \frac12\|\xi\cdot f\|_{L^2_\xi(\infty)}^2 - 2\Re \<\nabla_\xi\cdot f,\xi\cdot f\>_{L^2_\xi(\infty)},\label{ID2}\\
\Re\lambda \|\nabla_\xi\cdot f\|_{L^2_\xi(\infty)}^2 &= \<\cL(\nabla_\xi\cdot f),\nabla_\xi\cdot f\>_{L^2_\xi(\infty)} + \frac12\|\nabla_\xi\cdot f\|_{L^2_\xi(\infty)},\label{ID3}
\end{align}
where the smooth function \(p(\xi) = -\frac1{\pi |\xi|^4}(1 - e^{-\frac14|\xi|^2}) + \frac1{4\pi|\xi|^2}e^{-\frac14|\xi|^2}\) so that (c.f.~\eqref{p-def})
\[
\nabla_\xi g = p(\xi)\left(\xi^\perp\otimes \xi\right) + \frac1{2\pi|\xi|^2}\left(1 - e^{-\frac14|\xi|^2}\right)\begin{bmatrix}0&-1\\1&0\end{bmatrix}.
\]

\smallskip\noindent\underline{\emph{Case 1:} \(\nabla_\xi\cdot f\neq 0\).} Here we observe that the projection \(P_0(\nabla_\xi\cdot f) = 0\) so we may apply the identity \eqref{ID3} to conclude that \(\Re\lambda\leq 0\). If \(\Re\lambda = 0\) then from the properties of \(\cL\) we have \(\nabla_\xi\cdot f\in \Sp\{\xi_1 G,\xi_2 G\}\). However, we may explicitly verify that if \(\nabla_\xi\cdot f\) satisfies \eqref{DivEV} with \(\lambda = 0\) then \(\nabla_\xi\cdot f\not\in \Sp\{\xi_1 G,\xi_2 G\}\). As a consequence, we must have \(\Re\lambda < 0\).

In order to prove the existence of \(\mu\), it remains to rule out the possibility that there exists a sequence of eigenvalues \(\{\lambda_j\}\) so that \(\Re\lambda_j\nearrow0\) and \(|\Im \lambda_j|\rightarrow\infty\). If \(\Re\lambda\geq - \frac12\) then integrating by parts in the second term in \eqref{ID3} we have
\[
\|\nabla_\xi(\nabla_\xi\cdot f)\|_{L^2_\xi(\infty)}^2 = (\tfrac32 - \Re\lambda)\|\nabla_\xi\cdot f\|_{L^2_\xi(\infty)}^2 \leq 2\|\nabla_\xi\cdot f\|_{L^2_\xi(\infty)}^2.
\]
Further, a similar computation to \eqref{ID3} yields
\begin{align*}
|\Im \lambda|\ \|\nabla_\xi\cdot f\|_{L^2_\xi(\infty)}^2 &= |\alpha|\ |\Im\<g\cdot\nabla_\xi(\nabla_\xi\cdot f),\nabla_\xi\cdot f\>_{L^2_\xi(\infty)}|\\\
&\leq |\alpha|\ \|g\|_{L^\infty_\xi}\|\nabla_\xi(\nabla_\xi\cdot f)\|_{L^2_\xi(\infty)}\|\nabla_\xi\cdot f\|_{L^2_\xi(\infty)}\\
&\lesssim |\alpha|\ \|\nabla_\xi\cdot f\|_{L^2_\xi(\infty)}^2.
\end{align*}
As the (discrete) spectrum of \(\cL - \alpha\Gamma\) on \(L^2_\xi(\infty)\) consists only of isolated points, it is then clear that there exists some \(0<\mu\leq \frac12\) so that
\[
\Re\lambda\leq-\mu.
\]

\smallskip\noindent\underline{\emph{Case 2:} \(\nabla_\xi\cdot f = 0\), \(\xi\cdot f\neq 0\).} Here we may use the identity \eqref{ID2} with the fact that
\[
\<\cL (\xi\cdot f),(\xi\cdot f)\>_{L^2_\xi(\infty)}\leq 0,
\]
to show that \(\Re\lambda\leq-\frac12\).

\smallskip\noindent\underline{\emph{Case 3:} \(\nabla_\xi\cdot f = 0 = \xi\cdot f\).} If \(\lambda = 0\) then from the identity \eqref{ID1} and the properties of \(\cL\) we have \(f \in \Sp\{G\}\), which is a contradiction as \(\nabla_\xi\cdot f = 0\). Otherwise we may integrate the equation \((\cL - \alpha \Gamma)f = \lambda f\) to show that \(P_0f = 0\) so from the properties of \(\cL\),
\[
\<\cL f, f\>_{L^2_\xi(\infty)}\leq -\tfrac12\|f\|_{L^2_\xi(\infty)}^2.
\]
We may then apply the identity \eqref{ID1} to again conclude that \(\Re\lambda\leq-\frac12\).
\epf

Next, we consider the operator \(\cL - \alpha\Lambda\):
\begin{lem}[{\cite[Proposition~4.1]{MR2123378}}]\label{lem:UghMod}
If \(f\) such that $\xi^k \nabla^\ell_\xi f \in L^2_\xi(\infty)$ for all $k,\ell$ is an eigenfunction of \(\cL - \alpha\Lambda\) with eigenvalue \(\lambda\), then \(\Re\lambda\leq 0\). Further, if \(\int_{\R^2} f\,d\xi = 0\) then \(\Re\lambda\leq -\frac12\).
\end{lem}
\bpf
A short computation shows that \(\Lambda\) is skew-adjoint on \(L^2_\xi(\infty)\). As a consequence,
\[
\Re\lambda\|f\|_{L^2_\xi(\infty)}^2 = \Re\<\cL f,f\>.
\]
The result then follows from Proposition~\ref{prop:FPProps}.
\epf

Next, we consider the operator \(\cL - \alpha\Xi\), and have the following:
\begin{lem}\label{lem:EV-Xi}
If \(f\) such that $\xi^k \nabla^\ell_\xi f \in L^2_\xi(\infty)$ for all $k,\ell$ is an eigenfunction of \(\cL - \alpha\Xi\) with eigenvalue \(\lambda\), then \(\Re\lambda\leq 0\).
\end{lem}
\begin{proof}
We recall that for \(F = \nabla_\xi\cdot f\) we have
\[
(\cL - \alpha g\cdot\nabla_\xi)F = (\lambda - \tfrac12)F.
\]
As \(P_0F = 0\) we may then use Proposition~\ref{prop:FPProps} to obtain
\[
(\Re\lambda - \tfrac12)\|F\|_{L^2_\xi(\infty)}^2 = \<\cL F,F\>_{L^2_\xi(\infty)}\leq -\tfrac12 \norm{F}_{L^2_\xi(\infty)}^2
\]
and hence \(\Re\lambda\leq 0\).
\end{proof}

Finally, we consider the operator \(\cL - \alpha\Pi\):
\begin{lem}\label{lem:EV-Pi}
If $f$ such that $\xi^k \nabla^\ell_\xi f \in L^2_\xi(\infty)$ for all $k,\ell$ is an eigenfunction of \(\cL - \alpha\Pi\) with eigenvalue \(\lambda\),  then \(\Re\lambda\leq 0\).
\end{lem}
\begin{proof}
We recall that \(F = \Div_\xi f\) satisfies the equation
\[
(\cL - \alpha\Gamma)F + \alpha (\nabla_\xi\cdot F)g = (\lambda - \tfrac12) F.
\]
Arguing as in Lemma~\ref{lem:GammaEV} we see that \(\xi\cdot F,\nabla_\xi\cdot F\in L^2_\xi(\infty)\) satisfy the equations
\begin{gather*}
\cL(\xi\cdot F) - \tfrac12(\xi\cdot F) - 2\nabla_\xi\cdot F - \alpha g\cdot\nabla_\xi(\xi\cdot F)  = (\lambda - \tfrac12)(\xi\cdot F),\\
\cL(\nabla_\xi\cdot F) + \tfrac12(\nabla_\xi\cdot F) = (\lambda - \tfrac12)(\nabla_\xi\cdot F)
\end{gather*}
Integrating by parts then yields the analogues of the identities \eqref{ID1}--\eqref{ID3}:
\begin{align}
\Re\lambda\|F\|_{L^2_\xi(\infty)} &= \<\cL F,F\>_{L^2_\xi(\infty)} + \tfrac12\|F\|_{L^2_\xi(\infty)}^2 + \alpha\Re\<p \xi\cdot F + q\nabla_\xi\cdot F,\xi^\perp\cdot F\>_{L^2_\xi(\infty)},\label{ID1*}\\
\Re\lambda\|\xi\cdot F\|_{L^2_\xi(\infty)}^2 &= \<\cL(\xi\cdot F),\xi\cdot F\>_{L^2_\xi(\infty)} - 2\Re\<\nabla_\xi \cdot F,\xi\cdot F\>_{L^2_\xi(\infty)},\label{ID2*}\\
\Re\lambda\|\nabla_\xi\cdot F\|_{L^2_\xi(\infty)}^2 &= \<\cL(\nabla_\xi\cdot F),\nabla_\xi\cdot F\>_{L^2_\xi(\infty)} + \|\nabla_\xi\cdot F\|_{L^2_\xi(\infty)}^2\label{ID3*},
\end{align}
where the functions \(
p = - \frac1{\pi|\xi|^4}(1 - e^{-\frac14|\xi|^2}) + \frac1{4\pi|\xi|^2}e^{-\frac14|\xi|^2}\), \(q = \frac1{2\pi|\xi|^2}(1 - e^{-\frac14|\xi|^2})\).

\smallskip\noindent\underline{\emph{Case 1:} \(\nabla_\xi\cdot F\neq 0\).} We observe that \(P_0(\nabla_\xi\cdot F) = 0\) and as \(F = \Div_\xi f\) we also have \(P_1(\nabla_\xi\cdot F) = 0\). From Proposition~\ref{prop:FPProps} we then have,
\[
\<\cL(\nabla_\xi\cdot F),\nabla_\xi\cdot F\>_{L^2_\xi(\infty)} \leq -\|F\|_{L^2_\xi(\infty)}^2,
\]
so from the identity \eqref{ID3*} we obtain \(\Re\lambda\leq 0\).

\smallskip\noindent\underline{\emph{Case 2:} \(\nabla_\xi\cdot F = 0\), \(\xi\cdot F\neq 0\).} Here we simply apply the identity \eqref{ID2*} to obtain \(\Re\lambda\leq 0\).

\smallskip\noindent\underline{\emph{Case 3:} \(\nabla_\xi\cdot F = 0 = \xi\cdot F\).} Here we apply the identity \eqref{ID1*} and use the fact that \(P_0F = 0\) to obtain
\[
\<\cL F,F\>_{L^2_\xi(\infty)} + \tfrac12\|F\|_{L^2_\xi(\infty)}^2\leq 0.
\]
\end{proof}

\subsection{Essential spectrum and growth bound}

Given a bounded operator $T$ on a Banach space $X$, its essential spectrum $\sigma_{ess}(T)$ is the set of $\lambda \in \mathbb{C}$ such that $T - \lambda$ is not Fredholm. The essential spectral radius is then given by
$$
r_{ess}(T,X) = \sup \{ |\lambda|,\; \lambda \in \sigma_{ess}(T) \}.
$$

\begin{lem} With the previous definition,
\begin{align*}
&r_{ess}(e^{\tau(\mathcal{L} - \alpha \Gamma)},L^2_\xi(m)) = r_{ess}(e^{\tau(\mathcal{L} - \alpha \Lambda)},L^2_\xi(m)) = e^{\frac{1 - m}2\tau},\\
&r_{ess}(e^{\tau(\cL - \alpha \Gamma)},L^1_\xi) = r_{ess}(e^{\tau(\cL - \alpha \Xi)},L^1_\xi) = r_{ess}(e^{\tau(\cL - \alpha \Pi)},L^1_\xi) = 1.
\end{align*}
\end{lem}
\begin{proof}
The essential spectrum is stable by compact perturbations. Relying on~\cite[Theorem~A.1]{MR1912106}, \cite[Proposition~A.2]{MR1912106}, the result follows from the fact that $e^{\tau(\mathcal{L} - \alpha \Gamma)}$ and $e^{\tau(\mathcal{L} - \alpha \Lambda)}$ are compact perturbations of $e^{\tau \mathcal{L}}$ on \(L^2_\xi(m)\), and similarly that $e^{\tau(\mathcal{L} - \alpha \Gamma)}$, $e^{\tau(\mathcal{L} - \alpha \Xi)}$, $e^{\tau(\mathcal{L} - \alpha \Pi)}$ are compact perturbations of \(e^{\tau\cL}\) on \(L^1_\xi\).
\end{proof}

Using the previous estimates we may now complete the proof of Propositions~\ref{prop:SemigroupL2m},~\ref{prop:SemigroupL1}:

\bpf[Proof of Proposition~\ref{prop:SemigroupL2m}]

By~\cite[Definition~IV.2.10]{MR1721989}, the essential growth bound of the semigroups $e^{\tau(\mathcal{L} - \alpha \Gamma)}$ and $e^{\tau(\mathcal{L} - \alpha \Lambda)}$ on \(L^2_\xi(m)\) is $\omega_{ess} = \frac{1 - m}{2}$. By~\cite[Corollary~IV.2.11]{MR1721989}, we learn that $\sigma(\mathcal{L} - \alpha \Gamma) \cap \{\lambda\; | \; \operatorname{Re} \lambda>\frac{1 - m}{2}\}$ and $\sigma(\mathcal{L} - \alpha \Lambda) \cap \{\lambda\; | \; \operatorname{Re} \lambda>\frac{1 - m}{2}\}$ only consist of eigenvalues.

Lemmas~\ref{lem:EV-G-Inf},~\ref{lem:GammaEV} show that all eigenvalues of \(\cL - \alpha\Gamma\) on $L^2_\xi(m)$, with $m>1$, satisfy \(\Re\lambda\leq 0\); with the improvement if $m>1+2\mu$ that \(\Re\lambda\leq -\mu\), for some \(\mu = \mu(\alpha)\in (0,\frac12]\) whenever \(\alpha\neq0\). Similarly, Lemmas~\ref{lem:EV-Lambda-Inf},~\ref{lem:UghMod} show that all eigenvalues of \(\cL - \alpha\Lambda\) on $L^2_\xi(m)$, with $m>1$,  satisfy \(\Re\lambda\leq 0\); with the improvement \(\Re\lambda\leq -\mu\) on $L^2_\xi(1+2\mu)$, with $0<\mu<\frac{1}{2}$, whenever $\int f^z\,d\xi = 0$.

Applying~\cite[Corollary IV.2.11]{MR1721989} once again gives the desired growth bounds of both semigroups, since for \(m>1 + 2\mu\),  we have \(e^{- \mu\tau} > e^{\tau\frac{1 - m}2}\).

\epf

\bpf[Proof of Proposition~\ref{prop:SemigroupL1}]
Applying an identical argument to the proof of Proposition~\ref{prop:SemigroupL2m} we see that the essential growth bound of the semigroups \(e^{\tau(\cL - \alpha \Gamma)}\), \(e^{\tau(\cL - \alpha \Xi)}\), \(e^{\tau(\cL - \alpha\Pi)}\) on \(L^1_\xi\) is \(\omega_{ess} = 0\). Further, applying Lemmas~\ref{lem:EV-G-Inf},~\ref{lem:GammaEV} for \(\cL - \alpha\Gamma\), Lemmas~\ref{lem:EV-ag-Inf},~\ref{lem:EV-Xi} for \(\cL - \alpha\Xi\) and Lemmas~\ref{lem:EV-Pi-Inf},~\ref{lem:EV-Pi} for \(\cL - \alpha \Pi\) we see that all eigenvalues of the corresponding operators satisfy \(\Re\lambda\leq 0\). The corresponding growth bounds for the semigroups then follow from~\cite[Corollary IV.2.11]{MR1721989}.
\epf

\end{appendix}
\bibliography{refs}
\end{document}